\documentclass{tac}

\usepackage{xy}

\usepackage{amsfonts}
\usepackage{amssymb}
 \usepackage{amsmath}
\usepackage{hyperref}
\input diagxy
\thirdleveltheorems

\author {Robert Par\'e}

\address{Department of Mathematics and Statistics, Dalhousie University\\
 Halifax, NS, Canada, B3H 4R2\\}

\title{Taut functors and the difference operator}

\copyrightyear{2024}


\eaddress{pare@mathstat.dal.ca}

\mathrmdef{id}
\mathrmdef{Id}

\mathrmdef{Ob}

\newcommand{\ov}[1]{{\bar{#1}}}

\newcommand{\todd}[2]{\xymatrix@1{\ar[r]|\bb^{#1}_{#2}&}}

\newtheorem{theorem}{Theorem}[section]
\newtheorem{proposition}{Proposition}[section]
\newtheorem{lemma}{Lemma}[section]
\newtheorem{corollary}{Corollary}[section]

\newtheoremrm{definition}{Definition}[section]
\newtheoremrm{remark}{Remark}[section]
\newtheoremrm{example}{Example}[section]
\newtheoremrm{examples}{Examples}[section]

\input xy
\xyoption{all}

\mathrmdef{Hom}
\mathbfdef{Set}

\def\liml{\varprojlim}
\def\limr{\varinjlim}

\def\Cat{{\cal C}{\it at}}
\def\Taut{{\cal T\!\!}{\it aut}}

\newbox\bbox
\setbox\bbox = \hbox to 0pt{\hss $\scriptstyle\bullet$\hss}
\def\bb{\usebox{\bbox}}

\begin{document}

\maketitle

% !TEX root = taut.tex
\begin{abstract}

We establish a calculus of differences for taut endofunctors of the
category of sets, analogous to the classical calculus of finite differences
for real valued functions. We study how the difference operator
interacts with limits and colimits as categorical versions, of the usual
product and sum rules. The first main result is a lax chain rule which
has no counterpart for mere functions.
We also show that many important classes of functors (polynomials,
analytic functors, reduced powers, ...) are taut, and calculate
explicit formulas for their differences. Covariant Dirichlet series
are introduced and studied. The second main result is a Newton
summation formula expressed as an adjoint to the difference
operator.

\end{abstract}

\setcounter{tocdepth}{2}
\tableofcontents
% !TEX root = taut.tex
\section*{Introduction}

The category of endofunctors of $ {\bf Set} $, the category of
sets, contains many interesting subcategories. Consider, for example,
the recent work on polynomial functors, spearheaded by Spivak
(see \cite{NiuSpi23} and the references there), building on
previous work of Kock chronicled in his arXiv paper \cite{Koc09}
(see also \cite{GamKoc13} which deals with polynomial monads).
Before that there were the analytic functors of
Joyal \cite{Joy81} stemming from his categorical treatment of
combinatorics, and developed extensively since then
(see \cite{GamJoy17} and its extensive bibliography).
But the study of endofunctors of $ {\bf Set} $
goes back to the early days of category theory.
Questions of rank for monads arose, showing the necessity of
looking more deeply into the structure of these endofunctors.

This suggests that the study of $ {\bf Set} $ endofunctors
goes to the very foundation of set theory. For example, a
non-principal ultrafilter, whose existence is well-known to
require some form of the axiom of choice, produces an
interesting endofunctor, the ultra-power functor. It is
left exact and preserves finite coproducts, and so it is
isomorphic to the identity on finite sets but not infinite ones.
In \cite{Trn71B} Trnkov\'a, and independently Blass \cite{Bla77A, Bla77B}
showed that the existence of a
non-trivial {\em exact} endofunctor of $ {\bf Set} $ is
equivalent to the existence of a measurable cardinal. Also
see \cite{Rei71} in this connection.

These considerations, among others, indicate that a
systematic study of the structure of endofunctors of
$ {\bf Set} $ might be desirable. And indeed, such a
study was initiated in \cite{Trn71A, Trn71B} by
Trnkov\'a, where she made an exhaustive study of
their preservation properties.

The present work revisits this project from a different
perspective taking into account the many developments
of the last century.  We develop a difference calculus
 for a rather large class of functors which parallels the
 classical finite difference calculus for real valued functions.
 The class of functors we consider are the taut functors
 introduced by Manes \cite{Man02}, who was motivated
 by theoretical computer science considerations. (Taut
 functors were also considered in \cite{Trn71A} under
 the name ``preimage preserving functors''.)

Section~\ref{Sec-Taut} begins by recalling Manes' definitions
of taut functor and taut natural transformation and
recording some of their basic properties to be referred to later.
Then follows a detailed study of the stability of tautness
under limits and colimits. The main result here is a
characterization of those colimits that preserve tautness,
Theorem~\ref{Thm-Confl}, which to our knowledge, has
never appeared in print.

In Section~\ref{Sec-SpecCl} we consider some known
classes of endofunctors: polynomials, divided power
series, analytic functors, reduced powers, all of which
are taut. We also consider taut monads. We introduce
covariant Dirichlet functors, and what we call sequential
Dirichlet functors, which are taut too.

It is in Section~\ref{Sec-DiffOp} that we introduce the
difference operator. We study its functorial properties
and then how it interacts with limits and colimits. As a
special case we get a product (Leibniz) rule. Half of the
section is taken up by the proof of one of the main results of
the paper, the lax chain rule and its properties.

In Section~\ref{Sec-SpecClDiff} we return to the special classes
considered in Section~\ref{Sec-SpecCl} and obtain
explicit formulas for the differences.

We round out the paper with a version of the Newton
summation formula, a discrete version of Taylor series.
This is meant to recover nice taut functors, analytic say,
from their iterated differences at $ 0 $. It is given by
a left adjoint to the difference functor whose unit is
an isomorphism for analytic functors.

Except for a few comments in passing, we mostly
ignore questions of size. The ``category'' of endofunctors
is an illegitimate one. Its homs may be proper classes.
None of our results depends on this, but the purist will
have no difficulty legitimizing things using standard
techniques -- Grothendieck universes or G\"{o}del-Bernays
set theory (with classes), for example.

We would like to thank Andreas Blass, Aaron Fairbanks,
Theo Johnson-Freyd, Deni Salja and Peter Selinger for
helpful comments.

% !TEX root = taut.tex

\section{Taut functors}
\label{Sec-Taut}

The notions of taut functor and natural transformation were
introduced by Manes \cite{Man02}. These are precisely what
we need to develop our difference calculus of functors.

Everything that follows centres around pullbacks in which
one leg is a monomorphism. Following Manes, we call them
{\em inverse image diagrams}.

\subsection{Definitions and functorial properties}
\label{SSec-Defs}

\begin{definition}[Manes \cite{Man02}]
\label{Def-Taut}

A functor is {\em taut} if it preserves inverse image diagrams.
A natural transformation is {\em taut} if the naturality squares
corresponding to monomorphisms are pullbacks.

\end{definition}

We record some of the general properties of tautness.

\begin{proposition}
\label{Prop-Taut}

\begin{itemize}

	\item[(1)] Taut functors preserve monos.
	
	\item[(2)] The composite of taut functors is taut.
	
	\item[(3)] If $ t \colon F \to G $ is taut and $ H $ is taut, then
	so is $ H t \colon HF \to HG $.
	
	\item[(4)] If $ t \colon F \to G $ is taut and $ K $ preserves monos,
	then $ t K \colon FK \to GK $ is taut.
	
	\item[(5)] If $ t \colon F \to G $ and $ u \colon G \to L $ are taut,
	then the vertical composite $ u t \colon F \to L $ is taut.
	
	\item[(6)] If $ t \colon F \to G $ is taut and $ G $ is taut, then
	so is $ F $.

\end{itemize}

\end{proposition}

\begin{proof}
Perhaps the only part that is not completely straightforward is (6).
Let
$$
\bfig
\square/ >->`>`>` >->/[A`B`C`D;m`f`g`n]

\efig
$$
be an inverse image diagram. Then we have
$$
\bfig
\square/`>`>` >->/[FC`FD`GC`GD;
`tC`tD`Gn]

\place(250,250)[(2)]

\square(0,500)/>`>`>`>/[FA`FB`FC`FD;
Fm`Ff`Fg`Fn]

\place(250,750)[(1)]

\place(900,500)[=]

\square(1300,0)/`>`>` >->/[GA`GB`GC`GD;
`Gf`Gg`Gn]

\place(1550,250)[(4)]

\square(1300,500)/>`>`>` >->/[FA`FB`GA`GB;
Fm`tA`tB`Gm]

\place(1550,750)[(3)]

\efig
$$
where (2), (3), (4) are pullbacks, so (1) is too.
\end{proof}

\begin{corollary}
\label{Cor-Taut}

Categories with inverse images, taut functors and taut
natural transformations give a sub-$2$-category
$\Taut $ of $ \Cat $, the $2$-category of
categories.

\end{corollary}

\subsection{Limits of taut functors}
\label{SSec-LimTaut}

Limits of taut functors are again taut. This is just
a case of limits commuting with limits but some attention
must be paid as to where the diagrams and limits
are taken. If $ {\bf A} $ and $ {\bf B} $ are categories
with inverse images, we have the $ \Cat $ functor
category $ \Cat ({\bf A}, {\bf B}) $ of all functors from
$ {\bf A} $ to $ {\bf B} $ and all natural transformations,
and the subcategory $ \Taut ({\bf A}, {\bf B}) $
of all taut functors and taut natural transformations. And,
we also have the full image of $ \Taut ({\bf A},
{\bf B}) $ in $ \Cat ({\bf A}, {\bf B}) $, $ \Taut_{full}
({\bf A}, {\bf B}) $, of taut functors and all natural
transformations.

\begin{proposition}
\label{Prop-LimTaut}

Assume that $ {\bf B} $ has $ {\bf I} $-limits. Then $ \Taut_{full}
({\bf A}, {\bf B}) $ is closed under $ {\bf I} $-limits in
$ \Cat ({\bf A}, {\bf B}) $. If $ t \colon \Phi \to \Psi \colon
{\bf I} \to \Taut_{full} ({\bf A}, {\bf B}) $, and $ t I $ is taut for every
$ I $, then $ \liml_I t (I) \colon \liml_I \Phi (I)
\to \liml_I \Psi (I) $ is also taut.

\end{proposition}

\begin{proof}

We use the fact that limits commute with limits, applied to
the inverse image diagrams
$$
\bfig
\square/>->`>`>`>->/<600,500>[\Phi(I)(B_0)\ `\Phi(I)(B)
`\Phi(I)(A_0)\ `\Phi(I)(A);
```]

\place(300,250)[\framebox{\scriptsize Pb}]

\place(1000,250)[\mbox{and}]

\square(1600,0)/>->`>`>`>->/<600,500>[\Phi(I)(A_0)\ `\Phi(I)(A)
`\Psi (I)(A_0)\ `\Psi (I)(A);
`t(I)(A_0)`t(I)(A)`]

\place(1900,250)[\framebox{\scriptsize Pb}]

\efig
$$
respectively, for the inverse image diagram
$$
\bfig
\square/>->`>`>`>->/[B_0\ \ `B`A_0\ \ `A\rlap{\ \ .};
```]

\place(250,250)[\framebox{\scriptsize Pb}]

\efig
$$
\end{proof}

\begin{corollary}
\label{Cor-TautPB}

If $ {\bf B} $ has pullbacks, then the pullback of a
taut transformation along any natural transformation
is again taut.

\end{corollary}

\begin{proof}

Consider the pullback of taut functors
$$
\bfig
\square[H\times_G F`F`H`G;
p_2`p_1`t`u]

\place(250,250)[\framebox{\scriptsize Pb}]

\efig
$$
with $ t $ taut. Apply the previous proposition to the
morphism of diagrams
$$
\bfig
\node a(-100,300)[H]
\node b(450,300)[G]
\node c(450,0)[H]
\node d(850,0)[G]
\node e(450,600)[F]
\node f(850,350)[G]

\arrow|l|/>/[a`c;\id]
\arrow|b|/>/[a`b;u]
\arrow|r|/>/[b`d;\id]
\arrow|b|/>/[c`d;u]
\arrow|r|/>/[e`b;t]
\arrow|r|/>/[e`f;t]
\arrow|r|/>/[f`d;\id]

\efig
$$
to get that
$$
H \times_G F \to^{\id \times_\id t} H \times_G G
$$
is taut. But $ \id \times_\id t $ is $ p_1 $ followed by
an iso.
\end{proof}

We emphasize that we are not assuming that the
 transformations $ \Gamma (i) \colon \Gamma (I)
\to \Gamma (J) $ are taut, but even if we did we still would not
get limits in $ \Taut ({\bf A}, {\bf B}) $. The projections are not taut.
For example, if $ {\bf B} $ has finite products, then the
product of two taut functors is taut but the projection
$$
F \times G \to F
$$
is not. For a mono $ m \colon A_0\  \to/>->/ A $, the
naturality square
$$
\bfig
\square/>->`>`>`>->/<1000,500>[F(A_0) \times G(A_0)\ `F(A) \times G(A)
`F(A_0)\ `F(A);
Fm \times Gm`p_1`p_1`Fm]

\efig
$$
is not usually a pullback; the pullback is $ F(A_0)
\times G(A) $. 

A simpler example is that the unique map
$ F \to 1 $ is not taut. However this is the only obstruction
as we see below (\ref{Prop-ConnLim}).

\begin{remark}
\label{Rem-LimTaut}

We assumed that $ {\bf B} $ had all $ {\bf I} $-limits but
the proposition holds for any limits that exist in
$ \Cat ({\bf A}, {\bf B}) $ as long as they are pointwise,
i.e.~calculated in $ {\bf B} $.

\end{remark}

\begin{remark}
\label{Rem-DoubLim}

We can't help pointing out that what we are dealing
with are double limits, taken in the double category
$ {\mathbb T}{\rm aut} ({\bf A}, {\bf B}) $ whose objects are taut functors,
horizontal arrows arbitrary natural transformations,
vertical arrows taut transformations, and commutative
squares as cells. This may be worth pursuing but here
is not the place.

\end{remark}

\begin{proposition}
\label{Prop-TautCone}

Let $ {\bf I} $ be a non-empty category, $ \Phi \colon
{\bf I} \to \Taut_{full} ({\bf A}, {\bf B}) $ a diagram of
taut functors, $ F \colon {\bf A} \to {\bf B} $ a taut functor,
and $ \gamma \colon F \to \Phi $ a cone on $ \Phi $
with each $ \gamma I \colon F \to \Phi I $ a taut
transformation. Then the induced transformation
$$
\langle \gamma I \rangle \colon F \to \liml_I \Phi I
$$
is also taut.

\end{proposition}

\begin{proof}

$ \gamma $ is a natural transformation from the constant
diagram with value $ F $ to $ \Phi $, and by Proposition~\ref{Prop-LimTaut}
we get that
$$
\liml_I \gamma I \colon \liml_I F \to
\liml_I \Phi I
$$
is taut. Now, $ \liml_I F = F^{\pi_0 ({\bf I})} $,
the product of $ F $'s, one for each component of $ {\bf I} $,
and $ \langle \gamma I \rangle $ is the composite
$$
F \to^{\Delta} F^{\pi_0 ({\bf I})} \to^{\liml_I \gamma I}
\liml_I \Phi I \ .
$$
So we only have to show that $ \Delta $ is taut.

Let $ J $ be a non-empty set and $ m \colon 
A_0\ \  \to/>->/ A $ a mono in $ {\bf A} $. Let $ b $ and
$ \langle b_{0j} \rangle $ make
$$
\bfig

\node a(-300,700)[B]
\node b(0,0)[(FA_0)^J\ ]
\node c(0,500)[FA_0\ ]
\node d(550,500)[FA]
\node e(550,0)[(FA)^J]

\arrow|l|/{@{>}@/_1.6em/}/[a`b;\langle b_{0j} \rangle]
\arrow|l|/>/[c`b;\Delta_{FA_0}]
\arrow/..>/[a`c;]
\arrow|a|/{@{>}@/^1.1em/}/[a`d;b]
\arrow/>->/[c`d;]
\arrow|b|/>->/[b`e;(Fm)^J]
\arrow|r|/>/[d`e;\Delta_{FA}]

\efig
$$
commute, i.e.~$ F(m) b_{0j} = b $ for every $ j $. Choose
any $ j \in J $. Then $ b_{0j} \colon B \to FA_0 $
(dotted arrow) makes the top triangle commute. But
the left triangle also commutes because $ (Fm)^J $ is
monic. Thus the square is a pullback and $ \Delta $ is
taut.
\end{proof}

\begin{remark}

Note that $ \Delta $ is taut but not cartesian. For example,
the pullback of $ F f \times F f $ along $ \Delta $ is the
kernel pair $ K $ of $ Ff $
$$
\bfig
\square/>` >->` >->`>/<900,600>[K`FA`F A_0 \times FA_0`FA \times FA;
```Ff \times Ff]

\place(450,300)[\framebox{Pb}]

\efig
$$
which will be $ FA_0 $ only if $ Ff $ is monic.

\end{remark}

So every cone in $ \Taut ({\bf A}, {\bf B}) $ gives a unique
taut transformation into the limit but this still doesn't 
make it into a limit in $ \Taut ({\bf A}, {\bf B}) $, because,
as we said, the projections are
not necessarily taut. But for non-empty connected limits,
pullbacks or equalizers for example, it does.

\begin{proposition}
\label{Prop-ConnLim}

Suppose $ {\bf B} $ has $ {\bf I} $-limits and that $ {\bf I} $
is non-empty and connected, then $ \Taut ({\bf A}, {\bf B}) $
is closed under $ {\bf I} $-limits in $ \Cat ({\bf A}, {\bf B}) $.

\end{proposition}

\begin{proof}

Let $ \Phi \colon {\bf I} \to \Taut ({\bf A}, {\bf B}) $ be an
$ {\bf I} $-diagram and $ I_0 \in {\bf I} $. We want to show
that the projection
$$
p_0 \colon \liml \Phi \to \Phi I_0
$$
is taut. To this end, let $ m \colon A_0 \ \ \to/>->/ A $
be a mono in $ {\bf A} $ and consider the commutative
diagram in $ {\bf B} $
$$
\bfig
\node a(300,0)[(\Phi I_0)(A_0)]
\node b(1100,0)[(\Phi I_0)(A)]
\node c(300,550)[(\liml \Phi)(A_0)]
\node d(1100,550)[(\liml \Phi)(A)]
\node e(0,850)[B]

\arrow/ >->/[a`b;]
\arrow/ >->/[c`d;]
\arrow|l|/>/[c`a;p_0 A_0]
\arrow|r|/>/[d`b;p_0 A]
\arrow|l|/{@{>}@/_1.8em/}/[e`a;b_0]
\arrow|r|/..>/[e`c;?]
\arrow|a|/{@{>}@/^1.3em/}/[e`d;\langle b_I \rangle]

\efig
$$
where $ \langle b_I \rangle $ is the morphism induced by a cone
for $ \Phi $, and which we wish to show factors through
$ (\liml \Phi) (A_0) $. This will happen if
each $ b_I \colon B \to \Phi (I)(A) $ factors through
$ \Phi (I) (A_0)\ \to/>->/ \Phi (I) (A) $. Let $ {\bf J}
\subseteq {\bf I} $ be the full subcategory of $ {\bf I} $
determined by those $ J $ for which we have
$ b'_J $ such that
$$
\bfig
\btriangle/..>`>` >->/<800,400>[B`\Phi (J)(A_0)`(\Phi J)(A)\rlap{\ .};
b'_J`b_J`\Phi(J)(m)]

\efig
$$
If $ J \in {\bf J} $ and $ j \colon J \to K $ then $ K $
is in $ {\bf J} $ because
$$
\bfig
\node a(0,0)[\Phi (K)(A_0)]
\node b(800,0)[\Phi (K) (A)]
\node c(0,500)[\Phi (J)(A_0)]
\node d(800,500)[\Phi (J) (A)]
\node e(0,900)[B]

\arrow|b|/ >->/[a`b;\Phi (K)(m)]
\arrow|b|/ >->/[c`d;\Phi (J)(m)]
\arrow|l|/>/[c`a;\Phi(j)(A_0)]
\arrow|r|/>/[d`b;\Phi (j)(A)]
\arrow|l|/>/[e`c;b'_J]
\arrow|r|[e`d;b_J]

\efig
$$
commutes and the composite on the right is
$ b_K $, so $ b'_K = \Phi (j)(A_0) \cdot b'_J $.
Conversely, as $ \Phi (j) $ is taut, the above
square is a pullback, so if we have a $ b'_K $,
then there exists a unique $ b'_J $. This means
that if $ j \colon J \to K $ and $ K \in {\bf J} $
then so is $ J $. We're given that $ I_0 \in {\bf J} $
and as anything connected to it will also be in
$ {\bf J} $, we have $ {\bf J} = {\bf I} $.
\end{proof}

\subsection{Colimits of taut functors}
\label{SSec-ColimTaut}

We will only be concerned with $ {\bf Set} $ valued
taut functors, and for these, there are special colimit
commutations. The main concept is the following.

\begin{definition}
\label{Def-Confl}

A category $ {\bf I} $ is {\em confluent} if for every
pair of arrows with the same domain, $ \alpha_i \colon
I \to I_i \ (i = 1, 2) $ there exist $ \beta_i \colon I_i \to J $
making $ \beta_1 \alpha_1 = \beta_2 \alpha_2 $
$$
\bfig
\node a(300,0)[I_2]
\node b(0,200)[I]
\node c(300,400)[I_1]
\node d(600,200)[J]

\arrow|l|/>/[b`a;\alpha_2]
\arrow|l|/>/[b`c;\alpha_1]
\arrow|r|/>/[c`d;\beta_1]
\arrow|r|/>/[a`d;\beta_2]

\place(700,0)[.]

\efig
$$

\end{definition}

\begin{theorem}
\label{Thm-Confl}

$ {\bf I} $-colimits commute with inverse images
in $ {\bf Set} $ iff $ {\bf I} $ is confluent.
\end{theorem}

Some preliminaries before we prove this. Recall that
for a diagram $ \Gamma \colon {\bf I} \to {\bf Set} $,
$\limr\Gamma $ can be computed as the
set of equivalence class of pairs $ (I, x \in \Gamma I) $
where the equivalence relation is generated by
$$
(I, x) \sim (I', \Gamma (\alpha)(x))
$$
for $ \alpha \colon I \to I' $. So $ (I, x) \sim (I', x') $ iff
there exists a zigzag path
$$
I = I_0 \to^{\alpha_1} I_1 \to/<-/^{\alpha_2} I_2 \to^{\alpha_3}
\cdots \to/<-/^{\alpha_2n} I_{2n} = I'
$$
and $ x_i \in \Gamma I_i $ such that
$$
x_0 = x, \ x_{2n} = x'
$$
$$
\Gamma (\alpha_{2i - 1}) (x_{2i - 2}) = x_{2i - 1} =
\Gamma (\alpha_{2i})(x_{2i}) \ \ i = 1, \cdots , n
$$

$$
\bfig\scalefactor{.8}
\node a(0,0)[\Gamma I]
\node b(500,0)[\Gamma I_0]
\node c(1000,0)[\Gamma I_1]
\node d(1500,0)[\Gamma I_2]
\node e(2000,0)[]
\node f(2500,0)[]
\node g(3000,0)[\Gamma I_{2n}]
\node h(3500,0)[\Gamma I'\rlap{\ .}]

\node a'(0,400)[1]
\node b'(500,400)[1]
\node c'(1000,400)[1]
\node d'(1500,400)[1]
\node e'(2000,400)[]
\node f'(2500,400)[]
\node g'(3000,400)[1]
\node h'(3500,400)[1]

\arrow|l|/>/[a'`a;x]
\arrow|r|/>/[b'`b;x_0]
\arrow|r|/>/[c'`c;x_1]
\arrow|r|/>/[d'`d;x_2]
\arrow|r|/>/[g'`g;x_{2n}]
\arrow|r|/>/[h'`h;x']

\arrow/=/[a`b;]
\arrow|b|/>/[b`c;\Gamma \alpha_1]
\arrow|b|/<-/[c`d;\Gamma \alpha_2]
\arrow|b|/>/[d`e;\Gamma \alpha_3]
\arrow//[e`f;\dots]
\arrow|b|/<-/[f`g;\Gamma \alpha_{2n}]
\arrow/=/[g`h;]

\arrow/=/[a'`b';]
\arrow/=/[b'`c';]
\arrow/=/[c'`d';]
\arrow/=/[d'`e';]
\arrow//[e'`f';\dots]
\arrow/=/[f'`g';]
\arrow/=/[g'`h';]

\efig
$$

\begin{lemma}
\label{Lem-Confl}

If $ {\bf I} $ is confluent, then $ n $ can be taken to
be $ 1 $ in the above description, i.e.~$ (I, x) \sim (I', x') $
iff there exist $ \beta \colon I \to J $ and $ \beta' \colon I' \to J $
with
$$
\Gamma (\beta) (x) = \Gamma (\beta')(x') \ .
$$

\end{lemma}

\begin{proof}

The ``if'' part is obvious.

On the other hand, if the shortest path joining
$ (I, x) $ to $ (I', x) $ had $ n > 1 $, then we could
take $ \beta, \beta' $ s.t.~$ \beta \alpha_2 =
\beta' \alpha_3 $
$$
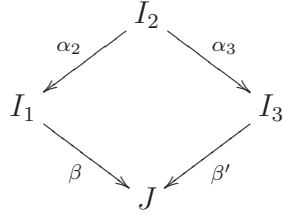
\bfig
\node a(400,0)[J]
\node b(0,300)[I_1]
\node c(800,300)[I_3]
\node d(400,600)[I_2]

\arrow|l|/>/[b`a;\beta]
\arrow|l|/>/[d`b;\alpha_2]
\arrow|r|/>/[d`c;\alpha_3]
\arrow|r|/>/[c`a;\beta']

\efig
$$
and replace
$ I_0 \to^{\alpha_1} I_1 \to/<-/^{\alpha_2} I_2 \to^{\alpha_3}
I_3 \to/<-/^{\alpha_4} I_4 $ by
$$
I_0 \to^{\beta \alpha_1} J \to/<-/^{\beta' \alpha_4} I_4 \ .
$$
Then a simple calculation shows that
$ \Gamma (\beta \alpha_1) (x_0) =
\Gamma (\beta' \alpha_4) (x_2) $, so we take 
$ y \in \Gamma J $ to be that common value and we get
a shorter path.
\end{proof}

\begin{lemma}
\label{Lem-ConflSub}

Let $ {\bf I} $ be confluent and $ \Gamma_0 $ a subdiagram
of $ \Gamma $, 
$ \Gamma_0 \subseteq \Gamma \colon {\bf I} \to {\bf Set} $.
Then limit $ \limr \Gamma_0 $ is isomorphic to
the set of elements $ [I, x] $ in $ \limr \Gamma $
such that there exists $ \beta \colon I \to J $ with
$ \Gamma (\beta) (x) \in \Gamma_0 J $.

\end{lemma}

\begin{proof}

An element of $ \limr \Gamma_0 $ is an
equivalence class of pairs $ (I, x) $ with $ x \in \Gamma_0 I $,
which for now we denote $ [I, x]_0 $. To it there is a
well-defined element $ [I, x] $ of $ \limr \Gamma $,
whether $ {\bf I} $ is confluent or not, but if $ {\bf I} $ is
confluent the equivalence relations as described in Lemma~\ref{Lem-Confl}
are the same for $ \Gamma $ and $ \Gamma_0 $ (crucial point!), so
the function $ [I, x]_0 \mapsto [I, x] $
is one-to-one. The classes are not the same -- $ [I, x] $
may contain more elements -- but we can identify
$ \limr \Gamma_0 $ with those elements of
$ \limr \Gamma $ that have a representative
in $ \Gamma_0 $. So $ [I, x] $ is in the image of
$ \limr \Gamma_0 $ iff there exist
$ x' \in \Gamma_0 I' $, $ \beta \colon I \to J $,
$ \beta' \colon I' \to J $ with
$
\Gamma (\beta) (x) = \Gamma (\beta') (x') \in \Gamma_0 J $.
Then, if $ x $ is in the image of $ \Gamma_0 $, there does exist
a $ \beta $ with $ \Gamma(\beta) (x) \in \Gamma_0 J $,
and conversely, if there is such a $ \beta $ we can always
take $ \beta' = 1_J $.
\end{proof}

From now on, we will identify $ \limr \Gamma_0 $
with its image in $ \limr \Gamma $.

\begin{proof}[of theorem]

First of all, let $ {\bf I} $ be confluent and
$$
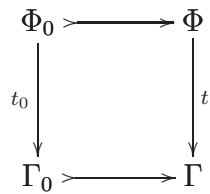
\bfig
\square/ >->`>`>` >->/[\Phi_0`\Phi`\Gamma_0`\Gamma;
`t_0`t`]

\efig
$$
be an inverse image diagram in $ {\bf Set}^{\bf I} $,
i.e.~$ \Phi_0 (I) = t(I)^{-1} \Gamma_0 (I) $, and take
the colimits
$$
\bfig
\square/ >->`>`>` >->/<550,500>[\limr\Phi_0`\limr\Phi
`\limr \Gamma_0`\limr \Gamma\rlap{\ .};
`\limr t_0`\limr t`]

\efig
$$
So $ \limr \Phi_0 $ is contained in the inverse image of
$ \limr \Gamma_0 $. An element in that inverse image is
$ [x, I] \in \limr \Phi $ such that $ [t(I) (x), I ] $ is in
$ \limr \Gamma_0 $, i.e.~there is $ \beta \colon I \to J $
such that $ \Gamma (\beta) t (I) (x) \in \Gamma_0 J $.
But $ \Gamma (\beta) t (I) (x) = t (J) \Phi (\beta) (x) $
and as it is in $ \Gamma_0 J $, $ \Phi (\beta) (x) \in
\Phi_0 J $, which implies $ [x, I] \in \limr \Phi_0 $. Thus
the inverse image of $ \limr \Gamma_0 $ is equal to
$ \limr \Phi_0 $, and confluent colimits commute with
inverse images in $ {\bf Set} $.

Let $ {\bf I} $ be any small category and
$$
\bfig
\Ctriangle/<-``>/<350,250>[I_1`I`I_2;\alpha_1``\alpha_2]

\efig
$$
be arrows in $ {\bf I} $. Consider the following inverse
image diagram
$$
\bfig
\morphism(0,0)/->>/<500,0>[{\bf I} (I_2, -)`\Gamma_0;]

\square(500,0)/ >->`>`>` >->/[\Phi_0`{\bf I}(I_1,-)
`\Gamma_0`{\bf I} (I,-)\rlap{\ .};
``{\bf I} (\alpha_1,-)`]

\efig
$$
Here $ \Gamma_0 $ is the image of the natural
transformation $ {\bf I} (\alpha_2, -) $ and $ \Phi_0 $
the inverse image of $ \Gamma_0 $ under
$ {\bf I} (\alpha_1, -) $. The colimit of a representable
is always $ 1 $, so taking colimits we get
$$
\bfig
\morphism(0,0)/->>/<500,0>[1`\limr\Gamma_0;]

\square(500,0)[\limr \Phi_0`1`\limr\Gamma_0`1\rlap{\ .};```]

\efig
$$
Then $ \limr \Gamma_0 $ must be $ 1 $ also, and
if $ {\bf I} $-colimits commute with inverse images,
$ \limr \Phi_0 $ also has to be $ 1 $. Now, $ \Phi_0 (J) $
is the set of all morphisms $ \beta_1 \colon I_1 \to J $ such
that there exists $ \beta_2 \colon I_2 \to J $ with
$ \beta_1 \alpha_1 = \beta_2 \alpha_2 $,
$$
\bfig
\Ctriangle/<-``>/<350,250>[I_1`I`I_2;\alpha_1``\alpha_2]

\Dtriangle(350,0)/`>`<-/<350,250>[I_1`J`I_2;`\beta_1`\beta_2]

\place(700,0)[.]

\efig
$$
So if $ \limr \Phi_0 $ is to be $ 1 $ there has to exist at
least one such pair $ (\beta_1, \beta_2) $. I.e.~$ {\bf I} $
is confluent.
\end{proof}

\begin{remark}

Saying that $ {\bf I} $-colimits commute with inverse images
means exactly that the colimit functor
$$
\limr \colon {\bf Set}^{\bf I} \to {\bf Set}
$$
is taut.

\end{remark}

\begin{proposition}
\label{Prop-ColimTaut}

$ \Taut_{full} ({\bf A}, {\bf Set}) $ is closed under confluent
colimits. If $ t $ is a natural transformation of confluent diagrams in
$ \Taut_{full} ({\bf A}, {\bf B}) $, 
and all the values of $ t $ are taut, then the induced
morphism $ \limr  t $ is also taut.

\end{proposition}

\begin{proof}

The proof of Proposition~\ref{Prop-LimTaut} carries over
verbatim, only using that confluent colimits commute with
inverse images in $ {\bf Set} $.
\end{proof}

Proposition~\ref{Prop-TautCone} also ``dualizes'' to
confluent colimits with some modifications.

\begin{proposition}
\label{Prop-TautColim}

Let $ {\bf I} $ be confluent and $ \Phi \colon {\bf I}
\to \Taut_{full} ({\bf A}, {\bf Set}) $ a diagram of taut
functors, $ F \colon {\bf A} \to {\bf Set} $ a taut
functor, $ \gamma \colon \Phi \to F $ a cocone with
each $ \gamma I $ taut, then the induced
transformation
$$
[\gamma I] \colon \limr_I \Phi I \to F
$$
is also taut.

\end{proposition}

\begin{proof}

By Proposition~\ref{Prop-ColimTaut} we get that
$$
\limr_I \gamma I \colon \limr_I \Phi I \to \limr_I F
$$
is taut, and $ \limr_I F = \pi_0 {\bf I} \times F $, the
coproduct of $ \pi_0 {\bf I} $ copies of $ F $. So it
will be sufficient to show that the codiagonal
$ \nabla \colon \pi_0 {\bf I} \times F \to F $
is taut, which is now clear from the pullback diagram
$$
\bfig
\square/ >->`>`>` >->/<800,500>[\pi_0 {\bf I} \times F(A_0)`
\pi_0 {\bf I} \times F (A)`FA_0`FA\rlap{\ .};
`\nabla`\nabla`]

\efig
$$
\end{proof}

\begin{proposition}
\label{Prop-TautCl}

$ \Taut ({\bf A}, {\bf Set}) $ is closed under confluent
colimits in $ \Cat ({\bf A}, {\bf Set}) $.

\end{proposition}

\begin{proof}

In view of the previous proposition, we only have to show
that the colimit injections
$$
j = j I_0 \colon \Phi (I_0) \to \limr_I \Phi I
$$
are taut. Let $ A_0   \ \ \to/>->/ \ A $ be a mono and
consider the commutative square, which we want to
show is a pullback
$$
\bfig
\square/ >->`>`>` >->/<800,500>[\Phi (I_0)(A_0)
`\Phi (I_0)(A)`\limr_I \Phi (I) (A_0)`\limr_I \Phi (I)(A)\rlap{\ .};
`j A_0`j A`]

\efig
$$
The elements of $ \limr_I \Phi (I) (A) $ are equivalence
classes $ [I, x \in \Phi (I) (A)] $ and from
Lemma~\ref{Lem-ConflSub} we can identify
$ \limr_I \Phi (I) (A_0) $ with those classes for which
there exist $ J $ and $ \beta \colon I \to J $ such that
$ \Phi (\beta) (A) (x) \in \Phi (J) (A_0) $. Let
$ x_0 \in \Phi (I_0) (A) $ be such that 
$ j (A) (x_0) \in \limr_I \Phi (I) (A_0) $. $ j(A) (x) =
[I_0, x_0] $ so there exists $ \beta \colon I_0 \to J $
with $ \Phi (\beta) (A) (x_0) \in \Phi (J) (A_0) $.
$ \Phi (\beta) $ is taut so the following is a pullback
$$
\bfig
\square/ >->`>`>` >->/<700,500>[\Phi (I_0) (A_0)`\Phi (I_0)(A)
`\Phi (J) (A_0)`\Phi (J) (A)\rlap{\ .};
`\Phi (\beta)(A_0)`\Phi (\beta)(A)`]

\efig
$$
So $ \Phi (\beta)(A) (x_0) \in \Phi (J)(A_0) $ implies
$ x_0 \in \Phi (I_0)(A_0) $, which shows that our original
square is a pullback, completing the proof.
\end{proof}

\begin{corollary}
\label{Cor-FiltTaut}

Filtered colimits of taut functors into $ {\bf Set} $ are taut.

\end{corollary}

\begin{corollary}
\label{Cor-GrQuoTaut}

The quotient of a taut functor into $ {\bf Set} $ by a group
action is taut.

\end{corollary}

Coproducts of set-valued functors will play a central role
in what follows so we end the section on colimits with 
some results specifically about them.

\begin{proposition}
\label{Prop-CoprodTaut}

A coproduct of functors into $ {\bf Set} $ is taut
if and only if each summand is taut.

\end{proposition}

\begin{proof}

A discrete category is confluent so a coproduct of
taut functors is taut, by Proposition~\ref{Prop-TautCl}
e.g. This is also easy to see directly.

Conversely, if $ \sum_{i \in I} F_i $ is taut, then so is
each $ F_i $ as it is a pullback of taut functors
$$
\bfig
\square<550,500>[F_i`\sum_{i \in I} F_i`1`\sum_{i \in I} 1\rlap{\ .};
```\ j_i]

\place(250,250)[\framebox{Pb}]

\efig
$$
\end{proof}

It is well-known that for a small category $ {\bf A} $,
every functor $ F \colon {\bf A} \to {\bf Set} $ is a
coproduct of indecomposable functors indexed by a
set $ \pi_0 F $ called the connected components of
$ F $ (see e.g.~\cite{BarPar}). $ \pi_0 $ is left adjoint
to the diagonal functor $ D \colon {\bf Set} \to {\bf Set}^{\bf A} $,
i.e.~it is the colimit functor. $ F $ is {\em connected} if
$ \pi_0 F = 1 $. Here we are concerned with the case
where $ {\bf A} = {\bf Set} $, which is, of course, not small.
If $ {\bf A} $ is a large category, then
$ \pi_0 F $ may well be a proper class, and things
fall apart. We don't get the adjointness to $ D $
for example. But, if $ {\bf A} $ has a terminal object,
everything is much nicer.

\begin{proposition}
\label{Prop-SumDecomp}

Let $ {\bf A} $ be a category with terminal object $ 1 $.
Then for a functor $ F \colon {\bf A} \to {\bf Set} $,
\begin{itemize}
	\item[(1)] $ \pi_0 F \cong F 1 $
	
	\item[(2)] $ F $ is connected (indecomposable)
	if and only if $ F 1 = 1 $
	
	\item[(3)] Every $ F $ is a coproduct of connected
	functors.

\end{itemize}

\end{proposition}

\begin{proof}

The unique morphism $ \tau X \colon X \to 1 $ for
every $ X $ gives a natural transformation from the
identity on $ {\bf Set} $ to the constant functor $ 1 $,
$ \tau \colon \id_{\bf Set} \to 1 $. If we apply $ F $ to
it we get a natural transformation $ F \tau \colon F \to F1 $,
which is easily seen to be a colimit cocone, which
gives (1). (2) is a trivial consequence of (1). Finally,
$ F 1 = \sum_{i \in F 1}  1 $ which gives a decomposition
of $ F $ into a coproduct $ F \cong \sum_{i \in F1} F_i $
where $ F_i $ is the pullback along $ i \colon 1 \to F1 $
$$
\bfig
\square[F_i`F`1`\sum_{i \in F1} 1\rlap{\ .};
```i]

\place(250,250)[\framebox{Pb}]

\efig
$$
As colimits are stable under pullback, we get
$ \limr F_i \cong 1 $. This is the decomposition
claimed in (3).
\end{proof}

\begin{corollary}
\label{Cor-SumDecomp}

If $ {\bf A} $ has a terminal object, then every
taut functor from $ {\bf A} $ to $ {\bf Set} $ is
a coproduct of connected taut functors.

\end{corollary}

There is a cancellation property
for functors into $ {\bf Set} $ which will be used in
the following sections. It is not deep but a little
delicate and best stated explicitly. Coproduct of
functors into $ {\bf Set} $ is not cancellative
$$
F + G \cong F + H \not\Rightarrow G \cong H \ .
$$
If we take $ {\bf A} = 1 $, we have $ {\mathbb N} + 1
\cong {\mathbb N} + 2 $ but $ 1 \not\cong 2 $, and
we can promote this to endofunctors of $ {\bf Set} $,
which is where we will be using it, by taking
constant functors with values $ {\mathbb N}, 1, 2 $.
However it is cancellative if we take injections into
account.

\begin{proposition}
\label{Prop-Cancellation}

Let $ F, G, H \colon {\bf A} \to  {\bf Set} $ be functors
and assume we have an isomorphism $ \phi $
commuting with injections
$$
\bfig
\Vtriangle/>`<-`<-/<350,450>[F + G`F+ H`F;
\phi`j`j]

\place(700,0)[,]

\efig
$$
then $ \phi $ restricts to an isomorphism
$ \psi \colon G \to H $.

\end{proposition}

\begin{proof}

For any $ A $, if $ x \in GA $ then $ \phi (A)(x) $
is in $ H (A) $ for if it weren't it would be in $ FA $
(on the right) and so $ x = \phi^{-1} (A) \phi (A) $
would be in $ FA $ on the left which it is not. So
$ \phi $ restricts to
$$
\bfig
\square/>` >->` >->`>/[G`H`F + G`F + H\rlap{\ .};
\psi`j'`j'`\phi]

\efig
$$
The same argument applied to $ \phi^{-1} $
gives $ \psi^{-1} $.
\end{proof}

The proof uses that $ \phi $ is invertible. Clearly
an arbitrary natural transformation would not
restrict. It is also specific to functors into
$ {\bf Set} $. It is false for $ {\bf Set}^{op} $
for example.

% !TEX root = taut.tex

\section{Some special classes of taut functors}
\label{Sec-SpecCl}

The class of taut functors is quite large as the
following examples will show.

\subsection{Polynomials}
\label{SSec-Poly}

Classically a polynomial is an expression of the form
\begin{equation}\tag{*}
P(X) = C_0 + C_1 X + C_2 X^2 + \cdots
+ C_d X^d = \sum^d_{n\,=\,0} C_n X^n
\end{equation}
which, if we want, can be interpreted in any category
with finite sums and products, and will produce a
polynomial endofunctor on that category. Of course,
the quality of these functors will depend on whether
the category in question has good properties,
e.g.~products distribute over sums. We're concerned
with the category of sets which has all
the properties we want and more.

Given sets $ C_0, C_1, \cdots, C_d $, (*) defines a
(finitary) polynomial functor
$$
P \colon {\bf Set} \to {\bf Set}
$$
where for a set $ X $, $ X^n $ is the set of $ n $-tuples
in $ X $ and $ C_n X^n $ is the cartesian product
$ C_n \times X^n $, and $ + $ is coproduct (disjoint
union). So $ X $ represents the identity functor
$ \Id (X) = X $ and $ X^n $ is the product of
$ n $-copies of $ X $. As $ \Id $ is taut and products
and sums of taut functors are taut we get the
following.

\begin{proposition}
\label{Prop-FinPolyTaut}

Any finitary polynomial functor is taut.

\end{proposition}

Unencumbered by questions of convergence we can
define power series to be
\begin{equation}\tag{**}
P (X) = \sum^\infty_{n\, = \, 0} C_n X^n  
\end{equation}

giving rise to power series functors
$$
P \colon {\bf Set} \to {\bf Set} \ .
$$

\begin{proposition}
\label{Prop-PowerTaut}

Power series functors are taut.

\end{proposition}

While we're at it, we may as well let the powers be
any sets and the sum also indexed by a set. This is
indeed the natural notion if we're thinking of
$ {\bf Set} $ as a categorified ring with coproduct
as addition and product as multiplication. Now it's more
natural not to collect like powers and simply allow
repetitions.

\begin{definition}
\label{Def-Poly}

A {\em polynomial} is an expression of the
form
$$
P (X) = \sum_{i \,\in\, I} X^{A_i}
$$
where $ I $ is an index set and
$ \langle A_i \rangle_{i\,\in\,I} $ is a family of sets
indexed by $ I $. This determines a {\em polynomial functor}
$$
P \colon {\bf Set} \to {\bf Set}
$$
given by the above formula.

\end{definition}

Again we have:

\begin{proposition}
\label{Prop-PolyTaut}

Polynomial functors are taut.

\end{proposition}

There is an extensive body of work on polynomial
functors. They have been around for a long time in
various settings and at different levels of abstraction
and are still an active area of research. It is beyond
the scope of this paper, and our competence, to give
a comprehensive and accurate history of the subject.
We mention only three names -- Andr\'e Joyal,
Joachim Kock, and David Spivak -- and three
references -- \cite{KocJoyBatMas10, GamKoc13, NiuSpi23}, --
with apologies to all those not mentioned. The
reader is referred to the historical comments and
references in these works.

A particular feature of polynomial functors is that
they have morphisms in contrast to polynomials
over a field or ring. Morphisms of polynomial functors
are simply natural transformations. We've seen that
the projection $ p_1 \colon X^2 \to X $ is not taut
so not all morphisms are taut.

As a polynomial functor is a sum of powers and
powers are just representable functors it is easy to
analyze morphisms in terms of the families of
powers. If $ P (X) = \sum_{i\,\in\, I} X^{A_i} $ and
$ Q (X) = \sum_{j\,\in \, J} X^{B_j} $, we have the
following bijections

\begin{center}
\begin{tabular}{c} 
$ t \colon P \to Q $ \\[3pt]  \hline \\[-12pt]
$ t \colon \sum_{i \in I} X^{A_i} \to \sum_{j\in J} X^{B_j} $  \\[3pt] \hline \\[-12pt]
$ \langle X^{A_i} \to \sum_{j\in J} X^{B_j} \rangle_{i \in I} $ \\[3pt] \hline \\[-12pt] 
$ \alpha \colon I \to J \ \&\  \langle X^{A_i} \to X^{B_{\alpha (i)}} \rangle_{i\in I} $ \\[3pt] \hline \\[-12pt]
$ \alpha \colon I \to J \ \&\  \langle f_i \colon B_{\alpha (i)} \to A_i \rangle . $
\end{tabular}
\end{center}

\noindent The third bijection is because $ X^{A_i} = {\bf Set} (A_i, -) $
is freely generated by a single element, $ 1_{A_i} $, so the
transformation must factor through some injection, the $ \alpha (i)^{th} $.
The fourth bijection is the Yoneda lemma.

Working back up the bijections we see that given
$ (\alpha, \langle f_i \rangle) $ as above, the natural
transformation
$$
t \colon P \to G
$$
is given as follows. An element of $ P (X) $ is a pair
$ (i, \phi) $ where $ \phi \colon A_i \to X $. $ t $ then
sends $ (i, \alpha) $ to $ (\alpha (i), \phi f_i) $ in
$ Q (X) $.

In \cite{NiuSpi23} a morphism $ (\alpha, \langle f_i \rangle) $ is
called {\em vertical} if $ \alpha $ is an isomorphism,
and it is {\em cartesian} if all $ f_i $ are isomorphisms.
Here we have something weaker than cartesian.

\begin{proposition}
\label{Prop-MorPolyTaut}

A morphism $ (\alpha, \langle f_i \rangle) $ of polynomials
is taut if and only if each of the $ f_i $ is an
epimorphism.

\end{proposition}

\begin{proof}

Let $ t \colon P \to Q $ be given by $ (\alpha, \langle f_i \rangle) $
as above and assume each of the $ f_i $ is an epimorphism.
Let $ X_0 \ \ \to/>->/ \ X $ be a monomorphism and consider
the commutative diagram
$$
\bfig
\square/ >->`>`>` >->/<750,500>[\sum_{i\,\in \, I} X_0^{A_i}
`\sum_{i\,\in \, I} X^{A_i}`\sum_{j\,\in \, J} X_0^{B_j}
`\sum_{j\,\in \, J} X^{B_j}\rlap{\ .};
`t(X_0)`t(X)`]

\efig
$$
Let $ (i, \phi) \in \sum_{i\,\in\,I} X^{A_i} $ be such that
$ t(X)(i, \phi) \in \sum_{j\,\in\,J} X_0^{B_j} $, i.e.~$ \phi f_i $
factors through $ X_0 $, giving $ \psi $ and a commutative
square
$$
\bfig
\square/->>`>`>` >->/[B_{\alpha(1)}`A_i`X_0`X\rlap{\ .};
f_i`\psi`\phi`]

\morphism(500,500)/-->/<-500,-500>[A_i`X_0;]

\efig
$$
By the diagonal fill-in property of factorizations we
have the dotted arrow above, i.e.~$ (i, \phi) \in
\sum_{i\,\in\,I} X_0^{A_i} $, showing that the original
square is a pullback. Thus $ t $ is taut.

Conversely, if $ t $ is taut, then for every mono
$ X_0 \to/ >->/ X $ and $ \phi $, $ \psi $ as above
we have the fill-in, so by the orthogonality property
of factorizations, the $ f_i $ are epis.
\end{proof}

\subsection{Divided powers}
\label{SSec-DivPow}

Divided power series are expressions of the form
$$
\sum^\infty_{n\,=\,0} a_n x^{[n]}
$$
where $ x^{[n]} $ is to be thought of as $ x / n! $. They
form a ring with componentwise addition, and
multiplication given by the bilinear extension of
$$
x^{[n]} x^{[m]} = \binom{n + m}{n}
 x^{[n + m]}
 $$
 as the intended meaning of $ x^{[n]} $ would suggest.
 Over a field of characteristic $ 0 $, we get a ring
 isomorphic to the ring of formal power series, but
 over a field of finite characteristic, or a ring, we get
 something different which often has better properties.
 
 In keeping with our program of extending ring
 theoretic concepts to $ {\bf Set} $ we can define
 the divided power $ X^{[n]} $ as follows. The
 symmetric group $ S_n $ acts on the right on
 $ X^n $,
 $$
 (x_1 \dots x_n)^\sigma = (x_{\sigma 1}, x_{\sigma 2},
 \dots , x_{\sigma n})
 $$
 and we take the quotient by this action
 $$
 X^{[n]} = X^n/S_n \rlap{\ .}
 $$
 If $ {\bf S}_n $ is the symmetric group considered as
 a one object category, a right action corresponds to
 a functor
 $$
 {\bf S}^{op}_n \to {\bf Set}
 $$
 and the quotient $ X^{[n]} $ is its colimit. So by
 Proposition~\ref{Prop-TautCl}, $ F (X) = X^{[n]} $
 gives a taut functor $ {\bf Set} \to {\bf Set} $.
 
 A {\em divided power series functor} $ F \colon
 {\bf Set} \to {\bf Set} $ is one given by
 $$
 F X = \sum_{n\,=\,0}^\infty C_n X^{[n]} \rlap{\ .}
 $$
 From the discussion above we get immediately
 the following.

 \begin{proposition}
 \label{Prop-DivPowSer}
 
 Any divided power series is taut.
 
 \end{proposition}

 Divided power series functors are not polynomial
 functors except in the affine case
 $$
 C_0 + C_1 X^{[1]} \rlap{\ .}
 $$
 They don't preserve pullbacks. E.g.~$ F X = X^{[2]} $
 takes the pullback below to a commutative diagram
 which is clearly not a pullback
 $$
 \bfig
 \square[4`2`2`1;```]
 
 \place(250,250)[\framebox{\scriptsize Pb}]
 
 \place(900,250)[\longmapsto]
 
 \square(1400,0)[10`3`3`1\rlap{\ .};```]
 
  \efig
 $$
 But they provide more examples of taut functors.
 
 The analogy with real divided power series is
 nice but only goes so far. Note, for example,
 that if $ X $ is a finite cardinal $ k $ then the
 cardinality of $ X^n $ is $ k^n $ but the cardinality
 of $ X^{[n]} $ is not $ k^n/n! $, which is not even
 an integer in general. Instead, its cardinality is
 \begin{equation}\tag{*}
 \frac{k^{\uparrow n}}{n!} \ =\  \frac{k(k + 1) \dots (k + n - 1)}{n!} 
 \end{equation}
 which is, of course, an integer, a binomial coefficient
 in fact.
 
 More importantly, they are not closed under
 binary products. We definitely don't have
 $$
 X^{[n]} \times X^{[n]} \cong
 \left( \begin{array}{c}
 n + m\\
 n
 \end{array} \right) X^{[n + m]}
 $$
 which is obvious if we let $ X = 1 $. Even if we
 thought it might be a series
 $$
 X^{[n]} \times X^{[n]} \cong
 \sum_{i\,=\,0}^\infty C_i X^{[i]} \rlap{\ ,}
 $$
 by letting $ X = 1 $ we see that all the $ C_i $
 must be $ 0 $ except for one of them, $ C_l = 1 $, so that
 we would have
 $$
 X^{[n]} \times X^{[n]} \cong X^{[l]} \rlap{\ .}
 $$
 This is impossible just on the grounds of cardinality: 
 (*) is a polynomial in $ k $ with leading term
 $ k^n/n! $ so we would need
 $$
 k^n/n! \cdot   k^m/m! = k^l/l!
 $$
 for all $ k $, which we can't have.
 
 What we do have is
 $$
 X^n/S_n \times X^m/S_m \cong X^{n + m}/S_n \times S_m
 $$
 where $ S_n \times S_m $ acts on $ X^{n + m} \cong
 X^n \times X^m $ in the obvious way, i.e.~componentwise.
 This leads to taking quotients of $ X^n $ by a subgroup
 $ G $ of $ S_n $ rather than the full $ S_n $. Then
 $ G $ acts on the left on $ X^n $ and we can take
 the quotient
 $$
 X^n/G \rlap{\ .}
 $$
 Now we can redefine our divided power series as follows.

 \begin{definition}
 \label{Def-ExDivPow}
 
 An {\em extended divided power series} is an expression
 of the form
 $$
 \sum_{n\,=\,0}^\infty \left( \sum_{i\,\in\,I_n} X^n/G_i \right)
 $$
 where, for each $ n $, $ \langle G_i \rangle_{i\,\in\,I_n} $
 is a family of subgroups of $ S_n $.
 
 \end{definition}

 \begin{proposition}
 \label{Prop-ExDivPow}
 
 Extended divided power series define taut functors.
 
 \end{proposition}

 Dare we write $ \sum_{i\,\in\,I_n} X^n/G_i $ as
 $ \sum_{i\,\in\,I_n} \left( 1/G_i\right) X^n $ and
 think of $ \sum1/G_i $ as some kind of rational
 number?
 
 Note that power series are now special cases of
 extended divided power series, by taking all of the
 $ G_i $ to be trivial.
 
 We can replace finite powers by arbitrary ones. For
 a set $ A $, $ S_A $ is the group of bijections
 $ A \to A $. $ S_A $ acts on the right on $ X^A $
 and we take the quotient by that action to get a
 taut functor
 $$
 F X = X^A/S_A \rlap{\ .}
 $$
 As before, no need to take the full symmetric group.
 We can restrict to a subgroup $ G \leq S_A $ and get
 $$
 F X = X^A/G \rlap{\ .}
 $$
 A bit more canonically, we can take any group $ G $
 and $ A $ a left $ G $-set. Then $ G $ will act on the
 right on $ X^A $ and we get again a taut functor.
 
 A {\em generalized divided power series functor} is one
 of the form
 $$
 F X = \sum_{i\,\in\,I} X^{A_i}/G_i
 $$
 where $ I $ is a set, $ \langle G_i \rangle $ a family
 of groups and $ \langle A_i \rangle $ a family of
 left $ G_i $-sets. $ F $ is taut.
 
 If we take all the $ G_i $ to be trivial we get
 polynomial functors and if we take $ S_n $,
 $ n \in {\mathbb N} $ we get the ``classical''
 divided power series.

 \subsection{Analytic functors}
 \label{SSec-Analytic}
 
 Joyal's original definition of {\em analytic functor}
 \cite{Joy81} was in the context of his reformulation
 of combinatorics in categorical terms and was strictly
 finitary. They have since been generalized in many directions
 (see \cite{GamJoy17} and references there).
 
 In \cite{Joy81}, a {\em species (of structure)} was
 defined as a functor $ F \colon {\bf Bij} \to {\bf Bij} $,
 where $ {\bf Bij} $ is the category of finite cardinals
 and bijections, but
 for general purposes there is no problem in taking
 $ F \colon {\bf Bij} \to {\bf Set} $. So $ F $ is a sequence of left
 $ S_n $-sets, $ Fn $, one for each $ n \in {\mathbb N} $.
 
 A species $ F $ determines an {\em analytic functor
 $ \widetilde{F} $} as the left Kan extension of $ F $
 along the inclusion of $ {\bf Bij} $ in $ {\bf Set} $
 $$
 \bfig
 \ptriangle/>` >->`<-/<600,500>[{\bf Bij}`{\bf Set}`{\bf Set};
 F``\widetilde{F}]
 
 \morphism(120,400)/=>/<100,-100>[`;]
 
 \place(600,0)[.]
 
 \efig
 $$
 $ \widetilde{F} $ is given by the formula
 $$
 \widetilde{F} X = \int^{n\in {\mathbb N}} X^n \times Fn
 $$
 which can be calculated as a colimit
 $$
 \widetilde{F} X = \limr_{a\in Fn} X^n
 $$
 taken over the category of elements of $ F $,
 $ {\bf El}(F) $. As $ {\bf Bij} $ is a groupoid,
 so is $ {\bf El} (F) $, and therefore $ \widetilde{F} $
 is taut.
 
 \begin{proposition}
 \label{Prop-AnalyticTaut}
 
 The classical analytic functors are taut.
 
 \end{proposition}

 Let's reformulate this a bit in order to compare
 it with power series and divided power series. Just
 like any colimit, a colimit over a groupoid is the
 coproduct of the individual colimits over the
 connected components of the groupoid. If we
 choose a representative object in each component,
 then these colimits are colimits taken over groups,
 the automorphism groups of the chosen objects.
 So an analytic functor can be written as
 $$
 \sum^\infty_{n\,=\, 0} X^n \otimes_{S_n} C_n
 $$
 for a family of left $ S_n $-sets $ C_n $. The
 $ C_n $ are the $ Fn $ from above, the notation
 chosen meant to suggest ``coefficients'', and to
 agree more with the previous sections.
 
 Just to be clear, $ X^n \otimes_{S_n} C_n $, which
 is the coend or colimit from above, is explicitly
 described as the set of all equivalences classes
 $ [x_1, \dots, x_n; c] = \langle x_i \rangle \otimes c $
 for $ x_i \in X $ and $ c \in C $. The equivalence
 relation is
 $$
 (x_1, \dots, x_n; c) \sim (y_1, \dots, y_n; d)\ 
 \mbox{iff}\  \exists \sigma \in S_n (\bigwedge_i y_i =
 x_{\sigma i}) \wedge (c = \sigma d)
 $$
 or put differently
 $$
 \langle x_i \rangle \otimes \sigma d =
 \langle x_{\sigma i} \rangle \otimes d \rlap{\ .}
 $$
 
 Now, $ X^n \cong X^n \otimes_{S_n} S_n $ so
 power series are analytic, and $ X^{[n]} \cong
 X^n \otimes_{S_n} 1 $ so divided power series
 are also analytic. In fact, the extended divided
 power series are analytic functors. For any subgroup
 $ G $ of $ S_n $, the left cosets of $ G $, $ S/G $
 form a left $ S_n $-set as usual, and
 $$
 X^n \otimes_{S_n} S_n/G \cong X^n/G \rlap{\ .}
 $$
 
 Note in passing that
 $$
 X^n/S_n \times X^m/S_m \cong
 X^{n+m}/(S_n \times S_m) \cong X^{n+m}
 \otimes_{S_{n+m}} (S_{n+m}/S_n \times S_m)
 $$
 which is the proper way of generalizing the equation
 $$
 x^{[n]} x^{[m]} = \binom{n + m}{n} x^{[n+m]} \rlap{\ .}
 $$
 The cardinality of $ S_{n+m}/S_n \times S_m $ is
 indeed $ \binom{n + m}{n} $.
 
 But analytic functors can also be presented as
 extended divided power series, so determine the
 same class of taut functors just presented
 differently. Indeed, if $ C $ is a left $ S_n $-set, it is a sum
 of indecomposables and these are isomorphic to
 left cosets of some subgroup of $ S_n $. So
 $$
 C = \sum_{i\,\in\,I} S_n/G_i
 $$
 and then
 $$
 X^n \otimes_{S_n} C \cong \sum_{i\,\in\,I} X^n/G_i \rlap{\ .}
 $$
 
 So our fanciful $ \sum 1/G_i $ from the previous section
 is just $ \sum S_n/G_i $ in the language of analytic
 functors.
 
 We summarize the above discussion in the following.

 \begin{proposition}
 \label{Prop-DivPowRAn}
 
 Extended divided power series and analytic functors
 determine the same class of taut functors.
 
 \end{proposition}

 We can generalize analytic functors, as we did for
 divided powers, allowing arbitrary powers not just
 finite ones.
 
 \begin{definition}
 \label{Def-GenAnFun}
 
 A {\em generalized analytic functor} is given by
 $$
 F X = \sum_{i\,\in\,I} X^{A_i} \otimes_{G_i} C_i
 $$
 where $ \langle G_i \rangle_{i\,\in\,I} $ is a family
 of groups, and $ \langle A_i \rangle_{i\,\in\,I} $
 and $ \langle C_i \rangle_{i\,\in\,I} $ are families
 of left $ G_i $-sets.
 
 \end{definition}

 As in the finitary case we have the following.
 
 \begin{proposition}
 
 \begin{itemize}
 
 	\item[(1)] Generalized analytic functors are taut.
	
	\item[(2)] Polynomial functors are generalized
	analytic.
	
	\item[(3)] Generalized divided power series give the
	same class of taut functors as generalized analytic
	functors.
 
 \end{itemize}
 
 \end{proposition}

 \subsection{Reduced powers}
 \label{SSec-RedPow}
 
 So far all of our examples, interesting in their own
 right, have been special cases of generalized
 symmetric power series functors, which brings us
 to reduced powers, which will give a new class.
 
 Reduced powers are a lot like symmetric powers
 in that they are (with one exception) quotients of powers,
 i.e.~can be defined as equivalence classes of functions from a fixed
 exponent into a variable set, but in fact give us a
 completely new class of taut functors as we will
 see below. The reader is referred to Blass' paper
 \cite{Bla76}, mentioned in the introduction, where
 he proves, among other things, that every left
 exact endofunctor is a directed union of reduced
 powers.

 \begin{definition}
 \label{Def-Filt}
 
 Let $ A $ be a set. A {\em filter} $ {\cal F} $ on $ A $ is
 a set of subsets of $ A $, $ {\cal F} \subseteq 2^A $,
 such that
\begin{itemize}
	
	\item[(1)] $ {\cal F} $ is closed under finite
	intersections, i.e.~$ A \in {\cal F} $ (empty
	intersection) and $ A_1, A_2 \in {\cal F}
	\Rightarrow A_1 \cap A_2 \in {\cal F} $.
	
	\item[(2)] $ {\cal F} $ is up-closed $ A_1 \in
	{\cal F} \ \&\ A_1 \subseteq A_2 \Rightarrow
	A_2 \in {\cal F} $.

\end{itemize}
\( {\cal F} \) is {\em proper} if \( \emptyset \notin {\cal F} \).
 
 \end{definition}
 
\begin{examples}
 
	\begin{itemize}
 
 		\item[(1)] The set of cofinite (finite complement)
		subsets of any infinite set, e.g.~$ {\mathbb N} $.
	
		\item[(2)] The set of subsets \( A \subseteq {\mathbb N} \)
		such that there exists \( n_0 \in A \) with \( k \geq n_0
		\Rightarrow k \in A \).
		
		\item[(3)] If $ A_0 \subseteq A $ is a non-empty
		subset, then the set of all $ A_1 $ containing $ A_0 $
		 is a {\em principal filter}.
 
 \end{itemize}
 
\end{examples}

\begin{definition}
\label{Def-}

Let \( A \) be a set and \( {\cal F} \) a filter on
\( A \). The {\em reduced power} \( X^{\cal F} \)
is the colimit
\[
X^{\cal F} = \limr_{B \in {\cal F}} X^B .
\]

\end{definition}

\noindent As \( X^{\cal F} \) is a filtered colimit
of representables we immediately get the following.

\begin{proposition}

For any filter \( {\cal F} \), the reduced power
\( X^{\cal F} \) gives a left exact endofunctor of
\( {\bf Set} \), in particular it is taut.

\end{proposition}

Just from the definition as a colimit we see that
\( X^{\cal F} \) consists of equivalence classes of
partial functions from \( A \) into \( X \),
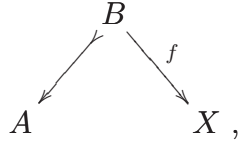
\[
\bfig
\Atriangle/ >->`>`/<300,350>[B`A`X \rlap{\ ,};`f`]

\efig
\]
with \( B \in {\cal F} \). The equivalence relation is
generated by restriction
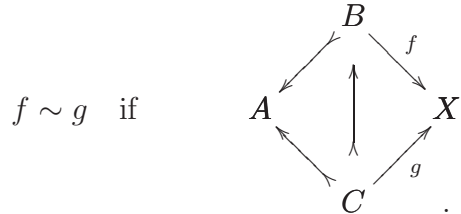
\[
\bfig
\place(0,0)[f \sim g\quad \mbox{if}]
\Atriangle(600,0)/ >->`>`/<300,300>[B`A`X;`f`]

\Vtriangle(600,-300)/`<-< `<-/<300,300>[A`X`C;``g]

\morphism(900,-100)/>->/<0,300>[`;]

\place(1200,-300)[.]

\efig
\]

Because the colimit is filtered, the equivalence
relation can be described as
\[
f \sim g \Leftrightarrow \mbox{\ \  there exists \ \ }
A \to/<-< /<200> D \to^h X \mbox{\ \ such that \ }
D \in {\cal F} \mbox{\quad and}
\]
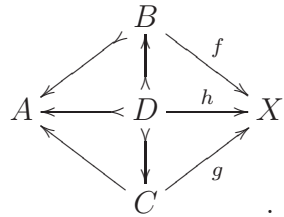
\[
\bfig
\node a(400,0)[C]
\node b(0,300)[A]
\node c(400,300)[D]
\node d(800,300)[X]
\node e(400,600)[B]

\arrow/>/[a`b;]
\arrow|r|/>/[a`d;g]
\arrow/<-< /[b`c;]
\arrow|a|/>/[c`d;h]
\arrow/<-< /[b`e;]
\arrow|r|/>/[e`d;f]
\arrow/ >->/[c`e;]
\arrow/ >->/[c`a;]

\place(800,0)[.]

\efig
\]
An even more amenable description is that
\[
f \sim g\quad  \Leftrightarrow \quad \{ a \in B \cap C \ |\  f a = g a \} \in {\cal F} \rlap{.}
\]

The reduced power \( X^{\cal F} \) is often defined as follows:
\[
X^{\cal F} = \{ f \colon A \to X \}/\sim = X^A/\sim
\]
where \( f \sim g \) iff \( \{ a \in A \ |\ f a = ga \} \in {\cal F} \).
This is a lot easier to handle and justifies the term
``reduced power''. It is equivalent to the colimit
definition except when \( {\cal F} \) is trivial, i.e.~contains
\( \emptyset \), and so all subsets. In that case, the colimit
definition gives
\[
X^{\cal F} = 1
\]
the constant functor with value \( 1 \), whereas the
quotient of the representable gives
\[
X^A/\sim \  = \left\{ \begin{array}{ll}
                    \emptyset & \mbox{if \(X = 0\)}\\
                    1 & \mbox{otherwise.}
                    \end{array}
                    \right.
\]
This functor is not taut as, e.g.
\[
\bfig
\square/ >->` >->` >->` >->/<400,400>[0`1`1`2;``1`0]

\place(200,200)[\framebox{\scriptsize Pb}]

\place(700,200)[\longmapsto]

\square(1000,0)<400,400>[0`1`1`1\rlap{\ .};```]

\efig
\]

We don't want to exclude the improper filter
in our definition but we nevertheless work with
the reduced power definition, which we find
easier, and check the degenerate case
separately.

If $ {\cal F} $ is a principal filter generated by
 $ A_0 \subseteq A $, then it is easily seen that
 $ X^{\cal F} \cong X^{A_0} $. But we have:
 
 \begin{proposition}
 
 If $ {\cal F} $ is a non-principal filter, then the
 reduced power functor $ X^{\cal F} $ is not a
 generalized symmetric power series.
  
 \end{proposition}
 
 \begin{proof}
 Suppose
 $
 X^{\cal F} \cong \sum_{i\,\in\,I} X^{B_i}/G_i 
 $.
 If we let $ X = 1 $, we get $ 1 \cong I $, so there's
 only one term in the series $ X^B/G $ where $ B $
 is a left $ C $-set, $ X^{\cal F} \cong X^B/G $.
 Now, $ X^{\cal F} $ preserves finite products, so
 $ X^B/G $ will also. Consider the special case
 \[
 (B \times B)^B/G \cong B^B/G \times B^B/ G .
 \]

 An element on the left is an equivalence class
 of pairs $ [\langle \phi , \psi \rangle] $ for
 $ \phi , \psi \colon B \to B $, and two pairs are
 equivalent
 \begin{equation}\tag{1}
 \langle \phi, \psi \rangle \sim \langle \phi', \psi' \rangle
 \Leftrightarrow \exists \ \sigma \in G
 (\sigma \phi = \phi' \wedge \sigma \psi = \psi') 
 \rlap{\ .} 
 \end{equation}
 (Wlog we have assumed that $ G $ is a subgroup
 of $ S_B $.)
 
 An element on the right is a pair of equivalence
 classes $ \langle [\phi], [\psi] \rangle $ with two
 being equal
 \begin{equation}\tag{2}
 \langle [\phi], [\psi] \rangle = \langle [\phi'], [\psi] \rangle
 \Leftrightarrow \exists \  \sigma, \tau \in G
 (\sigma \phi= \phi' \wedge \tau \psi = \psi)
 \rlap{\ .} 
 \end{equation}
 
 Of course (1) $ \Rightarrow $ (2). That's the
 canonical product comparison morphism. If (2)
$  \Rightarrow $ (1) then for any $ \sigma \in G $,
we have $ \langle [\id], [\id] \rangle =
\langle [\sigma], [\id] \rangle $ so
$ [\langle \id, \id \rangle ] = [ \langle \sigma,
\id \rangle ] $ i.e.~there is $ \sigma' \in G $ such
that
$$
\sigma' \cdot \id = \sigma \mbox{\quad and\quad}
\sigma' \cdot \id = \id
$$
so $ \sigma = \id $, i.e.~$ G $ is trivial.

So now we're reduced to $ X^{\cal F} = X^B $,
a representable. This means that $ X^{\cal F} $
preserves all products, which implies that
$ {\cal F} $ is closed under arbitrary intersections
and that means $ {\cal F} $ is principal, contrary
to our assumption.
\end{proof}

Non trivial reduced powers are quotients of
representables which are projective, so if \( {\cal F} \)
is a filter on \( A \) and \( {\cal G} \) one on \( B \),
a natural transformation
\( t \colon X^{\cal F} \to X^{\cal G} \) lifts
\[
\bfig
\square/-->`->>`->>`>/[X^A`X^B`X^{\cal F}`X^{\cal G};
\ov{t}```t]

\efig
\]
which is equivalent to a function \( \phi \colon B \to A \).
Not every \( \phi \) will give a natural transformation
which descends to the reduced powers, and two
\( \phi \)'s that do, may give the same reduced
transformation. For an element of \( X^{\cal F} \),
i.e.~an equivalence class \( [f] \) of functions
\( f \colon A \to X \), the induced \( t \) is given by
\[
t (X) [f] = [f \phi] \ .
\]

Some basic set theoretical calculations lead to
the following result, from \cite{Trn71A} Proposition~VI.3
(also contained in Theorem 3 of \cite{Bla76}).

\begin{proposition}
\label{Prop-RedTransf}

For filters \( {\cal F} \) and \( {\cal G} \) on \( A \)
and \( B \) respectively, the natural transformations
between the reduced power functors
\[
t \colon X^{\cal F} \to X^{\cal G}
\]
are in bijection with equivalence classes of
functions \( \phi \colon B \to A \) such that for
each \( A_0 \in {\cal F} \) we have
\( \phi^{-1} A_0 \in {\cal G} \). \( \phi \) is
equivalent to \( \psi \colon B \to A \) if
\[
\{ b \in B \ |\ \phi (b) = \psi (b) \} \in {\cal G}\ .
\]

\end{proposition}

We think of the elements of \( {\cal F} \) (or \( {\cal G} \))
as large subsets of \( A \) (resp.~\( B \)). Then natural
transformations between the reduced powers
correspond to equivalence classes of functions
which map large subsets of \( A \) to large subsets
of \( B \) by inverse image, much like continuous
functions do for opens. We now characterize those
classes \( [\phi] \) for which the corresponding
transformation is taut as those that map large subsets
of \( B \) to large subsets of \( A \) by direct image,
like open maps. First a lemma.

\begin{lemma}
\label{Lem-RedPowMonos}

Let \( z \colon Y \to/ >->/ X \) be a monomorphism.
An element \( [f] \) of \( X^{\cal F} \) is in (the image
of) \( Y^{\cal F} \) if and only if there is an
\( A_0 \in {\cal F} \) and a factorization
\[
\bfig
\square/>`<-< `<-< `-->/[A`X`A_0`Y\rlap{\ .};f``z`g]

\efig
\]

\end{lemma}

\begin{proof}

If \( Y = \emptyset \) then \( Y^{\cal F} = \emptyset \) and
the result is trivial. Otherwise, if there are \( A_0 \) and
\( g \) as above, we can extend \( g \) to \( \ov{g} \colon
A \to Y \). Then \( z \ov{g} \) and \( f \) agree on \( A_0 \)
so \( [f] = [z \ov{g}] \), i.e.~\( [f] \) in \( Y^{\cal F} \).

In the other direction, if \( [f] \) is in \( Y^{\cal F} \),
i.e.~\( [f] = [z \ov{g}] \) for some \( \ov{g} \colon A \to Y \),
then \( f \) and \( z \ov{g} \) agree on some
\( A_0 \in {\cal F} \) and we take \( g = \ov{g} |_{A_0} \).
\end{proof}

\begin{proposition}
\label{Prop-RedPowTautTransf}

Let \( {\cal F} \) and \( {\cal G} \) be filters on \( A \) and \( B \)
respectively, and \( \phi \colon B \to A \) a function
mapping elements of \( {\cal G} \) to elements of
\( {\cal F} \) by inverse image. Then the induced
natural transformation
\[
t \colon X^{\cal F} \to X^{\cal G}
\]
\[
[f] \longmapsto [f \phi]
\]
is taut if and only if
\[
B_0 \in {\cal G} \Rightarrow \phi (B_0) \in {\cal F} .
\]
 
\end{proposition}

\begin{proof}
First assume \( t \) is taut and let \( B_0 \subseteq B \)
Then we get \( \phi (B_0) \subseteq A \) and a
 pullback square, by tautness
\[
\bfig
\square/ >->`>`>` >->/[\phi (B_0)^{\cal F}`A^{\cal F}
`\phi (B_0)^{\cal G}`A^{\cal G}\rlap{\ .};
`t(\phi (B_0))`t A`]

\place(250,250)[\framebox{\scriptsize Pb}]

\efig
\]
Consider the class \( [1_A] \in A^{\cal F} \). It gets
sent to \( [\phi] \) in \( A^{\cal G} \) by \( t (A) \)
and \( [\phi] \) is in \( \phi (B_0)^{\cal G} \)
(as \( \phi \) factors through \( \phi (B_0) \)). So
\( [1_A] \) is in \( \phi(B_0)^{\cal F} \) and thus
there are \( A_0 \in {\cal F} \) and a restriction
\[
\bfig
\square/>`<-< `<-< `-->/[A`A`A_0`\phi (B_0)\rlap{\ .};
1_A```]

\efig
\]
So \( A_0 \subseteq \phi (B_0) \) and we get
\( \phi (B_0) \in {\cal F} \).

Now assume that \( \phi \) preserves large
sets (i.e., elements of the filter) and consider a mono \( Y\  \to/>->/ X \).
We wish to show that
\[
\bfig
\square/ >->`>`>` >->/[Y^{\cal F}`X^{\cal F}`Y^{\cal G}`X^{\cal G};
`t Y`t X`]

\efig
\]
is a pullback, so let \( [f] \in X^{\cal F} \) be such that
\( [f \phi] = t X [f] \in Y^{\cal G} \). By the lemma, there
are \( B_0 \in {\cal G} \) and a restriction \( g \)
\[
\bfig
\square/`<-< `<-< `{-->}/<1000,500>[B`X`B_0`Y;```g]

\morphism(0,500)/>/[B`A;\phi]

\morphism(500,500)/>/[A`X;f]

\efig
\]
and \( g \) factors through the image \( \phi (B_0) \)
\[
\bfig
\square/>`<-< `<-< `->>/[B`A`B_0`\phi(B_0);
f```]

\square(500,0)/>``<-< `>/[A`X`\phi(B_0)`Y\rlap{\ .};
f```]

\efig
\]
Because \( \phi (B_0) \in {\cal F} \) we get
\( [f] \in Y^{\cal F} \). This gives our pullback and
\( t \) is taut.
\end{proof}

There is a category of filters \( {\mathbb F} \)
introduced in \cite{Trn71A} and further studied in
\cite{Bla77B}, whose objects are pairs \( (A, {\cal F}) \)
with \( A \) a set and \( {\cal F} \) a filter on it and
whose morphisms are equivalences of functions
as in Proposition~\ref{Prop-RedTransf}. It is proved
in \cite{Trn71A} that the epimorphisms in \( {\mathbb F} \)
are precisely the \( [\phi] \) when \( \phi \) satisfies
the conditions of \ref{Prop-RedPowTautTransf}.

\subsection{Monads}
\label{SSec-Mon}

Another instance where endofunctors appear
naturally in category theory is in the theory of
monads and, in fact, that's what Manes was
studying when he introduced taut functors
\cite{Man02}. His interests lay in applications
to computer science (collection monads, e.g.)
and categorical topology ($ T_0 $ spaces, e.g.),
so many of his examples centred around lists,
e.g.~the free monoid monad, and around filters.
Among his examples were indeed the free
monoid or semigroup monads and the filter
monad $ {\mathbb F} $. $ {\mathbb F} (X) $
is the set of filters on $ X $ well-known to be
a monad. He showed that not only is $ {\mathbb F} $
taut (functor, unit and multiplication are taut) but
a monad $ {\mathbb T} $ is taut iff it admits a
taut morphism of monads $ {\mathbb T} \to
{\mathbb F} $ \cite{Man02}. He
also showed that any submonad of $ {\mathbb F} $
is taut (Theorem 3.12), so that the ultrafilter monad
$ \beta $ is taut. See {\em loc.~cit.} for more examples.

As just mentioned the free monoid monad is taut.
In fact it's a polynomial monad
$$
T X = 1 + X + X^2 + \dots
$$
The free commutative monoid monad is also taut
as it is a divided powers monad
$$
T X = 1 + X + X^2/S_2 + X^3/S_3 + \dots
$$
On the other hand the free group monad or the free
Abelian group monad are not taut. Consider, e.g., 
the free Abelian group monad and the following
inverse image diagram in $ {\bf Set} $
$$
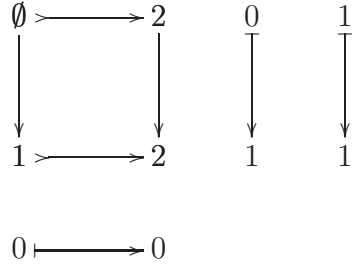
\bfig
\morphism(0,0)/|->/<450,0>[0`0;]
\square(0,300)/ >->`>`>` >->/<450,450>[\emptyset`2`1`2;```]

\morphism(750,750)/|->/<0,-450>[0`1;]
\morphism(1050,750)/|->/<0,-450>[1`1;]

\place(1150,0)[.]

\efig
$$
If we apply $ T $ we get
$$
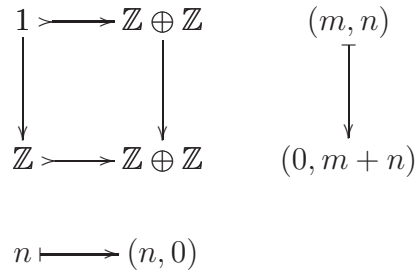
\bfig
\morphism(0,0)/|->/<450,0>[n`(n, 0);]
\square(0,300)/ >->`>`>` >->/<450,450>[1`{\mathbb Z} \oplus {\mathbb Z}
`{\mathbb Z}`{\mathbb Z} \oplus {\mathbb Z};```]

\morphism(1050,750)/|->/<0,-450>[(m, n)`(0, m+n);]

\efig
$$
which is not a pullback. The pullback is
$ \{(m, n) | m + n = 0 \} $. The difference between monoids
and groups is that in groups there are equations with
 different variables on either side, e.g.~$ x x^{-1} = 1 $.

In 1967, P{\l}onka \cite{Plo67} studied universal algebras
defined by what he termed regular equations, equations
that have the same variables on both sides. Szawiel and
Zawadowski \cite{SzaZaw15} have shown that finitary monads
are taut (which they call semi-analytic) iff they can be
presented by regular equations.

Other non-finitary monads considered by Manes are
the covariant power-set monad which is taut and the
double dualization monad which is not.

We examine the tautness of the covariant power-set
functor $ P $ as it comes up in the next section. Consider
the inverse image diagram in $ {\bf Set} $
$$
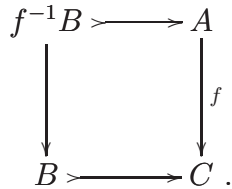
\bfig
\square/ >->`>`>` >->/[f^{-1} B`A`B`C\rlap{\ .};
``f`]

\efig
$$
The pullback of $ P B $ along $ P f $
$$
\bfig
\square/ >->`>`>` >->/<600,500>[(Pf)^{-1}(PB)`PA`PB`PC;
``Pf`]

\efig
$$
consists of the set of all subsets of $ A $,
$ A_0 \subseteq A $ such that $ f (A_0) \subseteq B $
which is equivalent to $ A_0\subseteq  f^{-1} B $, 
i.e.~$ (Pf)^{-1} (PB) = P(f^{-1} B) $, which means that
$ P $ is taut.

$ P $ doesn't preserve all pullbacks though. Just by
cardinality arguments we see that for cardinals
 $ m, n > 2 $, the pullback
 $$
 \bfig
 \square[mn`n`m`1;```]
 
 \efig
 $$
 is not preserved: the cardinality of $ Pm \times_{P1} Pn $
 is less than the cardinality of $ Pm \times Pn = 2^{m + n} $
 which itself is less than $ 2^{mn} $, the cardinality of
 $ P(mn) $.

 \subsection{Dirichlet series}
 \label{SSec-Dir}
 
 Classically, Dirichlet series are series of the form
 $$
 \sum^\infty_{n = 1} \frac{c_n}{n^s}
 $$
 which could be written as
 $$
 \sum^\infty_{n = 1} c_n \left( \frac{1}{n} \right)^s
 $$
 and generalize to $ {\bf Set} $
 \begin{equation}\tag{*}
 \sum_{i\,\in\,I} C_i L^X_i 
 \end{equation}
 for sets $ I, C_i, L_i, X $. This is the definition given in
 \cite{MyeSpi20}, a coproduct of contravariant representables.
 It would be nice to get an actual endofunctor of $ {\bf Set} $
 rather than a functor $ {\bf Set}^{op} \to {\bf Set} $. We
 already have an example of this, the covariant
 powerset functor $ P X = 2^X $. We can bootstrap this
 to get other examples, $ (2^A)^X \cong 2^{A \times X} $.
 We could also take $ 3^X $ whose elements are nested
 pairs of subsets to which direct image also applies. And
 so on.
 
 What makes this work is that the base $ L $ is a
 complete lattice, a sup lattice to be precise, and
 functoriality is given by left Kan extension.

 \begin{proposition}
 
 Let $ L $ be a sup complete lattice. Then left
 Kan extension makes $ L^X $ into a taut endofunctor
 of $ {\bf Set} $.
 
 \end{proposition}

 \begin{proof}
 
 Let $ f \colon A \to B $ be a function, then
 $$
 L^f \colon L^A \to L^B
 $$
 is given by
 $$
 L^f (\phi) (b) = \bigvee_{f(a) = b}\  \phi(a) \rlap{\ ,}
 $$
 the left Kan extension of $ \phi $ along $ f $
 $$
 \bfig
 \Vtriangle<350,400>[A`B`L;f`\phi`L^f_{(f)}]
 
 \place(350,250)[\leq]
 
 \place(700,0)[.]
 
 \efig
 $$
 It is well-known, and easily seen, that this makes
 $ L^{(\ )} $ into a functor $ {\bf Set} \to {\bf Set} $.
 
 Consider the following inverse image diagram in
 $ {\bf Set} $
 $$
 \bfig
 \square/ >->`>`>` >->/[A_0`A`B_0`B;
 m``f`n]
 
 \place(250,250)[\framebox{{\scriptsize Pb}}]
 
 \efig
 $$
 $ A_0 = f^{-1} B_0 $. First note that if $ \psi_0 \in L^{B_0} $ then its
 extension to $ B $ is given by
                       
\[
\begin{array}{lcl}
L^n (\psi_0) (b) & = &\bigvee_{\phi(b_0) =b} \psi_0 (b_0)\\
                                      & = &\left\{\begin{array}{ll}
                                   b  & \mbox{if $ b \in B_0$}\\
                                        \bot & \mbox{o.w.}
                                    \end{array} \right.
   \end{array}
\]
So we can identify $ L^{B_0} $ with its image in
$ L^B $, i.e.
$$
\{ \psi \colon B \to L | \psi (b) = \bot \mbox{\ for all\ } b \notin B_0 \}.
$$

Now take the pullback   
$$
\bfig
\square/>`>`>` >->/[Pb`L^A`L^{B_0}`L^B\rlap{\ .};``L^f`]
\place(250,250)[\framebox{{\scriptsize Pb}}]
 
\efig
$$
An element of $ Pb $ is $ \phi \colon A \to L $ such
that
$$
L^f (\phi) \in L^{B_0}
$$
i.e.
$$
\bigvee_{f(a) = b} \phi (a) = \bot
$$
for all $ b \notin B_0 $. This means $ \phi (a) = \bot $
for all $ a $ with $ f(a) \notin B_0 $, i.e.~all
$ a \notin f^{-1} B_0 = A_0 $. Thus $ Pb $ is
$ L^{A_0} $ and $ L^X $ is taut.
 \end{proof}

 Other than the few examples of monads mentioned
 in Section~\ref{SSec-Mon} all of our examples had
 rank, i.e.~were a small colimit of representables.
 The functors $ L^X $ (and $ L^{[X]} $ introduced
 below) are of unbounded rank (i.e.~have no rank),
 unless $ L = 1 $.
 
 To see this, define the support of $ \phi \colon A \to L $
 to be $ \sigma (\phi) = \{a| \phi(a) \neq \bot \} $. If
 $ f \colon A \to B $ and $ \psi = L^X (f) (\phi) $, i.e.
 $$
 \psi (b) = \bigvee_{f(a)=b} \phi (a)
 $$
 then $ \sigma (\psi) $ is the image of $ \sigma (\phi) $
 under $ f $ as is easily seen. Thus the cardinality of
 $ \sigma (\psi) $ is bounded by that of $ \sigma (\phi) $.
 This means that no set of elements can generate $ L^X $
 because any element with support bigger than the
 supports of all the supposed generators could never
 be attained.
 
 Examples of sup-complete lattices which may be of
 interest are the open sets of a topological space, the
 subgroups of a group and of special interest to us
 here, are the finite ordinals.
 
 The functor $ L^X $ is not connected. By
 Proposition~\ref{Prop-SumDecomp},
 $ \pi_0 (L^X) = L^1 = L $, and $ L^X $
 decomposes into a sum of connected functors,
 one for each $ l \in L $,
 $$
 L^X \cong \sum_{l\,\in\,L} F_l \rlap{\ .}
 $$
 The transformation $ L^X \to L $ is given by
 sup, so $ F_l = \{ f \colon X \to L\  |\  \bigvee f(x)
 = l \} $ which is the same as all functions into
 $ D (l) = \{ l' \in L\  |\  l' \leq l \} $, the down-set
 of $ l $, whose sup is the top element of
 $ D (l) $, i.e.~$ l $ itself. This leads to the
 following:

 \begin{definition}
 \label{Def-NormExp}
 
 Let $ L $ be a sup lattice. The {\em normalized
 exponential functor} with base $ L $ is
 $$
 L^{[X]} = \{ f \colon X \to L \ | \ \bigvee f(x) = \top \}
 $$
 where $ \top $ is the top element of $ L $.
 
 \end{definition}

 \begin{proposition}
 \label{Prop-NormExp}
 
 \begin{itemize}
 
 	\item[(1)] The normalized exponential is
	connected and taut.
	
	\item[(2)] \( L^X \cong \sum_{l\,\in\, L} D(l)^{{[}X{]}} \rlap{\ .} \)
		
	\item[(3)] If \( L_1, L_2 \) are sup lattices, then
	\[
	L_1^{[X]} \times L_2^{[X]} \cong (L_1 \times L_2)^{[X]} \rlap{\ .}
	\]
 
 \end{itemize}
 
 \end{proposition}
 
 \begin{proof}
 
 (1) and (2) are immediate by Corollary~\ref{Cor-SumDecomp}.
 (3) is simply that ``sup'' and ``top'' are component-wise in
 $ L_1 \times L_2 $.
 \end{proof}
 
 \begin{example}
 \label{Ex-RedPow}
 
 Let \( n = \{ 0, 1, \dots, n - 1 \} \) represent the \( n^{th} \)
 cardinal number and \( {\bf n} \) the corresponding
 ordinal \( 0 < 1 < \dots < n - 1 \). \( {\bf n} \) is
 a complete lattice and its down sets \( D(l) \) from above
 are \( {\bf 1}, {\bf 2}, \dots, {\bf n} \) so
 \[
 {\bf n}^X \cong {\bf 1}^{[X]} + {\bf 2}^{[X]} + \dots + {\bf n}^{[X]} \rlap{\ .}
 \]
 
 \end{example}
 
 The reason we are interested in connected functors
 is that it allows for a simple analysis of natural
 transformations between sums of them. If
 \( \langle F_i \rangle_{i \in I} \) and \( \langle G_j \rangle_{j \in J} \)
 are families of connected endofunctors of \( {\bf Set} \),
 then a natural transformation
 \[
 t \colon \sum_{i \in I} F_i \to \sum_{j \in J} G_j
 \]
 is given by a function \( \alpha \colon I \to J \)
 and a family of natural transformations
 \[
 \langle t_i \colon  F_i \to G_{\alpha(i)} \rangle_{i \in I} .
 \]
 So, if \( t \) were an isomorphism then so would
 \( \alpha \) and all the \( t_i \).
 
 If we were to define Dirichlet series as a coproduct
 \[
 \sum_{i \in I} L^X_i
 \]
 as suggested above, there could be different
 families \( \langle L_i \rangle \) giving isomorphic
 functors, which is not desirable. For example, if
 \( \langle n_i \rangle \) is an unbounded sequence
 of natural numbers, then the functor
 \[
 \sum_{i \in {\mathbb N}} {\bf n}^X_i
 \]
 will have infinitely many \( {\bf k}^{[X]} \) summads
 for each \( k \) and thus
 \[
 \sum_{i \in {\mathbb N}} {\bf n}^X_i \cong
 {\mathbb N} \times \sum_{k \in {\mathbb N}}
 {\bf k}^{[X]} .
 \]
 So any two unbounded sequences give
 isomorphic functors.
 
 This motivates the following.

 \begin{definition}
 \label{Def-DirFun}
 
 A {\em Dirichlet functor} is a coproduct of
 {\em normalized exponentials}
 $$
 F X = \sum_{i\,\in\, I} L_i^{[X]} \rlap{\ .}
 $$
 
 \end{definition}
 
 So, in particular, by (2) above, $ L^X $ defines a Dirichlet functor.
 
 For future reference we record the following:

 \begin{proposition}
 \label{Prop-DirTaut}
  
  Dirichlet functors are taut.
 
 \end{proposition}
 
 Any natural transformation of Dirichlet functors
 \[
 t \colon \sum_{i \in I} L^{[X]}_i \to
 \sum_{j \in J} M^{[X]}_j
 \]
 must, because the \( L^{[X]}_i \) are connected,
 come from a function \( \alpha \colon I \to J \)
 and a family of natural transformations
 \[
 t_i \colon L^{[X]}_i \to M^{[X]}_{\alpha (i)} \rlap{\ .}
 \]
 In particular if
 \[
 \sum_{i \in I} L^{[X]}_i \cong
 \sum_{j \in J} M^{[X]}_j
 \]
 then there is a bijection \( \alpha \colon I \to J \)
 and isomorphisms
 \[
 L^{[X]}_i \cong M^{[X]}_{\alpha (i)} \rlap{\ .}
 \]
 This raises the question of whether the lattices
 \( L_i \) and \( M_{\alpha (i)} \) are themselves
 isomorphic. This is indeed the case but the proof
 is not that simple. In fact a similar statement for
 the non-normalized powers \( L^X \) is false.
 
 \begin{proposition}
 \label{Prop-TopPresFns}
 
 Any top preserving morphism of sup-lattices
 \( \phi \colon L \to M \) induces, by composition,
 a natural transformation
 \[
 t_\phi \colon L^{[X]} \to M^{[X]} \rlap{\ .}
 \]
 \( t_\phi \) is taut if and only if \( \phi \) reflects
 ``bottom''
 \[
 \phi(a) = \bot \Rightarrow a = \bot \rlap{\ .}
 \]
 
 \end{proposition}

 \begin{proof}
 
 For \(\lambda \colon X \to L\) in \(L^{[X]}\),
 \(t_\phi (\lambda)\) is  \(\phi \lambda \colon
 X \to M \). \[ \bigvee_{x \in X} \phi \lambda (x) =
 \phi (\bigvee_{x \in X} \lambda (x)) =
 \phi (\top) = \top\ , \] so \( \phi \lambda \) is in
 \( M^{[X]} \).
 
 For naturality, let \( f \colon Y \to X \) be any
 function. Then for \( \lambda \in L^{[Y]} \) and
 \( x \in X \), going across the top and then down
 in
 \[
 \bfig
 \square[L^{[Y]}`L^{[X]}`M^{[Y]}`M^{[X]};
 L^{[f]}`t_\phi (Y)`t_\phi (X)`M^{[f]}]
 
 \efig
 \]
 gives
 \[
 t_\phi (X) L^{[f]} (\lambda)(x) =
 \phi (L^{[f]} (\lambda)(x)) =
 \phi (\bigvee_{fy = x} \lambda y)
 \]
 and first going down and then over
 \[
 M^{[f]} t_\phi (Y) (\lambda) (x) =
 M^{[f]} (\phi \lambda) (x) =
 \bigvee_{f y = x} \phi \lambda (y)
 \]
 which are equal because \( \phi \) preserves sup.
 
 Now assume that \( \phi \) reflects \( \bot \) and let
 \( f \) be monic.
 We can identify \( L^{[Y]} \) with its image in
 \( L^{[X]} \), which consists of all \( \lambda \colon
 X \to L \) such that \( \lambda (x) = \bot \) for
 \( x \notin Y \). And similarly for \( M^{[Y]} \).
 
 If \( \lambda \in L^{[X]} \) is such that
 \( t_\phi (X) (\lambda) \) is in \( M^{[Y]} \), then
 for \( x \notin Y \)
 \[
 \phi \lambda (x) = \bot
 \]
 so
 \[
 \lambda (x) = 1
 \]
 and thus \( \lambda \) is in \( L^{[Y]} \). The square
 is a pullback and \( t_\phi \) is taut.
 
 Conversely, if \( t_\phi \) is taut, consider the mono
 \( 0 \colon 1\ \to/>->/ 2 \), giving a pullback
 \[
 \bfig
 \square/ >->`>`>` >->/[L^{[1]}`L^{[2]}`M^{[1]}`M^{[2]}\rlap{\ .};
 `t_\phi (1)`t_\phi (2)`]
 
 \place(250,250)[\framebox{\scriptsize Pb}]
 
 \efig
 \]
 The image of \( L^{[1]} \ \to/>->/ L^{[2]} \) is the singleton
 \( (\top, \bot) \) and similarly for \( M^{[1]} \ \to/>->/ M^{[2]} \).
 So if \( \phi (a) = \bot \), then \( t_\phi (\top, a) = (\top, \bot) \in
 M^{[1]} \) and then \( (\top, a) \in L^{[1]} \), i.e.~\( a = \bot \).
\end{proof}
 
 \begin{theorem}
 \label{Thm-RepOfTransf}
 
 Every natural transformation \( t \colon L^{[X]}
 \to M^{[X]} \) is of the form \( t_\phi \) for a unique
 top preserving sup-map
 \[
 \phi \colon L \to M \rlap{\ .}
 \]
 
 \end{theorem}

 \begin{proof}
 
 Consider a natural transformation
 \( t \colon L^{[X]} \to M^{[X]} \). 
 \( t(2) \colon L^{[2]} \to M^{[2]} \) takes pairs
 \( (a, b) \) such that \( a \vee b = \top \) to pairs
 \( (f (a, b), f' (a, b)) \) and naturality of \( t \) for
 the switch map \( \sigma \colon 2 \to 2 \) shows
 that \( f' (a, b) = f (b, a) \). So
 \[
 t (2) (a, b) = (f (a, b), f (b, a))
 \]
 for some function \( f \colon L^{[2]} \to M \)
 with \( f (a, b) \vee f (b, a) = \top \).
 
 We'll show that \( \phi = f (-, \top) \colon
 L \to M \) is a top preserving sup-map and
 that \( t = t_\phi \).
 
 Consider \( t (X) \colon L^{[X]} \to M^{[X]} \). For
 \( x \in X \), define \( \alpha_x \colon X \to 2 \) by
 \begin{eqnarray*}
 \alpha_x (x)  & = & 0\\
 \alpha_x (y)  & =  & 1 \mbox{\ \ for \( y \neq x . \)}
 \end{eqnarray*}
 Then \( L^{[\alpha_x]} \colon L^{[X]} \to L^{[2]} \)
 takes \( \lambda \) to \( (\lambda x, \bigvee_{y \neq x} \lambda y) \)
 (same for \( M^{[\alpha_x]} \)). Naturality with respect
 to \( \alpha_x \)
 \[
 \bfig
 \square[L^{[X]}`M^{[X]}`L^{[2]}`M^{[2]};
 t(X)`L^{[\alpha_x]}`M^{[\lambda_x]}`t(2)]
 
 \efig
 \]
 gives for \( \lambda \in L^{[X]} \),
 \begin{equation}\tag{1}
 (t (X) (\lambda) (x), \bigvee_{y \neq x} t (X) (\lambda) (y))
 = (f (\lambda x, \bigvee_{y \neq x} \lambda y),
 f(\bigvee_{y \neq x} \lambda y, \lambda x))
 \end{equation}
 From which we get
 \begin{equation}\tag{2}
 t (X) (\lambda) (x) = f (\lambda x, \bigvee_{y \neq x} \lambda y)
  \end{equation}
  so that \( t \) is completely determined by \( f \),
  which will be the uniqueness part of our bijection
  once it's established.
  
  Let \( \lambda \colon X \to L \) be an arbitrary
  function and let \( l \in L \) be such that
  \( l \vee \bigvee_x \lambda x = \top \) so that
  \( [\lambda, l] \colon X + 1 \to L \) is in
  \( L^{[X + 1]} \).
  
  Equality of the second coordinates of (1), when  applied to
  \( [\lambda, l] \) with  \( y = 0 \), will give 
  \begin{equation*}
  f (\bigvee_x \lambda x, l)  =  \bigvee_x t (X + 1)[\lambda, l] (x)
  \end{equation*}
  and using (2), we get
  
  \begin{equation}\tag{3}
  f (\bigvee_x \lambda x, l) = \bigvee_x f (\lambda x, l \vee
  \bigvee_{y \neq x} \lambda y)
  \end{equation}
  If we take \( l = \top \), then we get
  \[
  f (\bigvee_x \lambda x, \top) = \bigvee_x f (\lambda x, \top)
  \]
  so \( \phi = f (-, \top) \) preserves \( \bigvee \). We also
  have \( \phi (\top) = \top \) as
  \[
  \phi (\top) = f (\top, \top) = f (\top, \top) \vee f (\top, \top) = \top .
  \]
 It remains only to show that \( t_\phi = t \). We have
 from (2)
 \[
 t (X) (\lambda) (x) = f (\lambda x, \bigvee_{y \neq x} \lambda y)
 \]
 and
 \[
 t_\phi (X) (\lambda) (x) = \phi \lambda (x) = f (\lambda x, \top) .
 \]
 
 A special case of (3) gives for \( a \vee b \vee c = \top \)
 \[
 f (a \vee b, c) = f (a, c \vee b) \vee f (b, c \vee a) .
 \]
 If \( a \vee c = \top \), and \( b = a \) we get
 \[
 \begin{array}{ll}
 f (a, c) & =  f (a \vee a, c)  =   f(a, c \vee a) \vee f (a, c \vee a)\\
            & =  f (a, \top) \vee f (a, \top)  = f (a, \top)
 \end{array}
 \]
 so that \( f (a, c) \) is independent of \( c \) and
 we do get \( t = t_\phi \).
 \end{proof}

 \begin{corollary}
 \label{Cor-RedPowIso}
 
 If \( L^{[X]} \cong M^{[X]} \) then the lattices \( L \)
 and \( M \) are isomorphic.
 
 \end{corollary}

 \begin{proof}
 
 The bijection \( \phi \longleftrightarrow t_\phi \) is
 functorial in \( \phi \).
 \end{proof}
 
 Perhaps surprisingly, a similar result doesn't hold
 for the full powers \( L^X \). \( L^X \) can be decomposed
 into a coproduct of reduced powers
 \[
 L^X \cong \sum_{l \in L} D(l)^{[X]}
 \]
 with \( D(l) = \{ a \in L \ |\ a \leq l \} \). So a natural
 transformation
 \[
 t \colon L^X \to M^X
 \]
 is given by an arbitrary function \( \alpha \colon L \to M \)
and a family of natural transformations
\[
t_l \colon D(l)^{[X]} \to D (\alpha (l))^{[X]}
\]
which correspond to top preserving sup-maps
\[
\phi_l \colon D(l) \to D (\alpha (l)) .
\]
Thus \( L^X \cong M^X \) if and only if there is a
bijection \( \alpha \colon L \to M \) such that the
down sets \( D (l) \) and \( D (\alpha l) \) are
isomorphic. In particular, \( L = D (\top) \cong
D (\alpha \top) \) so \( L \) is isomorphic to a
sublattice of \( M \) and vice versa, and if
\( L \) and \( M \) are finite then \( L \cong M \),
but not in general. 

Let \( L \) be the closed subset of the unit interval
\( [0, 1] \) given by
\[
\{1\} \cup [1/2, 1/3] \cup \{1/4\} \cup [1/5, 1/6] \cup \dots
\cup \{0\}
\]
and \( M \) given by
\[
[1, 1/2] \cup \{1/3\} \cup [1/4, 1/5] \cup \dots
\cup \{0\} .
\]

\( L \) and \( M \) are sup complete because they are
closed and bound subsets of \( [0, 1] \). There are in
each case four types of down sets \( D (a) \).
\begin{itemize}
	\item[(1)] If \( a{\neq 0} \) comes from a singleton
	\( \{a\} \) in \( L \) or \( M \), then \( D (a) \cong L \)
	and there are countably many in each case.
	
	\item[(2)] If \( a \) is an interval \( [c, d) \) in
	\( L \) or \( M \), then \( D (a) \cong M \) and
	there are the power of the continuum of these
	for each of \( L \) and \( M \).
	
	\item[(3)]  If \( a \) is the right end point of an
	interval \( [c, a] \), then \( D (a) \cong L + \{\top\} \)
	and there are countably many of these for each of
	\( L \) and \( M \).
	
	\item[(4)] If \( a = 0 \) then \( D(a) = \{0\} \) in both cases.

\end{itemize}

We conclude that \( L^X \cong M^X \), but \( L \) is not
isomorphic to \( M \) because in \( L \) the top element,
\( 1 \), is isolated but it's not in \( M \).

 In a more speculative vein, suppose we want
 something resembling the familiar Dirichlet series,
 $$
 F X = \sum^\infty_{n = 1} C_n (1/n)^{[X]}
 $$
 we would need sup lattices $ n_* $ to play the
 role of $ 1/n $. We would like $ n_* \times m_*
 \cong (n m)_* $ so that our new Dirichlet functors
 multiply in the familiar way. We also want the
 nullary version $ 1_* \cong 1 $. This means that
 $ n_* $ is determined (up to isomorphism) by
 the prime factors of $ n $. In order for these new
 Dirichlet functors to be closed under the
 difference operator $ \Delta $, to be
 introduced in the next section, we would like
 the down sets of $ n_* $ to be of the same
 form. This more or less (though perhaps not
 quite) forces the following definition.

 \begin{definition}
 \label{Def-Star}
 
 \begin{itemize}
 
 	\item[(1)] If \( p \) is a prime number, then \( p_* \)
	is the totally ordered set
	\[
	p_* = \{ q \leq p \ |\ q \mbox{\ \ is prime} \},
	\]
	including \( p_0 = 1 \), the \( 0^{th} \) prime.
	
	\item[(2)] If $ n = p_1^{\alpha_1} p_2^{\alpha_2}
	\dots p_k^{\alpha_k} $ is the prime decomposition
	of $ n $, then $ n_* = (p_{1*})^{\alpha_1} \times
	(p_{2*})^{\alpha_2} \times \dots \times (p_{k*})^{\alpha_k} $.
 
 \end{itemize}
 
 \end{definition}
 
 So \( 1_* = \{1\}, \quad 2_* = \{1, 2\},  \quad 3_* = \{1, 2, 3\},  \quad 5_* = \{1, 2, 3, 5\}, \quad 
 7_* = \{1, 2, 3, 5, 7\}, \dots \) that is, we allocate ordinals in
 order to each prime, so we have
 \[
 1_* \cong {\bf 1},  \quad 2_* \cong {\bf 2}, \quad  3_* \cong {\bf 3}, \quad 
 5_* \cong {\bf 4},  \quad 7_* \cong {\bf 5}, \dots
 \]

 We then extend this to all $ n $ by cartesian
 product
 \[
 4_* = 2_* \times 2_*, \quad  6_* = 2_* \times 3_*,  \quad \dots \quad 
 12_* = 2_* \times 2_* \times 3_*  \quad \dots
 \]
 To be precise, we take the prime factors in
 increasing order, as illustrated above.
 
 The lattices $ n_* $ are all different, as one would
 hope.
 
 \begin{proposition}
 \label{Prop-nStarDiff}
 
 Let \( {\bf n}_1, \dots, {\bf n}_k \) and \( {\bf m}_1,
 \dots, {\bf m}_l \) be finite ordinals \( > 1 \). If the
 lattices \( \prod {\bf n}_i \) and \( \prod {\bf m}_j \)
 are isomorphic then \( k = l \) and there is a
 permutation of the subscripts \( \sigma \) such
 that \( {\bf n}_i = {\bf m}_{\sigma i} \) for all \( i \).
 
 \end{proposition}
 
 \begin{proof}
 
 Let \( \phi \colon \prod {\bf n}_i \to \prod {\bf m}_i \)
 be a lattice isomorphism. The atoms in \( \prod {\bf n}_i \)
 are \( e_i = (0, \dots 0,1, 0 \dots 0) \) with a \( 1 \) in the
 \( i^{th} \) coordinate. Similarly for the atoms
 \( e'_j  \) in \( \prod {\bf m}_j \).
 \( \phi \) preserves (and reflects) atoms so we have
 a bijection \( \sigma \) on the subscripts such that
 \( \phi ({\bf e}_i) = {\bf e'}_{\sigma i} \). In particular,
 \( k = l \).
 
 For any \( i \), the set
 \[
 A_i = \{ a \in \prod {\bf n}_i \ | \ e_j \nleq a \mbox{\ for all \( j \neq i \)} \}
 \]
 is the set \( \{ (0, 0 \dots r \dots 0) \ |\ r \in {\bf n}_i \} \)
 which is isomorphic to \( {\bf n}_i \).
 
 If \( B_j \) is the similarly defined set for \( \prod {\bf m}_j \),
 then \( \phi \) restricts to an isomorphism
 \[
 A_i \to^\cong B_{\sigma (i)}
 \]
 and so \( {\bf n}_i = {\bf m}_{\sigma i} \) for all \( i \).
 \end{proof}
 
 With these $ n_* $ we can define what we call,
 for lack of a better name, sequential Dirichlet functors.
 
 \begin{definition}
 \label{Def-SeqDir}
 
 A {\em sequential Dirichlet functor} is one of the
 form
 $$
 F X = \sum^\infty_{n = 1} C_n n^{[X]}_*
 $$
 for any arbitrary sequence of sets $ C_n, n \in
 {\mathbb N}^+ $.
 
 \end{definition}

 \begin{proposition}
 \label{Prop-SeqDir}
 
 Let $ G X = \sum^\infty_{n = 1} D_n n_*^{[X]} $
 be another sequential Dirichlet functor. Then
 $ F \times G $ is also a sequential Dirichlet functor.
 In fact we have
 $$
 (F \times G) (X) \cong \sum^\infty_{n = 1}\ 
 \sum_{rs = n} (C_r \times D_s) n_*^{[X]} \rlap{\ .}
 $$
 
 \end{proposition}

 \begin{proof}
 
 This follows simply using distributivity, the
 isomorphisms $ r^{[X]}_* \times s^{[X]}_*
 \cong (r s)^{[X]}_* $, and then collecting like
 terms.
 \end{proof}

We end this section with even more speculation.

\begin{definition}
\label{Def-Zeta}

The sequential Dirichlet series
$$
Z  (X) = \sum^\infty_{n = 1} n_*^{[X]}
$$
is called the {\em zeta functor}.

\end{definition}

The Euler product formula is
$$
\sum_{n \in {\mathbb N}} {1 \over n^s}  =
\prod_{p\ prime} 
\ \ \sum_{k \in {\mathbb N}} \ {1 \over p^{ks}}\rlap{\ .}
$$
We get a similar formula for the Zeta functor
though we have to replace the infinite product by
the colimit of its finite factors.

Let $ P $ be a finite set of primes and $ P^* $
the set of all $ n $ whose prime factors lie in $ P $.

\begin{proposition}
\label{Prop-Euler}

\begin{itemize}

	\item[(1)] For $ P $ a finite set of primes, we
	have an isomorphism
	$$
	\sum_{n \in P^*}\ n^{[X]}_* \cong
	\prod_{p \in P} \ \sum_{k \in {\mathbb N}}\ 
	p^{k [X]}_*
	$$
	
	\item[(2)] 
	\[
	Z (X) \cong \limr
	\prod_{p \in P} 
	\sum_{k \in {\mathbb N}}
	p^{k [X]}_*
	\]
	where the colimit is taken over all finite
	sets of primes \( P \).
	\end{itemize}

\end{proposition}

\begin{proof}

\begin{itemize}

	\item[(1)] An element on the left is an 
	$ n = \prod_{p \in P} p^{k_p} $ and a function
	$$
	\phi \colon X \to \prod_{p \in P} \  p^{k_p}_* 
	$$
	whose sup is the top element.
	
	An element on the right is $ P $-tuple of functions
	$$
	\phi_p \colon X \to p^{k_p}_*
	$$
	whose sup is the top element, which corresponds
	bijectively to the $ \phi $ above.
	
	\item[(2)] Take $ \limr $ of the isos in (1), and note
	that $ Z (X) $ is the colimit of the left sides.

\end{itemize}
\end{proof}

\begin{remark}
\label{Rem-Euler}

If we list the primes in increasing order
$ p_1, p_2, \dots $, as usual, we can restrict
the colimit to the final subset of initial segments
and get
$$
Z (X) \cong \limr_n \ \prod_{p\  \leq\  p_n} \
p^{k[X]}_*
$$
which some may prefer.

\end{remark}

% !TEX root = taut.tex

\section{The difference operator}
\label{Sec-DiffOp}

Given a functor $ F \colon {\bf Set} \to {\bf Set} $,
we wish to study how it grows as the input set
grows. Given a set $ A $, we want to perturb it a
bit $ A \leadsto A' $ and measure the change
$ F A \leadsto F A' $. The smallest perturbing of
$ A $ is simply adding a new element, and the
perturbation is the coproduct injection
$$
j \colon A \to A + 1 \rlap{\ .}
$$
We want to see what new elements $ F $ has
acquired in passing from $ A $ to $ A + 1 $,
i.e.~the elements in the set difference
$$
F (A + 1) \setminus  {\rm Im} F (j) \rlap{\ .}
$$
If $ A \neq 0 $ then $ j $ is a split mono so that
$ F (j) $ is also a mono and we can write, by
abuse of notation,
$$
F (A + 1) \setminus  F A \rlap{\ .}
$$
One shouldn't expect this to be functorial in
$ A $ but, perhaps somewhat surprisingly, it
is for taut functors.

\subsection{Definition and functorial properties}
\label{SSec-Def}

Let $ F \colon {\bf Set} \to {\bf Set} $ be a functor.
For any set $ A $ define
$$
\Delta{[}F{]} (A) = F (A + 1) \setminus  {\rm Im} (j)
$$
for $ j \colon A \to A + 1 $ the coproduct injection.

The $ \setminus $ is set difference, not something usually
considered by category theorists as it is not functorial.
But it is more functorial than one might think. The
following lemma will be pivotal in our discussion.

\begin{lemma}
\label{Lem-SetDiff}

Let $ f \colon A \to B $ be a function, $ A_0 \subseteq A $,
$ B_0 \subseteq B $ subsets, and assume $ f $ restricts to
$ f_0 \colon A_0 \to B_0 $
\begin{equation}\tag{*}
\bfig
\square/ >->`>`>` >->/[A_0`A`B_0`B;
`f_0`f`]

\efig
\end{equation}
Then $ f $ restricts to $ f  \setminus f_0 \colon
A \setminus A_0 \to B \setminus B_0 $
\begin{equation}\tag{**}
\bfig
\square/ >->`>`>` >->/[A\setminus A_0`A`B\setminus B_0`B;
`f\setminus f_0`f`]

\efig
\end{equation}
iff $ (*) $ is a pullback diagram, i.e.~$ A_0 = f^{-1} B_0 $. When
this is the case, $ (**) $ will also be a pullback diagram.

\end{lemma}

\begin{proof}

It is a triviality, though perhaps worth mentioning, that
a function $ f \colon A \to B $ restricts to $ A_0 \to B_0 $
iff
$$
a \in A_0 \Rightarrow f a \in B_0 \rlap{\ ,}
$$
and the resulting square is a pullback iff
$$
a \in A_0 \Leftrightarrow f a \in B_0 \rlap{\ .}
$$
This is equivalent to
$$
a \notin A_0 \Leftrightarrow f a \notin B_0
$$
whence the lemma.
\end{proof}

We will find it useful to have a name for the functor
$ X \longmapsto X + 1 $. Let's call it $ S $ for
{\em successor} 
$$
S X = X + 1 \rlap{\ .}
$$
$ S $ could also stand for {\em shift} as it will
be used for precomposing, as in the proposition
below.

\begin{proposition}
\label{Prop-DeltaTaut}

If $ F \colon {\bf Set} \to {\bf Set} $ is taut, then
$ \Delta{[}F{]} $ is a taut subfunctor of $ F S $.

\end{proposition}

\begin{proof}

Let $ f \colon A \to B $ be any function. Then, by 
tautness of $ F $
$$
\bfig
\square/ >->`>`>` >->/<650,500>[FA`F(A + 1)`FB`F(B + 1);
Fj_A`Ff`F(f + 1)`Fj_B]

\efig
$$
is a pullback, so by Lemma~\ref{Lem-SetDiff},
$ F (f + 1) $ restricts to
$$
\bfig
\square/ >->`>`>` >->/<1000,500>[F(A + 1)\setminus {\rm Im}(Fj_A)
`F(A + 1)`F (B + 1)\setminus {\rm Im} (Fj_B)`F(B + 1);
``F(f + 1)`]

\efig
$$
which makes $ \Delta{[}F{]} $ into a subfunctor of $ F S $.
Furthermore, this square is also a pullback so the
inclusion
$$
\Delta{[}F{]}\  \to/>->/ F S
$$
is taut, and by Proposition~\ref{Prop-Taut}, part (6),
$ \Delta{[}F{]} $ is also taut.
\end{proof}

\begin{corollary}
\label{Cor-Sum}

For $ F $ taut, the transformation induced by
$ F j $ and the inclusion
$$
F + \Delta{[}F{]} \to F S
$$
is an isomorphism.

\end{corollary}

\begin{definition}

$ \Delta{[}F{]} $ is called the {\em difference functor} of $ F $.

\end{definition}

\noindent {\bf Notation:} Taut functors preserve monos so
$ F (j) $ is monic, and we will identify $ F X $ with its
image in $ F (X + 1) $, so $ \Delta{[}F{]} (X) = F (X + 1)
\setminus F X $.

\begin{proposition}
\label{Prop-TautTransf}

If $ t \colon F \to G $ is a taut transformation, then
$ t S \colon F S \to G S $ restricts to a taut
transformation $ \Delta t \colon \Delta{[}F{]} \to \Delta{[}G{]} $
$$
\bfig
\square/ >->`>`>` >->/[\Delta{[}F{]}`F S`\Delta{[}G{]}`G S\rlap{\ .};
`\Delta {[}t{]}`t S`]

\efig
$$

\end{proposition}

\begin{proof}

$ t $ is taut so
$$
\bfig
\square/ >->`>`>` >->/<650,500>[F X`F (X+1)`G X
`G (X + 1);F j`t X`t(X +1)`G j]

\efig
$$
is a pullback, so by Lemma~\ref{Lem-SetDiff}, 
$ t (X + 1) $ restricts to
$$
\bfig
\square/ >->`-->`>` >->/<850,500>[F(X + 1)\setminus FX
`F(X + 1)`G(X + 1) \setminus G X`G(X + 1);
``t(X + 1)`]

\efig
$$
giving $ \Delta [t](X) $ and a pullback square. $ \Delta [t] $
is automatically natural.

For any mono $ A\ \ \to/>->/ B $ we have
$$
\bfig
\node a(0,0)[\Delta{[}G{]}A]
\node b(600,0)[\Delta{[}G{]}B]
\node c(1200,0)[G(B+1)]

\node d(2000,0)[\Delta{[}G{]}A]
\node e(2600,0)[G(A+1)]
\node f(3300,0)[G(B+1)\rlap{\ .}]

\node g(0,500)[\Delta{[}F{]}A]
\node h(600,500)[\Delta{[}F{]}B]
\node i(1200,500)[F(B+1)]

\node j(2000,500)[\Delta{[}F{]}A]
\node k(2600,500)[F(A+1)]
\node l(3300,500)[F(B+1)]

\arrow/ >->/[a`b;]
\arrow/ >->/[b`c;]
\arrow/ >->/[d`e;]
\arrow/ >->/[e`f;]

\arrow/ >->/[g`h;]
\arrow/ >->/[h`i;]
\arrow/ >->/[j`k;]
\arrow/ >->/[k`l;]

\arrow|l|/>/[g`a;\Delta {[}t{]} A]
\arrow|r|/>/[h`b;\Delta {[}t{]} B]
\arrow|r|/>/[i`c;t(B+1)]

\arrow|l|/>/[j`d;\Delta {[}t{]} A]
\arrow|r|/>/[k`e;t(A+1)]
\arrow|r|/>/[l`f;t(B+1)]

\place(1650,250)[=]

\efig
$$
The second, third and fourth squares are pullbacks so
the first square is also a pullback. Thus $ \Delta [t] $
is taut.
\end{proof}

This defines the {\em difference functor}
$$
\Delta \colon {\bf Taut} \to {\bf Taut}
$$
on the category of taut endofunctors of $ {\bf Set} $ and taut natural 
transformations.

\subsection{Commutation properties}
\label{SSec-Comm}

We give functorial analogues of the usual
properties of finite differences.

\begin{proposition}
\label{Prop-Const}

Let $ C \colon {\bf Set} \to {\bf Set} $ be the constant
functor with value $ C $, then $ \Delta [C] = 0 $.

\end{proposition}

\begin{proposition}
\label{Prop-Identity}

For $ F X = X $, the identity functor $ {\bf Set} \to {\bf Set} $,
the difference is $ \Delta [X] = 1 $.

\end{proposition}

\begin{remark}

There is a (unique) natural transformation
$ X \to 1 $ but if we take differences we get
$ 1 $ and $ 0 $ and there is no transformation
$ 1 \to 0 $. This shows that the tautness
condition in Proposition~\ref{Prop-TautTransf}
is necessary.

\end{remark}

For any functor $ F \colon {\bf Set} \to {\bf Set} $
we write  $ C F $ for the product of the constant
functor $ C $ with $ F $, which is isomorphic to the
coproduct of \( C \) copies of \( F \).

\begin{proposition}
\label{Prop-ConstMult}

$ \Delta [C F] \cong C \Delta{[}F{]} $.

\end{proposition}

\begin{proposition}
\label{Prop-Sum}

$ \Delta [F + G] \cong \Delta{[}F{]} + \Delta{[}G{]} $.

\end{proposition}

The two previous propositions, easy to prove
directly, are special cases of a much more general
result, namely that $ \Delta $ commutes confluent
colimits of taut diagrams.

\begin{theorem}
\label{Thm-ConflColim}

Let $ {\bf I} $ be a small confluent category and
$ \Gamma \colon {\bf I} \to {\bf Taut} $, then we
have an isomorphism
$$
\limr_I \Delta{[}\Gamma I{]} \cong \Delta \bigl[\limr_I \Gamma (I)\bigr]\rlap{\ .}
$$

\end{theorem}

\begin{proof}

If $ \alpha \colon I \to J $ in $ {\bf I} $, then by
assumption $ \Gamma (\alpha) $ is a taut
transformation $ \Gamma (I) \to \Gamma (J) $
and so $ \Gamma (\alpha) S $ restricts to a
taut natural transformation $ \Delta [\Gamma (\alpha)]
\colon \Delta{[}\Gamma I{]} \to \Delta{[}\Gamma J{]} $
$$
\bfig
\square/ >->`>`>` >->/<600,500>[\Delta{[}\Gamma I{]}`\Gamma (I)S
`\Delta{[}\Gamma J{]}`\Gamma (J) S\rlap{\ .};
`\Delta \Gamma (\alpha)`\Gamma(\alpha) S`]

\efig
$$
This makes $ \Delta \Gamma $ into another diagram
$ {\bf I} \to {\bf Taut} $ and we have a natural
isomorphism
$$
\Delta{[}\Gamma I{]} + \Gamma (I) \to^\cong 
\Gamma (I) \circ S \rlap{ .}
$$
Colimits commute with coproducts, and
precomposing with $ S $, so
$$
\limr_I \Delta{[}\Gamma I{]} + \limr_I \Gamma (I)
\to^\cong \bigl(\limr_I \Gamma\bigr) \circ S
$$
which by Proposition~\ref{Prop-Cancellation}
means that
$$
\limr_I \Delta{[}\Gamma I{]} \cong \Delta \bigl[\limr_I \Gamma (I)\bigr]
\rlap{\,.}
$$
\end{proof}

Proposition~\ref{Prop-ConstMult} is a special case
with $ {\bf I} $ a discrete category with $ C $ elements
and $ \Gamma \colon {\bf I} \to {\bf Taut} $ the
constant diagram with value $ F $, and
Proposition~\ref{Prop-Sum} with $ {\bf I} = 2 $.
Of course, more generally
$$
\Delta \Bigl[\sum_i F_i \Bigr] \cong \sum_i \Delta{[}F{]}_i \rlap{\ .}
$$

There is a product rule for finite differences.

\begin{theorem}
\label{Thm-ProdRule}

If $ F, G \colon {\bf Set} \to {\bf Set} $ are taut,
then
$$
\Delta [F \times G] \cong (\Delta{[}F{]} \times G) +
(F \times \Delta{[}G{]}) + (\Delta{[}F{]} \times \Delta{[}G{]})
\rlap{\ .}
$$

\end{theorem}

\begin{proof}

If we take the product of the two isomorphisms
$$
\Delta{[}F{]} + F \to^\cong FS \quad \mbox{and}\quad
\Delta{[}G{]} + G \to^\cong GS
$$
and use distributivity of product over sum we get
$$
(\Delta{[}F{]} \times \Delta{[}G{]}) + (\Delta{[}F{]} \times G) 
+ (F \times \Delta{[}G{]}) + (F \times G) \to^\cong
F S \times G S = (F \times G) S
$$
and the result follows by Lemma~\ref{Prop-Cancellation}.
\end{proof}

This is purely a set theoretical result clearly
illustrated by

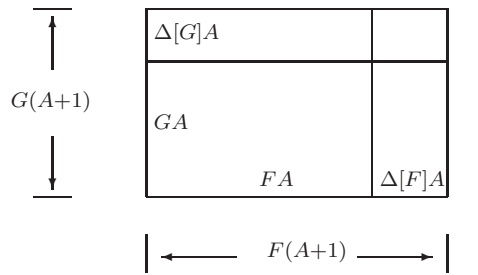
\begin{center}
 \setlength{\unitlength}{1mm}
 \begin{picture}(40,35)
\put(0,10){\line(1,0){40}}
\put(0,28){\line(1,0){40}}
\put(0,35){\line(1,0){40}}
\put(0,10){\line(0,1){25}}
\put(30,10){\line(0,1){25}}
\put(40,10){\line(0,1){25}}

\put(-15,35){\line(1,0){5}}
\put(-15,10){\line(1,0){5}}
\put(1,31){$\scriptstyle{\Delta{[}G{]} A}$}
\put(1,19){$\scriptstyle{GA}$}
\put(15,11.5){$\scriptstyle{FA}$}
\put(31,11.5){$\scriptstyle{\Delta{[}F{]}A}$}
\put(-12.5,27){\vector(0,1){7}}
\put(-12.5,18){\vector(0,-1){7}}
\put(-18,22){$\scriptstyle{G(A+1)}$}
\put(0,0){\line(0,1){5}}
\put(40,0){\line(0,1){5}}
\put(12,2){\vector(-1,0){10}}
\put(28,2){\vector(1,0){10}}
\put(16,2){$\scriptstyle{F(A+1)}$}
\put(44,0){.}

 \end{picture}
 \end{center}

 The reader can easily write down, or just imagine, the
 seven terms for $ \Delta [F \times G \times H] $ gotten
 by applying the theorem several times. In fact, easy
 set theoretical techniques will reveal the formula for
 infinite products.

\begin{theorem}
 
Given a set $ I $ and a family $ \langle F_i \rangle $
of taut functors $ {\bf Set} \to {\bf Set} $, we have
$$
\Delta \Bigl[\prod_{i\,\in\,I} F_i \Bigr] \cong
\sum_{J\subsetneq I} \Bigl( \prod_{j\in J} F_j \Bigr) \times
\Bigl( \prod_{k\notin J} \Delta{[}F_k{]} \Bigr) ,
$$ 
(the sum is taken over proper subsets $ J $ of $ I $).
 
\end{theorem}

To complete the commutativity/distributivity properties
of $ \Delta $ with limits, we have the following.

\begin{theorem}

Let $ {\bf I} $ be non-empty and connected, and
$ \Gamma \colon {\bf I} \to {\bf Taut} $ a taut diagram.
Then
$$
\Delta \bigl[ \liml_I \Gamma (I)\bigr] \cong
\liml_I \Delta{[}\Gamma I{]} \rlap{\ .}
$$

\end{theorem}

\begin{proof}

Because $ \Gamma $ takes its values in $ {\bf Taut} $,
we get a diagram $ \Delta \Gamma \colon {\bf I} \to
{\bf Taut} $ such that
$$
\Delta{[}\Gamma I{]} + \Gamma (I) \to^\cong
\Gamma (I) \circ S \rlap{\ ,}
$$
just like in the proof of Theorem~\ref{Thm-ConflColim}.
In $ {\bf Set} $, non-empty connected limits commute
with coproducts, so
$$
\liml_I \Delta{[}\Gamma I{]} + \liml_I \Gamma (I) \to^\cong
\liml \Gamma (I) \circ S
$$
and the result follows by Proposition~\ref{Prop-Cancellation}.
\end{proof}

Although arbitrary limits of taut functors are taut,
for the theorem it is necessary that the transition
morphisms $ \Gamma (\alpha) \colon
\Gamma (I) \to \Gamma (J) $ be taut. One sees
the problem immediately when attempting to
apply $ \Delta $ to the pullback
$$
\bfig
\square[F \times G`G`F`1\rlap{\ .};```]

\efig
$$

\subsection{The lax chain rule}
\label{SSec-ChainRule}

Generally speaking there is no good chain rule
for finite differences, notwithstanding the work
of Alvarez-Picallo and Pacaud-Lemay \cite{AlvLem20}, which deals with a
different situation. Ideally, we would have
$$
\Delta{[}G\circ F{]} \cong (\Delta{[}G{]} \circ F) \times \Delta{[}F{]}
$$
but this fails even for such simple functors as
$ F (X) = G (X) = X^2 $. Indeed, an easy calculation
shows that in this case
$$
\Delta{[}G\circ F{]} (X) \cong 4 X^3 + 6 X^2 + 4 X + 1
$$
whereas
$$
(\Delta{[}G{]} \circ F (X)) \times \Delta{[}F{]} (X) \cong
4 X^3 + 2 X^2 + 2 X + 1 \rlap{\ .}
$$

However, for functors, there is a lot of extra room
to maneuver and we get a comparison morphism,
which will be an isomorphism only in the simplest
of cases as the above example shows, but with
good properties nonetheless.

We will sometimes write \( \circ \) for composition
of functors, where we think it makes things clearer.

\begin{theorem}
\label{Thm-ChainRule}

For taut functors $ F, G \colon {\bf Set} \to {\bf Set} $
we have a natural comparison
$$
\gamma \colon (\Delta{[}G{]} \circ F) \times \Delta{[}F{]} \to
\Delta{[}G\circ F{]}
$$
the {\em chain rule transformation}. $ \gamma $ is
taut and monic.

\end{theorem}

\begin{proof}

Let $ A $ be a set and take an element
$ x \in \Delta{[}F{]} (A) $. This gives a function
$$
\phi_x = [Fj_A, x] \colon F A + 1 \to F (A + 1)\rlap{\ ,}
$$
which is $ F j_A $ on the first summand and $ x $
on the second. As $ x $ is not in the image of
$ F j_A $, $ \phi_x $ is monic and
$$
\bfig
\square/ >->`=` >->` >->/<550,500>[FA`FA+1`FA`F(A+1);
j_{FA}``\phi_x`Fj_A]

\efig
$$
is a pullback. $ G $ is taut so
$$
\bfig
\square/ >->`=` >->` >->/<750,500>[GFA`G(FA+1)
`GFA`GF(A+1);
G(j_{FA})``G(\phi_x)`GF(j_A)]

\efig
$$
is also a pullback, and by Lemma~\ref{Lem-SetDiff},
$ G (\phi_x) $ restricts to $ \gamma_x $ giving
another pullback
$$
\bfig
\square/ >->` >->`>->` >->/<800,500>[\Delta{[}G{]}(FA)
`G(FA+1)`\Delta{[}G\circ F{]}(A)`GF(A+1)\rlap{\ .};
`\gamma_x`G(\phi_x)`]

\efig
$$
If we put all these $ \gamma_x $ together by taking
the coproduct of the top arrows we get our
$ \gamma = [\gamma_x]_x $ and
$$
\bfig
\square/ >->`>`>` >->/<1200,500>[\Delta{[}G{]}(FA) \times \Delta{[}F{]}A
`G(FA+1) \times \Delta{[}F{]}A`\Delta{[}G\circ F{]}(A)
`GF(A+1)\rlap{\ ,};
`{[}\gamma_x{]}_x`{[}G(\phi_x){]}_x`]

\efig
$$
also a pullback.

We have one such $ \gamma $ for each $ A $,
so we should write $ \gamma (A) = [\gamma_x (A)]_x $.
To check naturality of $ \gamma $ it is sufficient to
check the commutativity of the naturality square
on each injection, that is that the square labelled
(?) below commutes for each $ x \in \Delta{[}F{]}(A) $:
$$
\bfig

\node a(0,0)[\Delta{[}G\circ F{]}(A)]
\node b(1200,0)[\Delta{[}G\circ F{]}(B)]
\node c(2200,0)[GF(B+1)\rlap{\ ,}]

\node d(0,550)[\Delta{[}G{]}(FA)]
\node e(1200,550)[\Delta{[}G{]}(FB)]
\node f(2200,550)[G(FB+1)]

\node g(600,250)[(?)]
\node h(1700,250)[(1)]

\arrow|b|/>/[a`b;\Delta{[}G\circ F{]}(f)]
\arrow/ >->/[b`c;]
\arrow|a|/>/[d`e;\Delta{[}G{]}(Ff)]
\arrow/ >->/[e`f;]
\arrow|l|/>/[d`a;\gamma_x(A)]
\arrow|r|/>/[e`b;\gamma_y(B)]
\arrow|r|/>/[f`c;G(\phi_y)]

\efig
$$
this for an arbitrary function $ f \colon A \to B $
and $ y = F(f + 1)(x) $. Compare this with the
following diagram
$$
\bfig

\node a(0,0)[\Delta{[}G\circ F{]}(A)]
\node b(1200,0)[GF(A+1)]
\node c(2200,0)[GF(B+1)\rlap{\ .}]

\node d(0,550)[\Delta{[}G{]}(FA)]
\node e(1200,550)[G(FA+1)]
\node f(2200,550)[G(FB+1)]

\node g(600,250)[(2)]
\node h(1700,250)[(3)]

\arrow|b|/ >->/[a`b;]
\arrow|b|/>/[b`c;GF(f+1)]
\arrow|a|/ >->/[d`e;]
\arrow|a|/>/[e`f;G(Ff+1)]
\arrow|l|/>/[d`a;\gamma_x(A)]
\arrow|r|/>/[e`b;G(\phi_x)]
\arrow|r|/>/[f`c;G(\phi_y)]

\efig
$$
The composites of the two top arrows of
each diagram are equal by definition of the
functoriality of $ \Delta{[}G{]} $, and the same
holds for the bottom arrows but for
$ \Delta{[}G\circ F{]} $. (1) and (3) commute
by definition of $ \gamma_x $ and $ \gamma_y $
respectively. As the bottom arrow of (1) is
monic, (?) will commute if (3) does.

(3) is $ G $ of the diagram
$$
\bfig
\square<800,500>[FA+1`FB+1`F(A+1)`F(B+1);
Ff+1`\phi_x`\phi_y`F(f+1)]

\place(400,250)[(4)]

\efig
$$
which on injections is
$$
\bfig
\square/>` >->` >->`>/<800,500>[FA`FB`F(A+1)`F(B+1);
Ff`Fj_A`Fj_B`F(B+1)]

\place(400,250)[(5)]

\place(1300,250)[\mbox{and}]

\square(2000,0)/=`>`>`>/<800,500>[1`1`F(A+1)`F(B+1);
`x`y`F(f+1)]

\place(2400,250)[(6)]

\efig
$$
each of which commutes, the second by definition
of $ y $. So (?) commutes establishing naturality
of $ \gamma $.

If $ f $ is monic so is $ F (f+1) $ and then (6) is
a pullback. (5) is always a pullback by tautness
so in this case (4) is a pullback, so (3) is too,
and (1) and (2) are pullbacks by definition of
$ \gamma_x $ and $ \gamma_y $, so (?) will be
a pullback, showing that $ \gamma $ is taut.

To show that $ \gamma $ is monic, first recall
that for each $ x $, $ \gamma_x $ itself is monic
(being the restriction of $ G(\phi_x) $). So it is
only necessary to show that the $ \gamma_x $
are pairwise disjoint. Let $ x \neq x' $, then
$$
\bfig
\square/ >->` >->` >->` >->/<700,500>[FA`FA+1`FA+1`F(A+1);
j_{FA}`j_{FA}`\phi_x`\phi_{x'}]

\efig
$$
is a pullback and $ G $ of it is too. So we have
the following pullbacks
$$
\bfig

\node a(0,0)[\Delta{[}G{]}(FA)]
\node b(900,0)[G(FA+1)]
\node c(1900,0)[GF(A+1)\rlap{\ .}]

\node d(0,550)[0]
\node e(900,550)[GFA]
\node f(1900,550)[G(FA+1)]

\arrow/ >->/[a`b;]
\arrow|b|/ >->/[b`c;G(\phi_{x'})]
\arrow/ >->/[d`e;]
\arrow|a|/ >->/[e`f;G(j_{FA})]

\arrow/ >->/[d`a;]
\arrow|l|/ >->/[e`b;G(j_{FA})]
\arrow|r|/ >->/[f`c;G(\phi_x)]

\efig
$$
The bottom arrow can be written as $ \gamma_{x'} $
followed by the inclusion, so if we pull back in stages
we get
$$
\bfig
\node a(0,0)[\Delta{[}G{]}(FA)]
\node b(900,0)[\Delta{[}G\circ F{]}(A)]
\node c(1900,0)[GF(A+1)]

\node d(0,550)[0]
\node e(900,550)[\Delta{[}G{]}(FA)]
\node f(1900,550)[G(FA+1)]

\arrow|b|/ >->/[a`b;\gamma_{x'}]
\arrow/ >->/[b`c;]
\arrow/ >->/[d`e;]
\arrow/ >->/[e`f;]
\arrow/ >->/[d`a;]
\arrow|r|/ >->/[e`b;\gamma_x]
\arrow|r|/ >->/[f`c;G(\phi_x)]

\efig
$$
so $ \gamma_x $ and $ \gamma_{x'} $ are disjoint,
giving the desired result, that $ \gamma $ is monic.
\end{proof}

For $ x \in \Delta{[}F{]} A $ and $ y \in \Delta{[}G{]} (FA) $, we have
$$
\gamma(y, x) = G(\phi_x) (y)
$$
which in fact is defined for all $ x \in F(A+1) $ and
$ y \in G(FA+1) $ giving an element of $ GF(A+1) $.
This ``full $ \gamma $'' is neither taut nor monic.
The above proof shows, in part, that if $ x \notin FA $
and $ y \notin GFA $, then $ \gamma (y, x) \notin GFA $,
and with these restrictions we do get tautness and
monicity.

\ 

The chain rule transformation is natural in $ F $ and $ G $.

\begin{theorem}
\label{Thm-ChainRuleNat}

Let $ F, F', G, G' $ be taut functors and $ t \colon F \to F' $
$ u \colon G \to G' $ be taut transformations, then the
following diagram commutes
$$
\bfig
\square<950,500>[(\Delta{[}G{]} \circ F) \times \Delta{[}F{]}
`\Delta{[}G\circ F{]}
`(\Delta{[}G'{]} \circ F') \times \Delta{[}F'{]}
`\Delta (G' \circ F')\rlap{\ .};
\gamma_{G, F}`(\Delta u) \circ F \times \Delta t
`\Delta (u \circ t)`\gamma_{G', F'}]

\efig
$$

\end{theorem}

\begin{proof}

In fact, for any $ A $
$$
\bfig
\square<1150,500>[G(FA+1) \times F(A+1)
`GF (A+1)`G'(F' A+1) \times F'(A+1)
`G'F'(A+1);
\gamma_{G, F}
`u \circ (tA +1) \times t(A+1)`(u \circ t)(A+1)
`\gamma_{G', F'}]

\efig
$$
commutes, {\em a fortiori} its restriction to the
diagram of the statement. The $ u \circ t $ on
the right is the horizontal composition of natural
transformations and expands to the composite
on the right below. The $ u \circ (tA + 1) $ is also a
horizontal composition and can be written as on
the left here:
$$
\bfig

\node a(0,0)[G'(F' A+1)\times F' (A+1)]
\node b(1300,0)[G'F'(A+1)\rlap{\ .}]
\node c(0,500)[G'(FA+1) \times F(A+1)]
\node d(1300,500)[G'F(A+1)]
\node e(0,1000)[G(FA+1) \times F(A+1)]
\node f(1300,1000)[GF(A+1)]

\arrow|b|/>/[a`b;\gamma_{G',F'}]
\arrow|b|/>/[c`d;\gamma_{G',F}]
\arrow|a|/>/[e`f;\gamma_{G,F}]
\arrow|l|/>/[e`c;u(FA+1) \times F(A+1)]
\arrow|l|/>/[c`a;G'(tA+1) \times t(A+1)]
\arrow|r|/>/[f`d;uF(A+1)]
\arrow|r|/>/[d`b;G't(A+1)]

\efig
$$

For $ x \in F(A+1) $, the restriction of the top
square to the $ x^{th} $ injection is
$$
\bfig
\square<850,500>[G(FA+1)`GF(A+1)
`G'(FA+1)`G'F(A+1);
G(\phi_x)`u(FA+1)`uF(A+1)`G(\phi_x)]

\efig
$$
which commutes by naturality of $ u $.

The second diagram, restricted to the $ x^{th} $
injection is
$$
\bfig
\square<850,500>[G'(FA+1)`G'F(A+1)
`G'(F'A+1)`G'F'(A+1);
G'(\phi_x)`G'(tA+1)`G't(A+1)`G'(\phi_{x'})]

\efig
$$
where $ x' = t(A+1) (x) $. This diagram is $ G' $
of
$$
\bfig
\square<850,500>[FA+1`F(A+1)`F'A+1`F'(A+1);
{[}F_{j_A}, x{]}`tA+1`t(A+1)`{[}F'_{j_A}, x'{]}]

\efig
$$
which commutes, on the first summand by
naturality of $ t $ and on the second by
definition of $ x'$.
\end{proof}

We also have the following associativity and unit
laws for $ \gamma $.

\begin{theorem}
\label{Thm-ChainRuleAssoc}

For taut functors $ F, G, H $ we have the following
commutativities:

\noindent (1)
$$
\bfig
\square<1850,600>[(\Delta{[}H{]} \circ G \circ F) \times
(\Delta{[}G{]} \circ F) \times \Delta{[}F{]}
`(\Delta{[}H{]} \circ G \circ F) \times \Delta{[}G\circ F{]}
`(\Delta{[}H \circ G{]} \circ F) \times \Delta{[}F{]}
`\Delta{[}H\circ G\circ F{]}\rlap{\ ,};
\id \times \gamma_{G,F}
`\gamma_{H,G} \circ F \times \id
`\gamma_{H, G \circ F}
`\gamma_{H \circ G, F}]

\efig
$$

\noindent (2)
$$
\bfig
\square|alra|/>`=`=`>/<1000,500>[(\Delta {[}\Id{]} \circ F) \times \Delta{[}F{]}
`\Delta {[}\Id \circ F{]}`1 \times \Delta{[}F{]}`\Delta{[}F{]}\rlap{\ ,};
\gamma_{\Id,F}```\cong]

\efig
$$

\noindent (3)
$$
\bfig
\square|alra|/>`=`=`>/<1000,500>[(\Delta{[}F{]} \circ \Id) \times \Delta {[}\Id{]}
`\Delta {[}F \circ \Id{]}`\Delta{[}F{]} \times 1`\Delta{[}F{]}\rlap{\ .};
\gamma_{F, \Id}```\cong]

\efig
$$

\end{theorem}

\begin{proof}

(1) Let $ A $ be a set. We'll show that
$$
\bfig
\square<2000,500>[H(GFA+1) \times G(FA+1) \times F(A+1)
`H(GFA+1) \times GF(A+1)
`HG(FA+1) \times F(A+1)
`HGF(A+1);
\id \times \gamma_{G,F} (A)
`\gamma_{HG} (FA) \times \id
`\gamma_{H,GF} (A)
`\gamma_{HG,F} (A)]

\efig
$$
commutes. Evaluate this diagram at an element
$ (x, y, z) $ of the domain
$$
\bfig
\square/|->`|->``|->/<1500,500>[(z,y,x)
`(x, G(\phi_x)(y)) = (z, w)
`(H(\phi_y)(z), x)
`HG(\phi_x) H(\phi_y)(z) 
\begin{array}{c}
?\\
\vspace{-23pt}\\
=
\end{array}
 H(\phi_w)(z)\rlap{\ .};
```]

\morphism(1900,470)/|->/<0,-370>[`;]

\efig
$$
So we have to show that $ HG(\phi_x) H(\phi_y) =
H(\phi_w) $ for $ w = G(\phi_x)(y) $, and it is
sufficient to show that $ G(\phi_x) \phi_y = 
\phi_w $, i.e.~that
$$
\bfig
\qtriangle<800,450>[GFA+1`G(FA+1)`GF(A+1);
\phi_y`\phi_w`G(\phi_x)]

\efig
$$
commutes. Restricting to the summands we
have
$$
\bfig
\qtriangle<900,450>[GFA`G(FA+1)`GF(A+1);
Gj_{FA}`GF(j_A)`G{[}Fj_A, x{]}]

\qtriangle(1700,0)<800,450>[1`G(FA+1)`GF(A+1);
y`w`G(\phi_x)]

\efig
$$
each of which commutes, the first by functoriality
of $ G $, the second by definition of $ w $. This
proves (1).

In (2) (and (3)) we denote the identity functor on
$ {\bf Set} $, which we have been calling $ X $,
by $ \Id $. Also $ 1 $ denotes the constant functor
with value $ 1 $, i.e.~the terminal endofunctor.
Then $ \Delta [\Id] (X) = (X+1) \setminus X = 1 $,
i.e.~$ \Delta [\Id] = 1 $. To calculate
$ \gamma_{\Id, F} $, take $ x \in F(A+1) $ and
consider $ \gamma $ on the $ x^{th} $ summand:
$$
\Id (\phi_x) \colon \Id (FA+1) \to \Id F (A+1)
$$
i.e.~just $ \phi_x \colon FA+1 \to F(A+1) $. Then
$ y \in \Delta [\Id](FA) $ must be $ * $ and
$ \phi_x (*) = x $. So
$$
\gamma_{\id,F} (*, x) = x
$$
which is what (2) is asserting.

For (3) consider
$$
\gamma \colon F(\Id A +1) \times \Id (A+1) \to
F \Id (A+1)
$$
i.e.
$$
\gamma \colon F(A+1) \times (A+1) \to F(A+1)\rlap{\ .}
$$
For $ x \in A +1 $, $ \phi_x \colon A+1 \to A+1 $ is the
identity on $ A $ but $ \phi_x (*) = x $. If we take
$ x \in \Delta [\Id] (A) $, $ x $ must be $ * $ so
$ \phi_* = 1_{A+1} \colon A+1 \to A+1 $ and
$ \gamma_* = 1_{F(A+1)} $. This, when restricted to
$ \Delta{[}F{]}(A) \times 1 $, is what (3) says.
\end{proof}

Getting a comparison
$$
\gamma \colon \Delta{[}G{]} \circ F \times \Delta{[}F{]}
\to/ >->/<250> \Delta{[}G\circ F{]}
$$
in that direction is a bit surprising. Normally we would
expect a morphism {\em into} a product and, as
$ \Delta{[}G\circ F{]} $ is a kind of cokernel, a morphism
out of it. One might think that there is a comparison
in the reverse direction for which $ \gamma $ is a
splitting. Looking more carefully we see that it seems
unlikely because we would need a natural transformation
$$
\Delta{[}G\circ F{]} \to \Delta{[}F{]} \rlap{\ .}
$$
So from an element of $ GF (X+1) $ we would need to
construct an element of $ F (X+1) $ in a natural way.
Nothing comes to mind but it's a bit difficult to pin down precisely. 

Consider the following example which illustrates well
the nature of $ \gamma $. Let $ F(X) $ be arbitrary
(taut) and $ G(X) = X^3 $. Then $ \Delta{[}G{]} (X) =
3X^2 + 3X +1 $ so that
$$
\Delta{[}G{]} (FX) \times \Delta{[}F{]}(X) =
3 \Bigl(F(X)^2 \times \Delta{[}F{]}(X)\Bigr) + 3\Bigl(F(X) \times
\Delta{[}F{]}(X)\Bigr) + \Delta{[}F{]} (X) \rlap{\ .}
$$
On the other hand $ (G \circ F) (X) = F(X) \times F(X)
\times F(X) $ so we can use the product rule
(Theorem~\ref{Thm-ProdRule}) (three times) to get
$$
\Delta{[}G\circ F{]} (X) = 3\Bigl(F (X)^2 \times \Delta{[}F{]}X\Bigr) +
3\Bigl(F(X) \times (\Delta{[}F{]}(X))^2\Bigr) + \Bigl(\Delta{[}F{]}(X)\Bigr)^3 \rlap{\ .}
$$
Then $ \gamma $ is the identity on the first three of
the seven summands, and given by diagonals on
the remaining ones. As $ \gamma $ is component-wise
on the summands, any splitting would have to be too.
For the first three summands there is only one choice, but
for summands of the form $ F(X) \times \Delta{[}F{]}(X)
\to F(X) \times \Delta{[}F{]}(X)^2 $ we have the two
projections, and for the last summand there are three.
This gives us $ 2 \cdot 2 \cdot 2 \cdot 3 = 24 $
``canonical'' splittings (for $ \gamma $ at $ F $ and
$ G $). But there may be many more depending
on $ F $. The simplest possible $ \Delta{[}F{]}(X) $
is $ X + 1 $ for $ F(X) = X^{[2]} = X^2/S_2 $. Then
the component of $ \gamma $ on the last summand
$$
\Delta{[}F{]}(X) \to (\Delta{[}F{]}(X))^3\quad  =\quad  (X + 1 \to<250>
X^3 + 3X^2 + 3X +1)
$$
$$
\begin{array}{rl}
x & \to/|->/<250> (x, x, x)\\
{*} & \to/|->/<250> * \rlap{\ .}
\end{array}
$$
Now there are only three splittings of $ X \to X^3 $,
the projections (Yoneda) and only one for $ 1 \to 1 $,
but it's arbitrary for the six terms in the middle, giving
648 splittings, by our count (could be wrong, but
there are lots of them). But none is natural in $ G $.
Any permutation of $ 3 $, $ \sigma \in S_n $ gives
an automorphism of $ G $ which percolates down
to the same permutation on the $ X^3 $ in
$ (\Delta{[}F{]} (X))^3 $, so no one projection would be
invariant. The upshot is that there is no global
splitting of $ \gamma $ natural in $ F $ and $ G $.

It will be useful in Subsection \ref{SSec-Filt} if we express the difference
operator with its lax chain rule in tangent category
terms (see \cite{CocCru14} for definitions). Our tangent
space for $ {\bf Set} $ is $ P_1 \colon {\bf Set}
\times {\bf Set} \to {\bf Set} $. Given a taut functor
$ F \colon {\bf Set} \to {\bf Set} $, we define
\( D (F) \colon {\bf Set} \times {\bf Set} \to
{\bf Set} \times {\bf Set} \) by
$$
D (F) (A, B) = (FA, \Delta{[}F{]} (A) \times B) \rlap{\ .}
$$
Note that $ D (F) $ is ``linear'' in the second variable,
i.e.~$ D (F)(A,-) $ preserves colimits, and as such is
completely determined by $ D(F)(A,1) $.

The chain rule says that $ D $ is a monoidal functor
on taut endofunctors.

\begin{theorem}
\label{Thm-DMMonoidal}

Let $ {\bf End}_{Taut} ({\bf Set}) $ and
$ {\bf End}_{Taut} ({\bf Set} \times {\bf Set}) $
be the strict monoidal categories of taut
endofunctors on $ {\bf Set} $ and $ {\bf Set} \times
{\bf Set} $ respectively, with taut natural
transformations, and tensor given by composition.
Then $ D $ defines a monoidal functor
$$
{\bf End}_{Taut} ({\bf Set}) \to {\bf End}_{Taut}
({\bf Set} \times {\bf Set}) \rlap{\ .}
$$

\end{theorem}

In the present context we might even go so far as to write
a general object of ${\bf Set} \times {\bf Set} $ as $ (A, \Delta A) $
where $ \Delta A $ is just another object, independent
of $ A $ but thought of as an increment in $ A $,
just like the $ d x $ in $ f(x) d x $. Then
$$
D(F) (A, \Delta A) = (FA, \Delta{[}F{]} (A) \times \Delta A) 
\rlap{\ .}
$$

% !TEX root = taut.tex

\section{Differences for the special classes}
\label{Sec-SpecClDiff}

We give explicit formulas for the
difference operator on the various classes of
taut functors studied in Section~\ref{Sec-SpecCl}.

\subsection{Polynomials}
\label{SSec-PolyDiff}

Let $ F(X) = X^A $ for some set $ A $. An element
$ \phi \in \Delta{[}F{]}(X) = (X+1)^A \setminus X^A $
is a function $ \phi \colon A \to X+1 $ with $ * \in {\rm Im}
(\phi) $ ($ * $ is the unique element of $ 1 $). If
$ S = \{ a \in A | \phi(a) \neq * \} $ is the support of
$ \phi $, then $ \phi $ is uniquely determined by its
restriction to $ S $, and the condition $ * \in {\rm Im}
(\phi) $ is that $ S $ is a proper subset of $ A $. This
produces a bijection
$$
\Delta [X^A] \cong \sum_{S \subsetneq A} X^S
$$
where the coproduct is taken over all proper subsets
of $ A $. The bijection is natural as
$ (f+1)^A \colon (X+1)^A \to (Y+1)^A $ preserves
support for any $ f \colon X \to Y $.

A polynomial functor is a small coproduct of
powers
$$
P(X) = \sum_{i\,\in\,I} X^{A_i}
$$
for a family $ \langle A_i \rangle_{i\,\in\,I} $ of sets
and as $ \Delta $ commutes with coproducts we have
$$
\Delta [P] (X) = \sum_{S \subsetneq A_i} X^S
$$
where the coproduct is taken over all $ i \in I $ and
proper subsets $ S \subsetneq A_i $.

If $ A $ is a finite cardinal $ n $, we can group like
powers of $ X $ together to get
$$
\Delta [X^n] = \sum_{k = 0}^{n - 1} \binom{n}{k} X^k
$$
where $ \binom{n}{k} $ is,
as usual, the binomial coefficient.

This is all we need for the following result.

\begin{proposition}
\label{Prop-PolyDiff}

\begin{itemize}

	\item[(1)] If $ P(X) $ is a polynomial functor
	then so is $ \Delta [P](X) $
	
	\item[(2)] If $ P(X) $ is a power series functor, so is
	$ \Delta [P] (X) $
	
	\item[(3)] If $ P(X) $ is a finitary polynomial functor,
	so is $ \Delta [P] (X) $.

\end{itemize}

\end{proposition}

\subsection{Divided powers}
\label{SSec-DivPowDiff}

Let $ n $ be a positive integer. Then
$ X^{[n]} = X^n/S_n $ consists of equivalence
classes of $ n $-tuples $ [x_1 \dots x_n] $ with
$ [x_1 \dots x_n] = [y_1 \dots y_n] $ iff there is
a permutation $ \sigma \in S_n $ such that
$ y_i = x_{\sigma (i)} $ for all $ i $. So
$ \Delta [X^{[n]}] $ consists of such equivalence
classes with $ x_i \in X $ or $ x_i = * $ for all
$ i $ and at least one $ x_i = * $. Within each of
these equivalence classes there are those
$ n $-tuples with all the $ * $'s at the end. If
there are $ k $ $ x $'s for $ X $ (and $ n-k $ stars)
then the equivalence relation reduces to the
existence of $ \sigma \in S_k $ with $ y_i = x_{\sigma i} $.
This way we see that
$$
\Delta [X^{[n]}] \cong X^{[0]} + X^{[1]} + \dots + X^{[n-1]} .
$$

If $ FX $ is a divided power series
$$
FX = C_0 + C_1 X^{[1]} + C_2 X^{[2]} + \dots
= \sum_{n = 0}^\infty C_n X^{[n]}
$$
then

\begin{align*}
\Delta{[}F{]}(X) &= \sum_{i = 1}^\infty C_i 
+ \left(\sum_{i = 2}^\infty C_i\right) X^{[1]} +
\left(\sum_{i = 3}^\infty C_i \right) X^{[2]} + \dots\\
&= \sum_{n = 0}^\infty \left(\sum_{i = n+1}^\infty C_i\right) X^{[n]} .
\end{align*}

\begin{proposition}

If $ F(X) $ is a divided power series then so is
$ \Delta{[}F{]} (X) $.

\end{proposition}

$ \Delta [X^{[A]}] $ for infinite sets $ A $ is more
complicated. For example, for $ A = {\mathbb N} $,
we have for each $ n $, equivalence classes of
the form $ (x_0, x_1, \dots, x_{n-1}, *, *, *, \dots) $
giving an $ X^{[n]} $ summand, but there are also
equivalence classes of the form
$$
(*** \dots * x_n x_{n+1} x_{n+2} \dots), n > 0
$$
giving countably many $ X^{[{\mathbb N}]} $
summands. And there's even one more type of
equivalence class with infinitely many $ * $'s
and infinitely many $ x_i \in X $, which we can
represent as
$$
(x_0, *, x_2, *, x_4, *, \dots)
$$
This gives us an extra copy of $ X^{[{\mathbb N}]} $.
Thus
$$
\Delta [X^{[{\mathbb N}]}] \cong {\mathbb N} \times
X^{[{\mathbb N}]} + \sum_{n=0}^\infty
X^{[n]} \rlap{\ .}
$$

A more general approach will lead to a
better understanding. Let $ G $ be a group and
$ A $ a left $ G $-set. Then
$$
\Delta [X^A/G]
$$
can be described as follows. $ G $ acts on the set
$ P' A $ of proper subsets of $ A $ by application
of the action element-wise
$$
B \subsetneq A \quad \to/|->/<200>^g \quad g B =
\{gb\,|\,b \in B \} .
$$
Let $ S \subseteq P' A $ be a choice of
representative for each orbit of this action. 
Then we have:

\begin{proposition}
\label{Prop-QuoDiff}

$$
\Delta [X^A/G] \cong \sum_{B \in S} A^B/{\rm Stab}(B)
$$
where $ {\rm Stab} (B) = \{g \in G \,|\, g B = B\} $, the
stabilizer of $ B $.

\end{proposition}

\begin{proof}

An element of $ \Delta [X^A/G] $ is an equivalence
class $ [f] $ of functions $ f \colon A \to X+1 $ with
$ * \in {\rm Im} (f) $. The equivalence relation is given by
$ f \sim f' $ iff there is $ g \in G $ with
$$
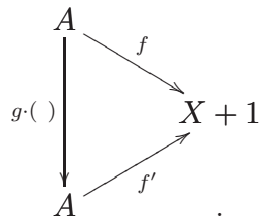
\bfig
\Dtriangle/>`>`<-/<500,300>[A`X+1`A;g\cdot (\ )`f`f']

\place(500,0)[.]

\efig
$$
A function into $ X+1 $ is equivalent to a partial
function
$$
A \to/<-</<200> \ B \to<200>^{\ov{f}} X
$$
and $ * \in {\rm Im} (f) $ is equivalent to $ B $
being a proper subset of $ A $. The equivalence
relation translates to the existence of $ g \in G $
such that
\begin{equation}\tag{*}
\bfig
\square/<-< `>`-->`<-< /[A`B`A`B';
`g \cdot (\ )`g \cdot (\ )`]

\place(250,250)[\framebox{Pb}]

\Dtriangle(500,0)/`>`<-/<400,250>[B`X`B';`\ov{f}`\ov{f}']

\efig
\end{equation}
and the square being a pullback is equivalent
to $ g (B) = B' $.

Any $ (B', \ov{f}') $ is equivalent to a $ (B, \ov{f}) $
with $ B \in S $ (there's a $ g \in G $ which maps
$ B' $ to $ B $, and defines $ \ov{f} $ by $ (*) $ above),
and for these $ (B, f) $ the equivalence relation
reduces to $ (*) $ with $ g \in {\rm Stab} (B) $.
\end{proof}

\begin{corollary}
\label{Cor-DivPowSeriesDiff}

The class of (generalized) divided power series is
closed under the difference operator $ \Delta $.

\end{corollary}

\subsection{Analytic functors}
\label{SSec-AnalyticDiff}

\begin{proposition}
\label{Prop-Analytic}

Let $ C $ be a left $ S_n $-set ($ n \in {\mathbb N} $). Then
there exist $ S_k $-sets $ C_k $, $ k = 0, 1, \dots, n - 1 $
such that
$$
\Delta [X^n \otimes_{S_n} C] \cong
\sum_{k = 0}^{n - 1} X^k \otimes_{S_k} C_k \rlap{\ .}
$$

\end{proposition}

\begin{proof}

An element of $ \Delta [X^n \otimes_{S_n} C] =
\bigl((X+1)^n \otimes_{S_n} C\bigr) \setminus (X^n \otimes_{S_n} C) $
is an equivalence class $ [x_1, \dots , x_n; c] $
where $ c \in C $ and $ x_i \in X $ or $ x_i = * $
for all $ i $, with at least one $ * $. $ [x_1, \dots , x_n; c] =
[y_1, \dots , y_n; d] $ iff there exists $ \sigma \in S_n $
such that $ y_i = x_{\sigma i} $ for all $ i $ and
$ c = \sigma d $.

The number of $ * $'s is invariant under the action
of $ S_n $ so is an invariant of the equivalence
class, but it is also invariant under the functorial
action (i.e.~for functions $ f \colon X \to Y $),
so that $ \Delta [X^n \otimes_{S_n} C] $
decomposes into a coproduct
$$
\sum_{k = 0}^{n - 1} \Phi_k
$$
of endofunctors $ \Phi_k $, where $ k $ is the
number of $ x $'s that are {\em not} stars.

Let $ C_k $ be the set of equivalence classes
$ [c] $ for $ c \in C $ with $ c \sim d $ iff there is
$ \sigma \in S_n $ such that $ \sigma(i) = i $
for all $ i \leq k $ and $ c = \sigma d $. $ S_k $
acts on $ C_k $ by $ \tau [c] = [\ov{\tau} c] $ 
where $ \ov{\tau} $ is $ \tau $ extended to a
permutation on $ n $ by the identity,
i.e.
$$
\ov{\tau} (i) = \left\{ \begin{array}{ll}
\tau (i) & 1 \leq i \leq k\\
i & k < i \leq n  \rlap{\ .}
\end{array} \right.
$$
The action is well defined because such
an extension will commute with the $ \sigma $'s
above as they act on disjoint sets. So
$$
\tau [c] = [\ov{\tau} c] = [\ov{\tau} \sigma d] =
[\sigma \ov{\tau} d] = [\ov{\tau} d] \rlap{\ .}
$$

The claim is that $ \Phi_k \cong X^k \otimes_{S_k} C_k $.
Each equivalence class in $ \Phi_k $ has a representative
of the form $ [x_1, \dots  x_k, *, \dots, *; c] $ (several in fact).

We define
$$
\phi \colon \Phi_k \to X^k \otimes_{S_k} C_k
$$
$$
\phi [x_1 \dots x_k, *, \dots, *;c] =
[x_1, \dots x_k, [c]] \rlap{\ .}
$$
$ \phi $ is well defined. Indeed, let $ [x_1, \dots, x_k, *,
\dots, *; c] = [y_1 \dots y_k, *, \dots, *;d] $ so there
exists $ \sigma \in S_n $ such that $ c = \sigma d $,
$ y_i = x_{\sigma i} $, and the $ * $'s get permuted
among themselves. Let $ \tau \in S_k $ be $ \sigma $
restricted to $ \{ 1, \dots, k \} $, and $ \ov{\tau} $ the
extension of $ \tau $ to $ \{1, \dots, n\} $ by the
identity on $ i > k $. Then
$$
\tau [d] = [\ov{\tau} d] = [\ov{\tau} \sigma^{-1} c] = [c]
$$
because $ \ov{\tau} $ and $ \sigma $ agree on
$ i \leq k $, i.e.~$ \ov{\tau} \sigma^{-1} (i) = i $.
As we have $ y_i = x_{\sigma (i)} = x_{\tau(i)} $
for $ i \leq k $, we have
$$
[x_1, \dots , x_k; [c]] = [y_1, \dots, y_k; [d]]
$$
and $ \phi $ is well defined.

We show that $ \phi $ is one to one. Suppose $ [x_1, \dots, x_k; [c]]
= [y_1, \dots, y_k ; [d]] $, i.e.~there exists $ \tau \in S_k $
such that $ y_i = x_{\tau i} $ and $ \tau [d] = [c] $. Thus
$ [\ov{\tau} d] = [c] $ ($ \ov{\tau} $ as above) which
means there is a $ \sigma \in S_n $ such that
$ \sigma (i) = i $ for $ i \leq k $ and $ c = \sigma \ov{\tau} d $.
For $ i > k $, $ \sigma (i) > k $ so

$$
\begin{array}{ll}
\sigma \ov{\tau} (i)   = \tau (i) & i \leq k\\
\sigma \ov{\tau} (i)   > k & i > k 
\end{array}
$$
giving our permutation in $ S_n $ making
$$
[x_1 \dots x_n, *, \dots, *; c] =
[y_1, \dots y_k, *, \dots, *;d] \rlap{\ .}
$$

It is clear that $ \phi $ is onto and just as clearly natural,
and so is the required isomorphism.
\end{proof}

\begin{corollary}
\label{Cor-Analytic}

If $ C \colon {\bf Bij} \to {\bf Set} $ is a species and
$$
F X = \sum_{n =  0}^\infty X^n \otimes_{S_n} C(n)
$$
the corresponding analytic functor, then
$$
\Delta{[}F{]} (X) \cong \sum_{n = 0}^\infty X^n
\otimes_{S_n} \left(\sum_{l = n+1}^\infty C (l)_n\right)
\rlap{\ .}
$$

\end{corollary}

\begin{proof}

By the proposition
$$
\Delta{[}F{]} (X) \cong \sum_{n = 0}^\infty
\left(\sum_{k = 0}^{n - 1} X^k \otimes_{S_k}
C(n)_k \right)
$$
and grouping like powers of $ X $ together we
get the desired result.
\end{proof}

The $ C_k $ above may seem mysterious,
at the very least a bit opaque. Passing to
permutations on infinite sets will force us into
a more conceptual presentation. Rather than
choose a proper subset of each cardinality as
we did in the case of $ [n] $ we consider them
all at once.

Let $ A $ be an arbitrary set and construct a
groupoid $ {\mathbb P} (A) $ ($ {\mathbb P} $
for proper powerset) whose objects are proper
subsets $ A_0 \subsetneq A $ and whose morphisms
are bijections $ \sigma $ of $ A $ preserving the
subset, i.e.
$$
A_0 \to B_0
\over{\sigma \colon A \to A \mbox{\quad s.t.\quad }
\sigma A_0 = B_0 \rlap{\ .}}
$$

For each set $ X $ we get a functor
$$
X^{(\ )} \colon {\mathbb P} (A)^{op} \to {\bf Set}
$$
with $ X^{A_0} $ the set of functions $ A_0 \to X $,
and for $ \sigma \colon A_0 \to B_0 $
$$
X^\sigma \colon X^{B_0} \to X^{A_0}
$$
$$
(g \colon B_0 \to X) \to/|->/<200>
(A_0 \to^{\sigma_0} B_0 \to^s X)
$$

$$
\bfig

\node a(0,0)[A]
\node b(500,0)[B_0]
\node c(800,250)[X]
\node d(0,500)[A]
\node e(500,500)[A_0]

\arrow/ >->/[b`a;]
\arrow/ >->/[e`d;]
\arrow|l|/>/[d`a;\sigma]
\arrow|l|/>/[e`b;\sigma_0]
\arrow|a|/>/[e`c;X^\sigma (g)]
\arrow|b|/>/[b`c;g]

\place(800,0)[.]

\efig
$$

Let $ C $ be a left $ S_A $-set. The action of
$ S_A $ on $ C $ restricts to a
functor
$$
\ov{C} \colon {\mathbb P} (A) \to {\bf Set}
$$
$$
\bfig
\square/|->`>`>`|->/[A_0`C`B_0`C\rlap{\ .};
`\sigma`\sigma \circ (\ )`]

\place(250,250)[\mapsto]

\efig
$$

\begin{proposition}

$$
\Delta [X^A \otimes_{S_A} C] \cong \int^{A_0 \in {\mathbb P}(A)}
X^{A_0} \times \ov{C} .
$$

\end{proposition}

\begin{proof}

An element of $ \Delta [X^A \otimes_{S_A} C] $ is
an equivalence class of pairs $ [f \colon A \to X+1, c] $
for $ f $ a function with $ * $
in its image and $ c \in C $, with $ [f, c] = [g, d] $ iff
there is $ \sigma \in S_A $ with $ f = g \sigma $ and
$ d = \sigma c $. $ f \colon A \to X + 1 $ is a partial
map 
$
\bfig

\node a (0,0)[A]
\node b(400,0)[A_0]
\node c(800,0)[X]

\arrow/ >->/[b`a;]
\arrow|a|/>/[b`c;f_0]
\efig
$
with $ A_0 $ a proper subset, and $ f = g \sigma $
means
$$
\bfig

\node a(0,0)[A]
\node b(500,0)[B_0]
\node c(800,250)[X]
\node d(0,500)[A]
\node e(500,500)[A_0]

\arrow/ >->/[b`a;]
\arrow/ >->/[e`d;]
\arrow|l|/>/[d`a;\sigma]
\arrow|r|/>/[e`b;\sigma_0]
\arrow|a|/>/[e`c;f_0]
\arrow|b|/>/[b`c;g_0]

\place(800,0)[.]

\efig
$$
This is exactly a description of a general
element of $ \int^{A_0} X^{A_0} \times C $.
\end{proof}

We can simplify things by choosing a
skeleton of $ {\mathbb P}(A) $. $ A_0 \cong B_0 $
iff there exists $ \sigma \colon A \to A $ in $ S_A $
such that $ \sigma A_0 = B_0 $ which implies
that the cardinality of $ A_0 $ is equal to that of
$ B_0 $, $ \# A_0 = \# B_0 $, but the complements
of $ A_0 $ and $ B_0 $ also have the same
cardinality, $ \# A'_0 = \# B'_0 $. And this is sufficient
to have $ A_0 \cong B_0 $. Thus for every pair
of non-zero cardinals $ \kappa = (\kappa_1, \kappa_2) $
such that $ \kappa_1 + \kappa_2 = \# A $, we choose
a subset $ A_\kappa \subsetneq A $ with
$ \# A_\kappa = \kappa_i $ and $ \# A'_\kappa =
\kappa_2 $, and take the full subcategory
$ {\mathbb P}' (A) $ of $ {\mathbb P} (A) $
determined by these. $ {\mathbb P}' (A) $ is now
the coproduct of the groups $ S_{A_\kappa} \times
S_{A'_\kappa} $, which are the groups of automorphisms
of the $ A_k $. The coend is then the coproduct of the
coends over each of these groups, which is what
we've been writing as the tensor product. Thus
we have the following:

\begin{corollary}

$$
\Delta (X^A \otimes_{S_A} C) =
\sum_\kappa X^{A_\kappa} \otimes_{S_{A_\kappa} \times
S_{A'_\kappa}} C
$$
where the coproduct is taken over $ \kappa = (\kappa_1, \kappa_2) $
with $ \kappa_1, \kappa_2 > 0 $ and $ \kappa_1 + \kappa_2 = \# A $.

\end{corollary}

This is still not in the form we would like. What we would
like is to express it in terms of $ \otimes $ over
symmetric groups. However
$$
X^{A_\kappa} \otimes_{S_{A_\kappa} \times
S_{A'_\kappa}} C \cong X^{A_\kappa} \otimes_{S_{A_\kappa}}
\left(S_{A_\kappa} \otimes_{S_{A_\kappa} \times
S_{A'_\kappa}} C\right)
$$
where the $ S_{A_\kappa} $ in the middle is given 
the left $ S_{A_\kappa} $ action by left multiplication
and the right $ S_{A_\kappa} \times
S_{A'_\kappa} $ action by right multiplication after
projecting onto $ S_{A_\kappa} $.

We can further analyze $ S_{A_\kappa} \otimes_{S_{A_\kappa} \times
S_{A'_\kappa}} C $. An element is an equivalence class
of pairs $ (\sigma, c) $ with $ \sigma \in S_{A_\kappa} $ and
$ c \in C $, with $ (\sigma, c) \sim (\tau, d) $ iff
there exist $ \rho \in S_{A_\kappa} $ and $ \rho' \in
S_{A'_\kappa} $ such that
$$
\tau = \sigma \rho \quad \mbox{and}\quad c= 
(\rho + \rho')d  \rlap{\ .}
$$
Each equivalence class has representatives
of the form $ (\id, c) $ and for these the equivalence
relation reduces to $ (\id, c) \sim (\id, d) $ iff
there exists $ \rho' \in S_{A'_\kappa} $ such that
$ c = (\id + \rho') d $. So
$$
S_{A_\kappa} \otimes_{S_{A_\kappa} \times S_{A'_\kappa}}
C \cong C_\kappa
$$
with $ C_\kappa= C/\sim $ where
$$
c \sim d \Leftrightarrow \exists \rho' \in S_{A_{\kappa'}} 
(c = (\id + \rho') d) \rlap{\ .}
$$ 

The upshot of this long discussion is that:

\begin{corollary}

$$
\Delta [X^A \otimes_{S_A} C)] =
\sum X^{A_\kappa} \otimes_{A_\kappa} C_\kappa
\rlap{\ ,}
$$
and $ \Delta $ of a generalized analytic functor is
again one.

\end{corollary}

\subsection{Reduced powers}
\label{SSec-RedPowDiff}

Let \( {\cal F} \) be a filter on \( A \). An element of
\( (X + 1)^{\cal F} \) can be identified with an
equivalence class of partial functions
\( A \to/<-< /<200> A_0 \to^f X \) with
\( (A \to/<-< /<200> A_0 \to^f X) \sim
(A \to/<-< /<200> B_0 \to^g X) \) if and only if
\[
\{ a \in A_0 \cap B_0 \ | \ f(a) = g(x) \} \cup
(A'_0 \cap B'_0) \in {\cal F}
\]
where \( A'_0 \) and \( B'_0 \) are the complements
of \( A_0 \) and \( B_0 \) respectively.

If we take \( X = 1 \), the \( f \) is redundant so that
an element of \( (1 + 1)^{\cal F} \) may be identified
with an equivalence class of subsets \( A_0 \subseteq A \)
with \( A_0 \sim B_0 \) if and only if
\[
(A_0 \cap B_0) \cup (A'_0 \cap B'_0) \in {\cal F} .
\]

The canonical map \( (X + 1)^{\cal F} \to (1 + 1)^{\cal F} \)
partitions \( (X + 1)^{\cal F} \) into a disjoint union indexed
by the equivalence classes of subsets \( [A_0] \),
\[
(X + 1)^{\cal F} \cong \prod_{[A_0]} (X + 1)^{\cal F}_{[A_0]} 
\]
where an element of \( (X + 1)^{\cal F}_{[A_0]} \)
corresponds to a partial function
\( A \to/<-< /<200> B_0 \to^g X \) if and only if
\( B_0 \sim A_0 \).

For any subset \( A_0 \subseteq A \) let
\[
{\cal F}_{A_0} = \{ A_1 \subseteq A_0 \ | \ A_1 \cup
A'_0 \in {\cal F} \} \ .
\]

\begin{proposition}

\( {\cal F}_{A_0} \) is a filter on \( A_0 \) and
\( (X + 1)^{\cal F}_{[A_0]} \cong X^{{\cal F}_{A_0}} \).

\end{proposition}

\begin{proof}

\( {\cal F}_{A_0} \) is easily seen to be a filter. An
element of \( X^{{\cal F}_{A_0}} \) is an equivalence
class of functions \( [A_0 \to^t X] \).
\[
[A_0 \to^f X] = [A_0 \to^g X] \Longleftrightarrow
\{a \in A_0 \ | \ fa = ga\} \in {\cal F}_{A_0}
\]
\[
\Longleftrightarrow \{a \in A_0 \ | \ fa =ga \}
\cup A'_0 \in {\cal F}
\]
\[
\Longleftrightarrow [A \to/<-< /<200> A_0 \to^f X]
= [A \to/<-< / <200> A_0 \to^g X] \mbox{\ \ in\ \ }
(X + 1)^{\cal F} .
\]
So taking \( [A_0 \to^f X] \) in \( X^{\cal F}_{A_0} \)
to \( [A \to/<-< /<200> A_0 \to^f X] \) in
\( (X + 1)^{\cal F} \) is well-defined and one-to-one.
To see that it's onto, first assume that
\( X \neq \emptyset \), and take
\( [A \to/<-< /<200> B_0 \to^g X] \) in
\( (X + 1)^{\cal F}_{[A_0]} \) and define
\( f \colon A_0 \to X \) by
\[
f (a) = \left\{ \begin{array}{ll}
      g(a)  \mbox{\quad if \( a \in A_0 \cap B_0\)} & \\
      \mbox{arbitrary, otherwise} & 
            \end{array}
      \right.
      \]
Then
\[
\{a \in A_0 \cap B_0 \ | \ f(a) = g(a)\} \cup (A'_0 \cap B'_0)
\]
\[
= (A_0 \cap B_0) \cup (A'_0 \cap B'_0) \in {\cal F}
\]
because \( A_0 \sim B_0 \). If \( X = \emptyset \) then
\( (X + 1)^{\cal F}_{[A_0]} \) is also empty unless
\( [A_0] = [\emptyset] \), and \( X^{{\cal F}_{A_0}} \)
which is contained in \( (X + 1)^{\cal F}_{[A_0]} \) is
also empty. For \( [A_0] = [\emptyset] \), 
\( (X + 1)^{\cal F}_{[A_0]} \cong 1 \) and so is
\( X^{{\cal F}_{[A_0]}} \) as \( \emptyset \in
{\cal F}_{[A_0]} \).
\end{proof}

\begin{corollary}

 \( \Delta [X^{\cal F}] = \sum_{[A_0] \neq [A]} X^{{\cal F}_{A_0}} \)
 
 \noindent where the coproduct is taken over a set of
 representatives of the \( {\cal F} \)-equivalence classes
 of subsets \( A_0 \subseteq A \) with \( [A_0] \neq [A] \).

\end{corollary}

\begin{proof}

\( (X + 1)^{\cal F} \cong \sum X^{\cal F}_{A_0} \) over all
classes and \( {\cal F}_A = {\cal F} \).
\end{proof}

\begin{corollary}

If $ {\cal U} $ is an ultrafilter, then
$$
\Delta [X^{\cal U}] \cong 1 \rlap{\ .}
$$

\end{corollary}

\begin{proof}

As \( {\cal U} \) is an ultrafilter, for any subset
\( A_0 \subseteq A \), either \( A_0 \in {\cal U} \)
or \( A'_0 \in {\cal U} \). If \( A_0 \in {\cal U} \)
then \( [A_0] = [A] \). If \( A'_0 \in {\cal U} \),
then \( [A_0] = [\emptyset] \) for
\[
(A_0 \cap \emptyset) \cup (A'_0 \cap \emptyset') =
A'_0 \in {\cal U} .
\]
Thus there are only two classes \( [A] \) and
\( [\emptyset] \) so \( \Delta [X^{\cal U}] =
X^{{\cal F}_\emptyset} \cong 1 \) as
\( {\cal F}_\emptyset = 2^A \).
\end{proof}

This gives an example of two functors, not
differing by a constant, with the same finite
difference, namely $ X $ and $ X^{\cal U} $.

We could define {\em reduced power series}
to be coproducts
$$
\sum_{i\,\in\,I} C_i \times X^{{\cal F}_i}
$$
and get that $ \Delta $ of such is again one.
Apart from $ \Delta [X^{\cal F}] $ we don't know
of any naturally arising examples.

Note that all filters $ {\cal F} $ on finite sets
are principal, i.e.~the set of all
subsets of $ A $ containing some fixed subset
$ A_0 $, namely the intersection of all elements
of $ {\cal F} $. Then $ X^{\cal F} \cong X^{A_0} $,
so reduced powers are an essentially infinite
phenomenon, infinite powers, so there would be
no corresponding thing in algebra or analysis.

\subsection{Monads}
\label{SSec-Filt}

All of the previous examples were variations on
the power series theme. The filter monad
$ {\mathbb F} $ is of a different nature. It was
central to Manes' paper introducing taut functors
\cite{Man02}. Not only is the functor $ {\mathbb F} $
taut but the unit $ \eta $ and multiplication $ \mu $
are taut transformations, i.e.~$ {\mathbb F} $ is
a taut monad.

Recall that $ {\mathbb F} (X) = \{{\cal F} \ |\  {\cal F} 
\mbox{\  is a filter on\ } X \} $. If $ f \colon X \to Y $
then we have $ {\mathbb F} (f) ({\cal F}) = \{Y_0 \subseteq Y
\ |\ f^{-1} Y_0 \in {\cal F}\} $. The unit $ \eta \colon
X \to {\mathbb F} (X) $ takes an element $ x $ to
the principal ultrafilter generated by $ x $,
i.e.~$\{X_0 \subseteq X \ |\ x \in X_0\}$. The
multiplication is a bit harder, and won't concern
us here.

\begin{proposition}

$$
\Delta [{\mathbb F}] = {\mathbb F} \rlap{\ .}
$$

\end{proposition}

\begin{proof}

Let $ {\cal F} $ be a filter on $ X + 1 $ and

\begin{align*}
{\cal F}_0 & = \{ X_0 \subseteq X\ |\ X_0 \in {\cal F} \},\\
{\cal F}_1 & = \{ X_1 \subseteq X \ |\ X_1 \cup\{*\} \in {\cal F}\}.
\end{align*}
$ {\cal F}_0 $ may be empty, but if not it's a filter
on $ X $ and $ {\cal F}_1 $ is always a filter on $ X $.
Clearly $ {\cal F}_0 \subseteq {\cal F}_1 $. If $ {\cal F}_0
\neq \emptyset $, then $ X \in {\cal F}_0 $ so $ X \in {\cal F} $.
Then for any $ X_1 \in {\cal F}_1 $, $ X_1 = X \cap
(X_1 \cup \{*\}) $ is in $ {\cal F} $ so $ X_1 \in {\cal F}_0 $ and
$ {\cal F}_0 = {\cal F}_1 $.

Conversely, given any filter $ {\cal F} $ on $ X $
we get two filters on $ X + 1 $
$$
\ov{\cal F} = \{X_0 \cup \{*\}\ |\ X_0 \in {\cal F}\}
$$
and
$$
\ov{\ov{\cal F}} = {\cal F} \cup \{X_0 \cup\{*\}\ |\ X_0 \in {\cal F}\}\rlap{\ .}
$$
Thus the filters on $ X + 1 $ fall into two disjoint classes,
the $ \ov{\cal F} $'s and the $ \ov{\ov{\cal F}} $'s, and the
$ \ov{\ov{\cal F}} $'s are precisely the images of $ \ov{\cal F} $'s
in $ {\mathbb F} (X) $ under the inclusion
$ {\mathbb F} (j) \colon F(X) \to F (X + 1) $. Consequently,
$ \Delta [{\mathbb F}](X) $ consists of all the $ \ov{\cal F} $'s
and is isomorphic to $ {\mathbb F} (X) $ itself.
\end{proof}

If $ {\mathbb F}' $ is the submonad of $ {\mathbb F} $
of proper filters then
for any filter $ {\cal F} $ on $ X $, $ \ov{\cal F} $ will
always be proper whereas $ \ov{\ov{\cal F}} $ will
only be proper when $ {\cal F} $ is. This gives the
following result.

\begin{proposition}

$$ \Delta [{\mathbb F}'] \cong {\mathbb F} \rlap{\ .}
$$

\end{proposition}

The ultrafilter monad, usually called $ \beta $, is
also a submonad of $ {\mathbb F} $. An ultrafilter
$ {\cal U} $ on $ X + 1 $ will be of the form
$ \ov{\ov{\cal F}} $ only if $ {\cal F} $ is the powerset
$ PX $ and then $ {\cal U} $ will be $ \langle *\rangle $,
the principal ultrafilter determined by $ * $. On the
other hand $ \ov{\ov{\cal F}} $ will be an ultrafilter
iff $ {\cal F} $ is one on $ X $. It follows that
$ \beta (X+1) \cong \beta X + 1 $ and we get:

\begin{proposition}

$ \Delta [\beta] = 1 $.

\end{proposition}

In fact, as is well known and easy to prove, $ \beta $
preserves finite coproducts and $ \beta 1 = 1 $, from
which the proposition follows, without the analysis of
$ {\mathbb F} $.

The covariant powerset monad $ {\mathbb P} $ can
also be considered a submonad of $ {\mathbb F} $,
by the ``principal filter inclusion'',
$$
A \subseteq X \longmapsto \langle A \rangle =
\{X_0 \subseteq X\ |\ A \subseteq X_0 \} \rlap{\ .}
$$
For $ f \colon X \to Y $,
\begin{align*}
{\mathbb F} (f) (\langle A\rangle) & =
 \{Y_0 \subseteq Y\ |\ f^{-1} Y_0 \in \langle A\rangle\}\\
         & = \{Y_0 \subseteq Y\ |\ A \subseteq f^{-1} Y_0\}\\
         & = \{Y_0 \subseteq Y \ |\ fA \subseteq Y_0\}\\
         & = \langle fA\rangle \rlap{\ .}
\end{align*}
So we could analyze $ \Delta [{\mathbb P}] $ in terms
of $ {\cal F} $, but it is easier done directly, though
it's really the same thing.

A subset of $ X + 1 $ either contains $ * $ or not.
In the first case it is of the form $ X_1 \cup\{*\} $
for $ X_1 \subseteq X $ and in the second it's
just $ X_0 \subseteq X $, the image of $ X_0 $
under $ {\mathbb P} (j) $. This gives:

\begin{proposition}

$ \Delta [{\mathbb P}] \cong {\mathbb P} \rlap{\ .} $

\end{proposition}

Although \( \Delta [{\mathbb F}] = \Delta [{\mathbb F'}] =
{\mathbb F} \), \( \Delta [{\mathbb P}] = {\mathbb P} \)
and \( \Delta [\beta] = 1 \) are all monads it's the
exception rather than the rule that \( \Delta \) of a
taut monad is again a monad. Note that, although
\( \Delta [\beta] = 1 \) is a monad, it is not a taut one
as the unit \( X \to 1 \) is not taut. Even for such a
basic monad as \( T X = M \times X \), for \( M \)
a monoid, its difference is the constant functor \( M \)
which may look like a monad but it's not (unless
\( M = 1 \)).

The reason is that for \( (T, \eta, \mu) \) a taut monad,
we have a pullback
\[
\bfig
\square/>`<-< `<-< `>/<650,500>[X + 1`(X + 1)`X`T(X)\rlap{\ ,};
\eta (X + 1)```\eta X]

\place(325,250)[\framebox{\scriptsize Pb}]

\efig
\]
so, in passing to the difference \( \Delta [T] (X) =
T (X + 1) \setminus T (X) \), we've removed the
unit.

However, as we saw in Theorem~\ref{Thm-DMMonoidal}, looking
at the difference operator as producing an endomorphism
of the ``tangent space'', we get a monoidal functor
\[
D \colon {\bf End}_{T\!aut} ({\bf Set}) \to {\bf End}_{T\!aut}
({\bf Set} \times {\bf Set})
\]
which does preserve monoids, i.e.~monads. So, for
a taut monad \( (T, \eta, \mu) \) on \( {\bf Set} \) we
get a taut monad \( D (T, \eta, \mu) \) on
\( {\bf Set} \times {\bf Set} \)
\[
\begin{array}{lll}
{\bf Set} \times {\bf Set} & \to/>/<200> & {\bf Set} \times {\bf Set}\\
(A, B)   & \to/|->/<200>  & (TA, \Delta [T](A) \times B)\rlap{\ .}
\end{array}
\]
The unit is
\[
(A, B) \to^{(\eta A, h A \times B)} (TA, \Delta [T](A) \times B)
\]
where \( h A \) is defined by
\[
\bfig
\square/>` >->` >->`>/<650,500>[1`\Delta {[T]} (A)`A +1`T(A+1);
h A`j_2``\eta (A+1)]

\efig
\]
which exists because \( \eta \) is taut. It is just
\( \Delta [\eta] (A) \).

The multiplication comes from the chain rule
transformation. It is
\[
(T^2 A, \Delta [T](TA) \times \Delta [T] (A) \times B)
\to^{(\mu A, m A \times B)}
(TA, \Delta [T] (A) \times B)
\]
where \( m A \) is the composite
\[
\Delta [T] (TA) \times \Delta [T] (A) \to/ >->/<200>^\gamma
\Delta [T^2] (A) \to^{\Delta[\mu](A)}
\Delta [T] (A) \rlap{\ .}
\]

The \( B \) looks just tacked on with nothing to do
with \( T \) or \( A \), but they are intertwined via
the associativity law for the monad
\( D (T, \eta, \mu) \). We won't look at this directly,
but it is apparent when we consider Eilenberg-Moore
algebras for it. An algebra is
\[
(\alpha, \beta) \colon (TA, \Delta [T] (A) \times B) \to (A, B)
\]
such that \( (A, \alpha) \) is a \( T \)-algebra and
\( \beta \) makes the following diagrams commute
\[
\bfig
\qtriangle/>`>`>/<800,450>[1 \times B`\Delta {[}T{]} (A) \times B`B;
h \times B`\cong`\beta]

\efig
\]
and
\[
\bfig
\square/>``>`>/<1600,1000>[\Delta {[}T{]} (TA) \times \Delta {[}T{]} (A) \times B
`\Delta {[}T{]}(A) \times B`\Delta {[}T{]} (A) \times B`B\rlap{\ .};
\Delta {[}T{]} (\alpha) \times B``\beta`\beta]

\morphism(0,1000)/>/<0,-500>[\Delta {[}T{]} (TA) \times \Delta {[}T{]} (A) \times B`
\Delta {[}T^2{]}(A) \times B;\gamma \times B]

\morphism(0,500)/>/<0,-500>[\Delta {[}T^2{]}(A) \times B`
\Delta {[}T{]} (A) \times B;\Delta {[}\mu{]} \times B]

\efig
\]
We see the \( \alpha \) appearing in the top row  in
the equations for \( \beta \).

\subsection{Dirichlet functors}
\label{SSec-DirichletDiff}

The analysis of $ {\mathbb P} $ easily generalizes to
the covariant exponentials $ L^X $.

\begin{proposition}

Let $ L $ be a sup-lattice, then
$$
\Delta [L^X] \cong L_* \times L^X
$$
where $ L_* = \{l \in L\ |\ l \neq \bot\} $.

\end{proposition}

\begin{proof}

An element of $ L^{(X+1)} $ is a function
$ X + 1 \to^{[\phi, l]} L $ where $ \phi \colon X \to L $
and $ l \in L $. If $ f \colon X \to Y $ and
$ L^f [\phi, l] = [\psi, m] \colon Y+1 \to L $,
then $ \psi (y) = \bigvee_{f x = y} \phi(x) $ and
$ m $ is the sup over all pre-images of $ * $,
i.e.~$ \{*\} $, so $ m = l $, i.e.~$ L^{(X+1)} \cong
L^X \times L $, as functors.

$ L^j (\phi) = [\phi, \bot] \colon X + 1 \to L $ where
$ \phi \colon X \to L $ and $ \bot $ is the bottom
element of $ L $, i.e.~the sup of the empty set,
$ j^{-1} \{*\} $. This gives the result.
\end{proof}

Recall that for $ l \in L $, a sup lattice, $ D(l) $
denotes the down-set of $ l $, $ \{l'\in L\ |\ l' \leq l\} $,
which is a sub sup lattice of $ L $.

\begin{proposition}

Let $ C_l = \{l' \in L \ |\ l \vee l' = \top \mbox{\ and\ }
 l' \neq \bot\} $. Then
 $$
 \Delta [L^{[X]}] \cong \sum_{l \in L} C_l \cdot
 D(l)^{[X]} \rlap{\ .}
 $$

\end{proposition}

\begin{proof}

The proof is similar to that of the previous proposition.
An element of $ \Delta [L^{[X]}] $ is a function
$ [\phi, l'] \colon X+1 \to L $ for $ l' \neq \bot $ and
$ \bigvee_{x \in X} \phi (x) \vee l' = \top $. If we let
$ l = \bigvee_{x \in X} \phi (x) $ then
$ \phi \in D(l)^{[X]} $. And there's one such $ \phi $
for each $ l' \neq \bot $ and $ l \vee l' = \top $.
\end{proof}

\begin{corollary}

$ \Delta $ of a Dirichlet functor is a Dirichlet functor.

\end{corollary}

\begin{corollary}
\label{Cor-DeltaOrd}

(1) If $ \top $ is join-irreducible then
$$
\Delta [L^{[X]}] = L_* \cdot L^{[X]} +
\sum_{l \neq \top} D(l)^{[X]}\rlap{\ .}
$$
\ 

\noindent (2) $ \Delta [{\bf n}^{[X]}] = (n - 1) \cdot 
{\bf n}^{[X]} + ({\bf n - 1})^{[X]} +
({\bf n - 2})^{[X]} + \dots +
{\bf 2}^{[X]} + {\bf 1}^{[X]} \rlap{\ .} $

\

\noindent (3) $ \Delta [p_*^{[X]}] = \pi (p) \cdot
p_*^{[X]} + \sum q_*^{[X]} + 1 $, the coproduct
taken over all primes $ < p $.

\end{corollary}

\begin{proof}

(1) $ \top $ join-irreducible means $ l \vee l' =
\top \Rightarrow l = \top $ or $ l' = \top $, so
$$ C_\top = L_* = \{l'\ |\ l' \neq \bot \}  \mbox{\ \ and if\ \ }
 l \neq \top
 $$
 then $ C_l = \{ \top\} \cong 1 $.

\noindent (2) The top element of $ {\bf n} $ is
join-irreducible, so this is a special case of (1).

\noindent (3) This is a special case of (2), given
the definition of $ p_* $.
\end{proof}

\begin{proposition}

$ \Delta $ of a sequential Dirichlet functor is
again one.

\end{proposition}

\begin{proof}

By Corollary~\ref{Cor-DeltaOrd} (3) we know
that for a prime $ p $, $ \Delta p_*^{[X]} $ is
a sequential Dirichlet functor. For a composite
$ n $, $ n_* $ is a cartesian product
of $ p_* $'s so by Proposition~\ref{Prop-SeqDir},
$ \Delta n_*^{[X]} $ is also a
sequential Dirichlet functor. Finally, $ \Delta $
of a sequential Dirichlet functor is a coproduct
of $ \Delta n_*^{[X]} $'s so is
again one.
\end{proof}

% !TEX root = taut.tex

\section{A Newton summation formula}
\setcounter{subsection}{1}

For a function $ f \colon {\mathbb R} \to {\mathbb R} $, the
difference operator $ \Delta $
$$
\Delta [f] (x) = f (x + 1) - f (x)
$$
can be iterated giving discrete versions of  higher
derivatives. Thus
$$
\Delta^2 [f] (x) = f(x + 2) - 2 f(x + 1) + f(x)
$$
$$
\Delta^3 [f] (x) = f (x + 3) - 3 f (x + 2) + 3 f(x + 1) - f (x)
$$
and so on.

The Newton series is a discrete analog of Taylor series,
trying to recover $ f $ or some reasonable approximation
of $ f $, from specific values $ \Delta^k [f] (a) $. It takes
the form
\begin{eqnarray}
g (x) &  = & \sum^\infty_{n = 0} {{\Delta^n [f](a)}\over {n!}} (x - a)^{\downarrow n}\\
 & = & \sum^\infty_{n = 0} \binom{x - a}{n}
  \Delta^n [f](a)
\end{eqnarray}
where $ (x - a)^{\downarrow n} $ is the {\em falling power}
$$
(x - a) (x - a - 1) (x - a - 2) \dots (x - a - n + 1)
$$
and $ \binom{x - a}{n} $ is the ``binomial coefficient''
$$
{(x - a) (x - a - 1) \dots (x - a - n + 1)}\rlap{\ \ .}\over
{n!}
$$

The formulas (1) and (2) look sufficiently combinatorial
to suggest that there may be a similar formula for
taut endofunctors of $ {\bf Set} $, and indeed there is.
That is what we develop in this section.

This section has significant overlap with \cite{SzaZaw15}
but it is hard to isolate what precisely. Their objective was
completely different from ours, concentrating on monads,
theories, and operads. But crucial points here appear at
various points in their work. They clearly recognized the
importance of the category $ {\bf Surj} $ of finite cardinals
and surjections (there called $ {\mathbb S} $) and its
relation to soft analytic functors (called semi-analytic).
Their Theorem~2.2 is basically Part~I of our main
theorem below. In particular the $ A [n] $ of their
Lemma~2.5 is $ \Delta^n [F] (0) $ by our
Corollary~\ref{Cor-DeltaK}.

Neither of the formulas (1) or (2) generalize directly
to functors, even for $ a = 0 $. We need more structure
on the higher differences $ \Delta^n [F] $, and to
study this we need a better understanding of them.

The suggestive notation $ \Delta [F] (X) = F (X + 1) \setminus
F(X) $ is too suggestive here. Blindly applying the
difference twice would give
\begin{eqnarray}
\Delta^2 [F] (X) & = & (F(X+2)\setminus F(X+1))\setminus
(F(X+1) \setminus F(X))\nonumber\\
 & = & F(X+2) \setminus F(X+1)\nonumber
\end{eqnarray}
which is definitely wrong, as can easily be seen by
taking $ F (X) = X^2 $. We must be mindful of which
injection we are complementing.

To track the injections, let's take two different one-point
sets $ 1_a = \{a\} $ as $ 1_b = \{b\} $ with corresponding
difference operators
\begin{eqnarray}
\Delta_a [F] (X) & = & F(X + 1_a) \setminus F(X)\nonumber\\
\Delta_b [F] (X) & = & F (X + 1_b) \setminus F(X)\nonumber\rlap{\ .}
\end{eqnarray}
Thus we have
\begin{eqnarray}
\Delta_a [\Delta_b [F]] (X) & = & \Delta_b [F] (X+1_a) \setminus \Delta_b[F] (X)\nonumber\\
 & = & \Bigl(F(X+1_a +1_b)\setminus F(X+1_a)\Bigr) \setminus \Bigl(F(X+1_b)\setminus F(X)\Bigr) \nonumber\\
 & = & F (X+1_a +1_b) \setminus \Bigl(F(X+1_a) \cup F (X+1_b)\Bigr)\nonumber
\end{eqnarray}
as $ F (X) \subseteq F (X + 1_a) $.

This leads to the following. Let $ S_A \colon {\bf Set} \to {\bf Set} $
be the translation functor $ S_A (X) = X + A $, which is obviously
taut. For any taut functor $ F \colon {\bf Set} \to {\bf Set} $ and
any set $ X $, let $ D_A [F] (X) $ be the subset of $ F(X+A) $
consisting of those $ a \in F(X+A) $ not in the image of
$ F(X+A_0) \to/ >->/<200> F (X+A) $ for any proper subset 
$ A_0 \subsetneq A $. As there will be much talk of proper
subsets in what follows, we will use an arrow with a double tail
$ A_0 \to/ >>->/<200> A $ to indicate a monomorphism which
is not an isomorphism.

\begin{proposition}
\label{Prop-DATaut}

$ D_A [F] $ is a taut subfunctor of $ FS_A $.

\end{proposition}

\begin{proof}

First we show it's a subfunctor. Let $ f \colon X \to Y $
and $ a \in FS_A (X) = F (X + A) $. Then $ F (f+A) (a) 
\in F(Y+A) $. Let $ A_0 \to/ >>->/<200> A $ be a proper
subset of $ A $. Then as
$$
\bfig
\square/ >->`>`>` >->/<550,500>[X+A_0`X+A`Y+A_0`Y+A;
`f+A_0`f+A`]

\efig
$$
is a pullback and as $ F $ is taut
$$
\bfig
\square/ >->`>`>` >->/<700,500>[F(X+A_0)`F(X+A)`F(Y+A_0)`F(Y+A);
`F(f + A_0)`F(f+A)`]

\efig
$$
is also a pullback. Thus if $ F(f+A)(a) $ were in
$ F(Y+A_0) $, $ a $ would be in $ F(X+A_0) $. So
if $ a \in D_A [F](X) $, $ F(f+A)(a) $ is in
$ D_A[F](Y) $, i.e.~$ D_A [F] $ is a subfunctor of
$ FS_A $.

$ D_A [F](X) $ is defined to be the complement of
$$
\bigcup F (X+A_0) \subseteq F (X+A) \rlap{ ,}
$$
the union taken over all proper subsets
$ A \to/ >>->/<200> A $. Thus $ D_A [F] $ is a
complemented subfunctor of $ F S_A $ and
by Proposition~\ref{Prop-CoprodTaut} is taut.
\end{proof}

\begin{proposition}

$ D_A [D_B [F]] \cong D_{A+B} [F] \rlap{ .} $

\end{proposition}

\begin{proof}

An element $ a \in D_A [D_B[F]] (X) $ is an
element $ a \in D_B[F] (X+A) $ which is not in
any $ D_B [F] (X+A_0) $ for a proper subset
$ A_0 \subsetneq A $. The first condition means
that $ a \in F(X+A+B) $ but not in any
$ F(X+A+B_0) $ for a proper subset $ B_0 \subseteq B $.
The second condition ($ (a \notin D_B [F] (X+A_0) $)
means that $ a $ is not in $ F(X+A_0+B) $ or there
exists a proper subset $ B_0 \subsetneq B $ with
$ a \in F (A+A_0+B_0) $. But this last condition is
impossible, for if $ a $ were in $ F(X+A_0+B_0) $
it would be in $ F(X+A+B_0) $ which it is not. The
conclusion is that $ a \in D_A[D_B [F] (X) $ if and
only if $ a \in F(X+A+B) $ but $ a \notin F(X+A+B_0) $
and $ a \notin F (X+A_0+B) $. This is equivalent to
$ a \in F(X+A+B) $ and $ a \notin F(X+C) $ for any
proper subset $ C \subsetneq A+B $ because such
a $ C $ would be contained in either a $ A_0 +B $
or a $ A + B_0 $. The result follows.
\end{proof}

\begin{corollary}
\label{Cor-DeltaK}

For any finite cardinal $ n $
$$
\Delta^n [F] = D_n [F]
$$
$$
\Delta^n [F] (X) = \{a \in F(X+n) | a \notin F(X+n_0) 
\mbox{\ \ for any\ \ } n_0 \to/ >>->/<200> n \} .
$$

\end{corollary}

\begin{proof}

$ \Delta^0 [F] = F = D_0 [F] $ and $ \Delta^1 [F]
= \Delta F = D_1 [F] $ by definition. The result now 
follows from the previous proposition by induction.
\end{proof}

We are hoping to recover $ F $ from the sequence
$ \Delta^n [F](0) $ via some version of the Newton
series, at least for polynomial functors. From the
above description of $ \Delta^n [F] (0) $, it is clear
that $ S_k $ acts on $ \Delta^n[F](0) $ so we get a
symmetric sequence (species) and a corresponding
analytic functor
$$
\sum^\infty_{n = 0} X^n \otimes_{S_n} \Delta^n [F](0) .
$$
This looks promising but it doesn't give $ F $ even
for polynomials. For example $ F (X) = X^n $ is
connected but the above analytic functor is not.
It's defined as a sum of (non trivial) functors. We
need something a bit tighter. In fact, not only do
bijections act on the $ \Delta^n [F](a) $ but
epimorphisms do too, in the appropriate sense of
course.

Let $ {\bf Surj} $ be the category of finite cardinals
with surjections as morphisms.

Although it will not be used below, it may be of
interest to note that \( {\bf Surj} \) is the free
symmetric monoidal category generated by a
commutative semigroup.

\begin{proposition}
\label{Prop-FunSurj}

For any taut functor $ F \colon {\bf Set} \to {\bf Set} $
and set $ A $, the family $ \langle \Delta^n [F](A)\rangle_n $
is the object part of a functor
$$
{\bf Surj} \to {\bf Set} .
$$

\end{proposition}

\begin{proof}

Let $ f \colon m \to/->>/<200> n $ be a surjection and
$ a \in \Delta^m [F](A) $, so $ a \in F (A+m) $ but
$ a \notin F(A+m_0) $ for any $ m_0 \to/ >>->/<200> m $.
Then $ F(A+f) (a) \in F (A+n) $ and we have to show
that it's not in any $ F(A+n_0) $ for $ n_0 \to/ >>->/<200> n $.
Suppose it were. Take the pullback
$$
\bfig
\square/ >>->`>`->>` >>->/[m_0\ `n`n_0\ `m\rlap{ .};
``f`]

\place(250,250)[\framebox{\scriptsize Pb}]

\efig
$$
As usual $ m_0 \to/ >->/<200> m $ is monic, but it
is also proper because $ f $ is surjective. $ F $
being taut produces a pullback
$$
\bfig
\square/ >->`>`>` >->/<700,500>[F(A+m_0)`F(A+m)`F(A+n_0)`F(A+n);
``F(A+f)`]

\place(350,250)[\framebox{\scriptsize Pb}]

\efig
$$
so if $ F (A+f) (a) $ were in $ F (A+n_0) $, $ a $ would be
in $ F(A+m_0) $ contrary to the definition of $ a $.
\end{proof}

Furthermore, the construction $ F \leadsto \langle \Delta^n [F]
(A) \rangle_n $ is itself functorial. If $ {\bf Taut} $ is the category
of taut endofunctors of $ {\bf Set} $ with taut natural
transformations as morphisms, then we have the following.

\begin{proposition}
\label{Prop-FunTaut}

For any set $ A $, taking iterated differences produces
a functor
$$
\Delta^*_A \colon {\bf Taut} \to {\bf Set}^{\bf Surj} \rlap{ .}
$$

\end{proposition}

\begin{proof}

For $ F \colon {\bf Set} \to {\bf Set} $ a taut functor,
let $ \Delta^*_A (F) $ be given by
$$
\Delta^*_A (F) (n) = \Delta^n [F] (A)
$$
the functor of the previous proposition. Let
$ t \colon F \to G $ be a taut transformation. By applying
Proposition~\ref{Prop-TautTransf} iteratively we see
that $ t S^n \colon FS^n \to GS^n $ restricts to a taut
transformation
$$
\bfig
\square/ >->`>`>` >->/[\Delta^n{[}F{]}`FS^n`\Delta^n{[}G{]}`GS^n\rlap{ .};
`\Delta^n {[}t{]}`tS^n`]

\efig
$$
This gives the formula for $ \Delta^*_A (t) $, namely
$ \Delta^*_A (t) (A) = t(A+n) $. So naturality of
$ \Delta^*_A (F) (m) $ in $ n $ is simply
$$
\bfig
\square<800,500>[F(A+m)`F(A+n)`G(A+m)`G(A+n);
F(A+f)`t(A+m)`t(A+n)`G(A+f)]

\efig
$$
for any surjection $ f \colon m \to/->>/<200> n $,
which is just naturality of $ t $ itself.

Functoriality of $ \Delta^*_A  [t] $ in $ t $ follows
from the fact that it is a restriction of $ t S^n $.
\end{proof}

Below we will be only concerned with \( \Delta^*_A \)
for \( A = 0 \) and write \( \Delta^* \) for \( \Delta^*_0 \).

\begin{definition}
\label{Def-Soft}

Let us call a functor $ G \colon {\bf Surj} \to {\bf Set} $
a {\em soft species} (``soft'' because the structures
can be compressed). A soft species produces a
{\em soft analytic functor}, the {\em Newton sum}
of $ G $.
$$
\widetilde{G} (X) = \int^{n \in {\bf Surj}} X^n \times G (n) ,
$$
i.e.~, left Kan extension of $ G $ along the inclusion $ J $
of $ {\bf Surj} $ into $ {\bf Set} $
$$
\bfig
\Vtriangle/ >->`>`>/<400,500>[{\bf Surj}`{\bf Set}`{\bf Set};
J`G`{\scriptsize {\rm Lan}}_J G = \widetilde{G}]

\morphism(320,300)/=>/<150,0>[`;]

\place(800,0)[.]

\efig
$$

\end{definition}

This is similar to the definition of analytic functor which is
the Kan extension along the inclusion of $ {\bf Bij} $ into
$ {\bf Set} $. For a species $ F \colon {\bf Bij}  \to {\bf Set} $
we can take the Kan extension of $ F $ in steps
$$
\bfig
\Vtriangle/`>`>/<600,700>[{\bf Bij}`{\bf Set}`{\bf Set};
`F`{\scriptsize {\rm Lan}}_{JI} F]

\morphism(0,700)/ >->/<600,0>[{\bf Bij}`{\bf Surj};I]

\morphism(600,700)/ >->/<600,0>[{\bf Surj}`{\bf Set};J]

\morphism(600,700)|r|/>/<0,-700>[{\bf Surj}`{\bf Set};{\scriptsize {\rm Lan}}_I F]

\efig
$$
so that every analytic functor is soft analytic: every
species can be softened.

It may be of interest to note that the soft species
associated to $ F $, $ \mbox{Lan}_I F $ is given by
$$
({\rm Lan}_I F) (n) = \sum_{m \to/->>/<120> n} F(m) /_\sim
$$
where the sum is taken over all surjections
$ f \colon m \to/->>/<200> n $, and the equivalence
relation is the quotient of $ F (m) $ by the subgroup
of $ S_m $ consisting of all elements preserving
$ f $
$$
\bfig
\Vtriangle/>`->>`->>/<350,450>[m`m`n;\sigma`f`f]

\place(500,0)[.]

\efig
$$

We can analyze the definition of $ \widetilde{G} $
further to connect it to analytic functors and
underline the similarity with formulas (1) and (2)
for Newton series. From the definition, $ \widetilde{G} (X) $
consists of equivalence classes $ [a \in G(n), f \colon n \to X] $
of triples $ (n, a, f) $ where the equivalence relation
is generated by the relation $ (n, a, f) \sim (m, b, g) $
if there exists a surjection $ \sigma \colon n \to/->>/<200> m $
such that $ f = g \sigma $ and $ b = G(\sigma)(a) $
$$
\bfig
\place(0,0)[b]

\place(150,0)[\in]

\place(350,0)[G(m)]

\place(0,400)[a]

\place(150,400)[\in]

\place(350,400)[G(n)]

\morphism(0,400)/|->/<0,-300>[`;]

\morphism(350,350)|r|/>/<0,-280>[`;G(\sigma)]

\square(800,0)|arrb|/>`->>`=`>/<400,400>[n`X`m`X\rlap{ .};f`\sigma``g]

\efig
$$
So the equivalence relation is expressed in terms of
a zigzag path of surjections and elements of $ G $
satisfying conditions as above. We will see presently
that we can assume that the $ f \colon n \to X $ is
monic, $ \sigma $ a bijection and the zigzag paths have
length one, i.e.~no zigzag at all. Everything we need
can be expressed in these terms.

\begin{proposition}
\label{Prop-GTilde}

\begin{itemize}

	\item[(1)] Every equivalence class in $ \widetilde{G} (X) $ contains
	a representation in which $ f $ is monic.
	
	\item[(2)] Two elements $ (n, a, f) $ and $ (m, b, g) $ with $ f $
	and $ g $ monic are equivalent if and only if $ m = n $ and
	there is a bijection $ \sigma \colon n \to m $ with $ f = g \sigma $
	and $ b = G(\sigma)(a) $.
	
	\item[(3)] For a morphism $ \phi \colon X \to Y $,
	$ \widetilde{G} (\phi) [n, a, f] = [m, b, g] $ where $ g $ is the
	image of $ \phi f $ and $ b $ comes from $ a $ as  in
	$$
\bfig
\place(0,0)[b]

\place(150,0)[\in]

\place(350,0)[G(m)]

\place(0,400)[a]

\place(150,400)[\in]

\place(350,400)[G(n)]

\morphism(0,400)/|->/<0,-300>[`;]

\morphism(350,350)|r|/>/<0,-280>[`;G(e)]

\square(800,0)|arrb|/ >->`->>`>` >->/<400,400>[n`X`m`Y\rlap{ .};f`e`\phi`g]

\efig
$$

	\item[(4)] If $ \phi \colon X \to/ >->/<200> Y $ is monic, then so is
	$ \widetilde{G} (\phi) $ and an element $ [b, m, g] $ in $ \widetilde{G}(Y) $
	is in the image of $ \widetilde{G} (\phi) $ if and only if $ g $ factors
	through $ \phi $
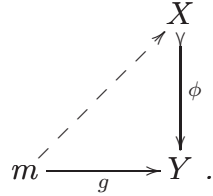
	$$
	\bfig
	\dtriangle/<--` >->`>/[X`m`Y\rlap{ .};`\phi`g]
	
	\efig
	$$

\end{itemize}

\end{proposition}

\begin{proof}

(1) For any equivalence class $ [n, a, f] $ take the image
factorization of $ f $, and since the quotient of a finite
cardinal is again one, we get
$$
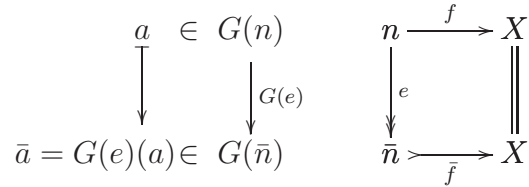
\bfig
\place(-150,0)[\ov{a} = G(e)(a)]

\place(150,0)[\in]

\place(350,0)[G(\ov{n})]

\place(0,400)[a]

\place(150,400)[\in]

\place(350,400)[G(n)]

\morphism(0,400)/|->/<0,-300>[`;]

\morphism(350,350)|r|/>/<0,-280>[`;G(e)]

\square(800,0)|arrb|/>`->>`=` >->/<400,400>[n`X`\ov{n}`X;f`e``\ov{f}]

\efig
$$
so that $ [n, a, f] = [\ov{n}, \ov{a}, \ov{f}] $.

(2) Suppose elements $ (n, a, f) $ and $ (m, b, g) $ of
the same equivalence class are related by a single
epimorphism $ \sigma \colon n \to/->>/<200> m $
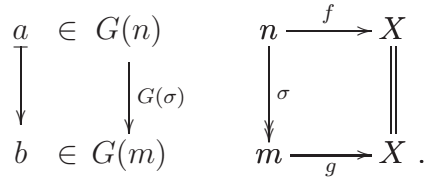
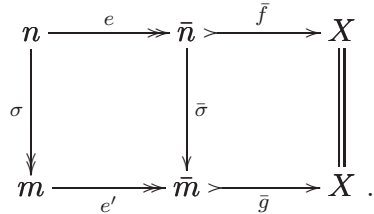
$$
\bfig
\place(0,0)[b]

\place(150,0)[\in]

\place(350,0)[G(m)]

\place(0,400)[a]

\place(150,400)[\in]

\place(350,400)[G(n)]

\morphism(0,400)/|->/<0,-300>[`;]

\morphism(350,350)|r|/>/<0,-280>[`;G(\sigma)]

\square(800,0)|arrb|/>`->>`=`>/<400,400>[n`X`m`X\rlap{ .};f`\sigma``g]

\efig
$$
Factor $ f $ and $ g $, giving
$$
\bfig
\square/`->>`=`/<1000,500>[n`X`m`X;`\sigma``]

\morphism(0,0)|b|/->>/<500,0>[m`\ov{m};e']

\morphism(500,0)|b|/ >->/<500,0>[\ov{m}`X\rlap{ .};\ov{g}]

\morphism(0,500)|a|/->>/<500,0>[n`\ov{n};e]

\morphism(500,500)|a|/ >->/<500,0>[\ov{n}`X;\ov{f}]

\morphism(500,500)|r|/>/<0,-500>[\ov{n}`\ov{m};\ov{\sigma}]

\efig
$$
We see that $ \ov{\sigma} $ is both one-one and onto,
so a bijection.  Let $ \ov{a} = G(e)(a) $ and $ \ov{b} =
G(e')(b) $. Then $ G(\ov{\sigma})(\ov{a}) = \ov{b} $. So
any zigzag path relating $ (n, a, f) $ to $ (m, b, g) $
$$
\bfig
\node a(0,0)[b]
\node b(150,0)[\in]
\node c(350,0)[G(m)]
\node d(0,300)[]
\node e(0,380)[\vdots]
\node f(350,300)[]
\node g(0,800)[a_2]
\node h(150,800)[\in]
\node i(350,800)[G(n_2)]
\node j(0,1200)[a_1]
\node k(150,1200)[\in]
\node l(350,1200)[G(n_1)]
\node m(0,1600)[a]
\node n(150,1600)[\in]
\node o(350,1600)[G(n)]
\node p(350,380)[\vdots]

\node a'(800,0)[m]
\node b'(1300,0)[X]
\node c'(800,800)[n_2]
\node d'(1300,800)[X]
\node d''(800,300)[]
\node f''(1300,300)[]
\node e'(800,1200)[n_1]
\node f'(1300,1200)[X]
\node g'(800,1600)[n]
\node h'(1300,1600)[X]
\node i'(800,380)[\vdots]
\node j'(1300,380)[\vdots]

\arrow/|->/[a`d;]
\arrow/>/[c`f;]
\arrow/|->/[g`e;]
\arrow/|->/[g`j;]
\arrow/>/[i`l;]
\arrow/|->/[m`j;]
\arrow/>/[o`l;]
\arrow/->>/[a'`d'';]
\arrow/=/[b'`f'';]
\arrow/->>/[c'`i';]
\arrow/=/[d'`j';]
\arrow|l|/->>/[c'`e';\sigma_2]
\arrow/=/[d'`f';]
\arrow/->>/[g'`e';\sigma_1]
\arrow/=/[h'`f';]

\arrow|b|/ >->/[a'`b';g]
\arrow|a|/>/[c'`d';f_2]
\arrow|a|/>/[e'`f';f_1]
\arrow|a|/ >->/[g'`h';f]
\arrow/>/[i`p;]

\efig
$$
can be replaced with one in which all the $ f_i $
are monic and all the $ \sigma_i $ bijections.
Replacing the backward $ \sigma $'s by their
inverses gives a path of length $ 1 $.

(3) For a morphism $ \phi \colon X \to Y $, 
$ \widetilde{G} (\phi) $ is given by composition,
$ \widetilde{G} (\phi) [n, a, f] = [n, a, \phi f] $ and
if we want to express this with a monic $ g $ we
factor $ \phi f $,
$$
\bfig
\node a(0,0)[b]
\node b(150,0)[\in]
\node c(350,0)[G(m)]
\node d(700,0)[m]
\node e(1200,0)[Y]
\node f(0,400)[a]
\node g(150,400)[\in]
\node h(350,400)[(n)]
\node i(700,400)[n]
\node j(1200,400)[X] 

\arrow/|->/[f`a;]
\arrow|r|/>/[h`c;G(e)]
\arrow|r|/->>/[i`d;e]
\arrow|r|/>/[j`e;\phi]
\arrow|b|/ >->/[d`e;g]
\arrow|a|/ >->/[i`j;f]

\efig
$$
and $ [n, a, \phi f] = [m, b, g] $.

(4) A monomorphism in $ {\bf Set} $, $ \phi \colon X \to/ >->/<200> Y $
is split, unless $ X = \emptyset \neq Y $, and will be automatically
preserved. In case $ X = \emptyset $, an element of
$ \widetilde{G} (\emptyset) $ is of the form
$ [a~\in~G(\emptyset), \emptyset \to/ >->/<200> \ \emptyset] $ and
$ \widetilde{G} (\phi) $ will take it to $ [a \in G(\emptyset), 
\emptyset \to/ >->/<200> Y] $.
This is also monic because the equivalence relation is
trivial.

For the second part of the statement, the ``if'' is
trivial. So suppose we have an element $ [m, b, g] $,
with $ g $ monic, in the image of $ \widetilde{G} (\phi) $
so that there is $ [n, a, f] $, with $ f $ monic, in
$ \widetilde{G} (\phi) $ with $ [m, b, g] = [n, a, \phi f] $.
Then by (2) there is a bijection $ \sigma $ such that
$$
\bfig
\node a(0,0)[b]
\node b(150,0)[\in]
\node c(350,0)[G(m)]
\node d(700,0)[m]
\node e(1300,0)[Y]
\node f(0,400)[a]
\node g(150,400)[\in]
\node h(350,400)[G(n)]
\node i(700,400)[n]
\node j(1000,400)[X]
\node k(1300,400)[Y] 

\arrow/|->/[a`f;]
\arrow/>/[c`h;]
\arrow|r|/>/[d`i;\sigma]
\arrow/=/[e`k;]
\arrow|b|/ >->/[d`e;g]
\arrow|a|/ >->/[i`j;f]
\arrow|a|/ >->/[j`k;\phi]

\efig
$$
and there is our factorization. 
\end{proof}

So we see that $ \widetilde{G} $ on objects is given
by
$$
\widetilde{G} (X) = \sum^\infty_{n = 0} G(n) \times {\rm Mono}
(n, X)/S_n
$$
which, if we take $ G = \Delta^* [F] $, looks a lot like
formula (2) for the Newton sum. It also looks like the
formula for an analytic functor. But $ \widetilde{G} $ is
not the coproduct of functors that this might suggest.
First of all $ {\rm Mono} (n, X) $ is not functorial in $ X $,
and secondly $ \widetilde{G} $ is a coend over $ {\bf Surj} $,
so surjections have to be accounted for. The point is that
the extra structure that $ G $ has compensates for
the shortfalls of $ {\rm Mono} (n, X) $.

Power series functors
$$
\sum^\infty_{n = 0} C_n X^n
$$
are analytic so soft analytic too but we can calculate
directly what the corresponding soft species is. We can
view such a polynomial functor as a left Kan extension
$$
\bfig
\Vtriangle/>`>`>/<350,450>[{\mathbb N}`{\bf Set}`{\bf Set};
K`C_{(\ )}`\sum C_n X^n]

\morphism(300,300)/=>/<150,0>[`;]

\efig
$$
where $ {\mathbb N} $ is the discrete category whose
objects are natural numbers, and $ K $ takes the number
$ n $ and sends it to the set of cardinality $ n $
$$
K (n) = \{1, 2, \dots, n\} .
$$
The corresponding soft species $ G $ is given by
$$
G (n) = \sum_{m \to/->>/ n} C_m \rlap{ ,}
$$
one summand for each surjection $ m \to/->>/ n $.

\begin{proposition}
\label{Prop-SoftTaut}

\begin{itemize}

	\item[(1)] For a soft species $ G $, the corresponding
	soft analytic functor $ \widetilde{G} $ is taut.
	
	\item[(2)] For a morphism of soft species $ t \colon G \to H $,
	i.e.~a natural transformation, the corresponding
	$ \tilde{t} \colon \widetilde{G} \to \widetilde{H} $ is taut.
	
	\item[(3)] $ \tilde{(\ )} $ defines a functor $ {\bf SoftSp} \to
	{\bf Taut} $, from the category of soft species $ {\bf Set}^{\bf Surj} $
	to the category of taut endofunctors of $ {\bf Set} $ and taut
	natural transformations.\end{itemize}

\end{proposition}

\begin{proof}

(1) Let 
$$
\bfig
\square/ >->`>`>` >->/[X_0`X`Y_0`Y;
i`\phi_0`\phi`j]

\place(250,250)[\framebox{\scriptsize Pb}]

\efig
$$
be an inverse image diagram. We want to show that
$$
\bfig
\square/ >->`>`>` >->/[\widetilde{G}(X_0)`\widetilde{G}(X)
`\widetilde{G}(Y_0)`\widetilde{G}(Y);
`\widetilde{G}(\phi_0)`\widetilde{G}(\phi)`]

\efig
$$
is a pullback. Take $ [a \in G(a), n \to/ >->/^f X] $ in
$ \widetilde{G}(X) $. 
\begin{align*}
\widetilde{G}(\phi) [a, n, f] & = [a \in G(n), n \to/ >->/ X \to^f Y]\\
      & = [G(e)(a) \in G(m), m \to/ >->/^g Y]
\end{align*}
where $ n \to/->>/^e m \to/ >->/^g Y $ is the image
factorization of $ \phi f $. Suppose there is
$ [b \in G(p), p \to/ >->/^h Y_0] $ in $ \widetilde{G}(Y_0) $
such that
$$
[b \in G(p), p \to/ >->/^h Y_0 \to/ >->/ Y] =
[G(e)(a) \in G(m), m \to/ >->/^g Y] .
$$
Then by \ref{Prop-GTilde} (2), there exists a bijection
$ \sigma \colon m \to p $ such that
$$
\bfig
\node a(0,0)[b]
\node b(150,0)[\in]
\node c(350,0)[G(p)]
\node d(700,0)[p]
\node e(1000,0)[Y_0]
\node f(1300,0)[Y]
\node g(0,400)[G(e)(a)\quad]
\node h(150,400)[\in]
\node i(350,400)[G(m)]
\node j(700,400)[m]
\node k(1300,400)[Y] 

\arrow/|->/[g`a;]
\arrow|r|/>/[i`c;G(\sigma)]
\arrow|l|/>/[j`d;\sigma]
\arrow/=/[k`f;]
\arrow|b|/ >->/[d`e;h]
\arrow|b|/ >->/[e`f;j]
\arrow|a|/ >->/[j`k;g]

\efig
$$
i.e.~$ b = G(\sigma) G(e)(a) $. Consider the diagram
$$
\bfig
\node a(300,0)[Y_0]
\node b(800,0)[Y\rlap{ .}]
\node c(0,250)[p]
\node d(300,500)[X_0]
\node e(800,500)[X]
\node f(0,500)[m]
\node g(0,800)[n]

\arrow|a|/{@{ >->}@/^1.2em/}/[g`e;f]
\arrow|r|/-->/[g`d;k]
\arrow|a|/ >->/[d`e;i]
\arrow|b|/ >->/[c`a;h]
\arrow|b|/ >->/[a`b;j]
\arrow|l|/->>/[g`f;e]
\arrow|l|/>/[f`c;\sigma]
\arrow|l|/>/[d`a;\phi_0]
\arrow|r|/>/[e`b;\phi]

\place(550,250)[\framebox{\scriptsize Pb}]

\efig
$$
$ j h \sigma e = g e = \phi f $ so there exists  $ k $
as above, and $ [n, a, f] $ is in $ \widetilde{G}(X_0) $.
So $ \widetilde{G} $ is taut.

(2) Let $ t \colon G \to H $ be a natural transformation.
Then $ \tilde{t} \colon \widetilde{G} \to \widetilde{H} $
is defined as follows
$$
\tilde{t} (X) \colon \widetilde{G} (X) \to \widetilde{H} (X)
$$
$$
[a \in G(n), n \to/ >->/ X] \longmapsto [t(n)(a) \in
H(n), n \to/ >->/ X] .
$$

We have several things to prove:

(i) Well-defined: If $ [a \in G(n), n \to/ >->/ X] =
[a' \in G(n'), n' \to/ >->/ X] $ then there is a bijection
$ \sigma \colon n \to n' $ such that
$$
\bfig
\node a(0,0)[a']
\node b(150,0)[\in]
\node c(350,0)[G(n')]
\node d(700,0)[n']
\node e(1200,0)[X]
\node f(0,400)[a]
\node g(150,400)[\in]
\node h(350,400)[G(n)]
\node i(700,400)[n]
\node j(1200,400)[X] 

\arrow/|->/[f`a;]
\arrow|r|/>/[h`c;G(\sigma)]
\arrow|r|/>/[i`d;\sigma]
\arrow|r|/=/[j`e;]
\arrow|b|/ >->/[d`e;]
\arrow|a|/ >->/[i`j;]

\place(560,400)[,]
\place(560,0)[,]

\efig
$$
i.e.~$ a' = G(\sigma)a $. So $ H(\sigma) t(n)(a) =
t(n')G(\sigma)a = t(n')a' $.

(ii) Natural:
$$
\bfig
\square<600,500>[\widetilde{G}(X)`\widetilde{H}(X)`\widetilde{G}(Y)`\widetilde{H}(Y);
\tilde{t}(x)`\widetilde{G}(\phi)`\widetilde{H}(\phi)`\tilde{t}(Y)]

\place(300,250)[?]

\morphism(1000,500)|r|/>/<0,-500>[X`Y;\phi]

\efig
$$
Choose an element $ [a \in G(n), \alpha \colon n \to/ >->/ X] $ of
$ \widetilde{G}(X) $ and chase it around the diagram
$$
\bfig
\square/|->`|->`|->`|->/<800,500>[{[}a, \alpha{]}`{[}t(n) (a), \alpha{]}
`{[}a, \phi \alpha{]}`{[}t(n)(a), \phi \alpha{]}\rlap{ .};```]

\efig
$$
$ \tilde{t} $ is indeed natural.

(iii) Taut: Let
$$
\bfig
\square/ >->`>`>` >->/[X_0`X`Y_0`Y;`\phi_0`\phi`]

\place(250,250)[\framebox{\scriptsize Pb}]

\efig
$$
be an inverse image diagram and consider
\begin{equation}\tag{*}
\bfig
\square/ >->`>`>` >->/[\widetilde{G} X_0`\widetilde{G} X
`\widetilde{G} Y_0`\widetilde{G} Y\rlap{ .};
`\widetilde{G} (\phi_0)`\widetilde{G} \phi`]

\efig
\end{equation}
If $ \widetilde{G} (\phi) [a \in G(n), n \to/ >->/ X] \in
\widetilde{G} Y_0 $, then it is $ [G(e)(a) \in G(m),
m \to/ >->/^g Y] $
where $ n \to^e m \to/ >->/^g Y $ is the image
factorization of $ \phi f $. By \ref{Prop-GTilde} (4),
$ \widetilde{G} (\phi) [a, n] $ is in $ \widetilde{G} Y_0 $
iff $ g \colon m \to/ >->/ Y $ factors through $ Y_0 $
$$
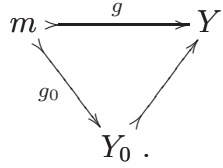
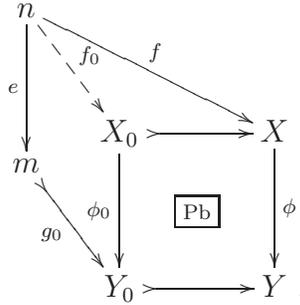
\bfig
\Vtriangle/ >->` >->`<-< /<300,400>[m`Y`Y_0\rlap{ .};g`g_0`]

\efig
$$
In the diagram
$$
\bfig
\node a(300,0)[Y_0]
\node b(800,0)[Y\rlap{ .}]
\node d(300,500)[X_0]
\node e(800,500)[X]
\node f(0,400)[m]
\node g(0,900)[n]

\arrow|l|/ >->/[f`a;g_0]
\arrow|r|/-->/[g`d;f_0]
\arrow|a|/>/[g`f;e]
\arrow/ >->/[d`e;]
\arrow/ >->/[a`b;]
\arrow|l|/>/[d`a;\phi_0]
\arrow|r|/>/[e`b;\phi]
\arrow|a|/>/[g`e;f]

\place(550,250)[\framebox{\scriptsize Pb}]

\efig
$$
the outside commutes, so there exists
$ f_0 \colon n \to X_0 $, as above, i.e.~$ f $
factors through $ X_0 $, and so
$ [a \in G(n), m \to/ >->/ X] \in \widetilde{G} (X_0) $.
Therefore $ (*) $ is a pullback, and $ \tilde{t} $ is
taut.

(3) $ \widetilde{(\  )} $ is automatically functorial
$ {\bf Set}^{\bf Surj} \to {\bf Set}^{\bf Set} $ because
it is Kan extension. The only question is whether
it takes its values in $ {\bf Taut} $, and that's what
was proved in (1) and (2).
\end{proof}

We are now ready for our main theorem which
might be dubbed ``The Fundamental Theorem of
Functorial Difference Calculus''. Part~I says
``Summing a soft species and then taking differences
gives the original soft species'' and Part~II says
``Taking differences of a taut functor and then summing
produces a best approximation by a soft analytic
functor''.

\begin{theorem}[Newton summation]

\noindent I. For a soft species $ G \colon {\bf Surj} \to {\bf Set} $
we have a natural isomorphism
$$
G \to^\cong \Delta^* [\widetilde{G}]  .
$$

\noindent II.  $ \widetilde{(\  )}  \colon {\bf SoftSp} \to {\bf Taut} $ is
left adjoint to $ \Delta^* \colon {\bf Taut} \to {\bf SoftSp} $. 

\end{theorem}

\begin{proof}

I. Using Proposition~\ref{Prop-GTilde} (1) and
Corollary~\ref{Cor-DeltaK}, we see that an element
of $ \Delta^n [\widetilde{G}] (0) $ is an equivalence
class
$$
[a \in G(k), k \to/ >->/ n]_k
$$
which is not equal to any
$$
[b \in G(k_0), k_0 \to/ >->/ n_0 \to/ >>->/ n]_{k_0}
$$
for a proper mono $ n_0 \to/ >>->/ n $. If $ k \to/ >->/ n $
were a proper mono, we would have
$$
\bfig
\node a(0,0)[a]
\node b(150,0)[\in]
\node c(350,0)[G(k)]
\node d(700,0)[k]
\node e(820,0)[=]
\node f(1200,0)[k \to/ >>->/ n\rlap{ ,}\quad\quad]
\node g(0,400)[a]
\node h(150,400)[\in]
\node i(350,400)[G(k)]
\node j(700,400)[k]
\node k(1200,400)[n] 
\arrow/|->/[g`a;]
\arrow|r|/>/[i`c;]
\arrow|r|/=/[d`j;]
\arrow/=/[f`k;]
\arrow/ >->/[j`k;]

\efig
$$
so $ [a \in G(n), k \to/ >->/ n] $ would not be in
$ \Delta^n [\widetilde{G}] (0) $. It follows that the
elements of $ \Delta^n [\widetilde{G}] (0) $ are of
the form
$$
[a \in G(n), \sigma \colon n \to n]
$$
for $ \sigma $ a bijection. Each such class has a
canonical representative with $ \sigma = \id $,
and for these the equivalence relation is
equality, i.e.~$ \Delta^n [\widetilde{G}] (0) $ is
in bijection with $ G(n) $ itself.

The natural transformation which gives this
bijection is
$$
\begin{array}{rcl}
\eta (n) \colon G(n)  & \to/>/<250> & \Delta^n [\widetilde{G}](0)\\
a   & \to/|->/<250>  & [a \in G(n), \id \colon n \to n] \rlap{ .}
\end{array}
$$

II. Let $ G \colon {\bf Surj} \to {\bf Set} $ be a
soft species and $ F \colon {\bf Set} \to {\bf Set} $
be a taut functor. As $ \widetilde{G} $ is the left
Kan extension of $ G $ along the inclusion
$ J \colon {\bf Surj} \to {\bf Set} $, we already have
 a natural bijection of {\em natural transformations}
 $$
 \widetilde{G} \to^t F \over G \to_u FJ \rlap{ .}
 $$
 Note that $ \Delta^* [F] $ is a subfunctor of $ FJ $.
 We will show that in the above bijection $ t $ is
 taut if and only if $ u $ factors through
 $ \Delta^* (F) \to/ >->/ FJ $.
 
 First assume $ t $ is taut. By Kan extension theory,
 the $ u $ corresponding to $ t $ is given by
 $$
 u (n) \colon G(n) \to F(n)
 $$
 $$
 u (n) (a) = t (n) [a \in G(n), \id \colon n \to/ >->/ n] \rlap{ .}
 $$

We want to show that $ u (n) (a) $ is in fact an
element of $ \Delta^n [F] (0) $, i.e.~$ u (n) (a) $ is not
in any $ F (n_0) \to/ >->/ F (n) $ for a proper mono
$ f \colon n_0 \to/ >>->/ n $. Suppose it were, so that
we would have $ x \in F (n_0) $ such that
$ F(f) (x) = u (n) (a) $. As $ t $ is taut, we have the
pullback
$$
\bfig
\square/ >->`>`>` >->/<600,500>[\widetilde{G}(n_0)`\widetilde{G}(n)
`F(n_0)`F(n);
\widetilde{G}(f)`t(n_0)`t(n)`F(f)]

\place(300,250)[\framebox{\scriptsize Pb}]

\efig
$$
and so there is an element $ [b \in G(m),
m \to/ >->/ n_0] $ in $ \widetilde{G}(n_0) $
such that $ t(n_0) $ of which is $ x $
and $ \widetilde{G}(f) [b \in G(m), m \to/ >->/ n_0] =
[a \in G(n), \id \colon n \to/ >->/ n] $. This last
equation means that there is a bijection
$ \sigma \colon n \to m $ with
$$
\bfig
\node a(0,0)[b]
\node b(150,0)[\in]
\node c(350,0)[G(m)]
\node d(0,450)[a]
\node e(150,450)[\in]
\node f(350,450)[G(n)]

\arrow/|->/[d`a;]
\arrow|r|/>/[f`c;G(\sigma)]

\square(700,0)|arlb|/>`>`=`/<800,450>[n`n`m`n;\id`\sigma``]
\morphism(700,0)/ >->/<400,0>[m`n_0;]
\morphism(1100,0)/ >->/<400,0>[n_0`n;]

\efig
$$
which implies that $ n_0 \to/ >->/ n $ is an iso,
i.e.~is not proper. This contradicts our choice of $ n_0 $,
so $ u (n) (a) $ is indeed an element of
 $ \Delta^n [F](0) $, and $ u $ factors through
 $ \Delta^* [F] \to/ >->/ F J $.
 
 Now, let $ u \colon G \to F J $ factor through
 $ \Delta^* [F] $. The corresponding $ t \colon
 \widetilde{G} \to F $ is given by
 $$
 t (X) \colon \widetilde{G} (X) \to F (X)
 $$
 $$
 [a \in G(n), n \to/ >->/^f X] \longmapsto 
 F (f) u (n) (a)
 $$
 and we want to show that it's taut. To this end,
 let $ \phi \colon Y \to/ >->/ X $ be a monomorphism.
 We will show that
 $$
 \bfig
 \square/ >->`>`>` >->/<550,500>[\widetilde{G} (Y)`\widetilde{G}(X)`FY`FX;
 G(\phi)`t(Y)`t(X)`F\phi]
 
 \efig
 $$
 is a pullback. Take $ [b \in G(n), g \colon n \to/ >->/ X] $
 in $ \widetilde{G} (X) $ such that there is $ y \in FY $ with
\begin{align*}
F(\phi)(y) & = t (X) [n, b, g]\nonumber\\
& = F(g) u (n)(a)\rlap{ .} \tag{*}
\end{align*}
We want to show that $ g $ factors through $ \phi $. Take
the pullback
$$
\bfig
\square/ >->` >->` >->` >->/[m`n`Y`X;\psi`f`g`\phi]

\place(250,250)[\framebox{\scriptsize Pb}]

\efig
$$
and, as $ F $ is taut, we get another pullback
$$
\bfig
\square/ >->` >->` >->` >->/[F(m)`F(n)`F(Y)`F(X)\rlap{ .};
F(\psi)`F(f)`F(g)`F(\phi)]

\place(250,250)[\framebox{\scriptsize Pb}]

\efig
$$
Because of $ (*) $, there is $ x \in F(m) $ such that
$ F(f) (x) = y $ and $ F(\psi) (x) = u (n) (a) $. Since
$ u (n) (a) $ is in $ \Delta^n [F] (0) $, $ \psi $ cannot
be proper so that $ g $ will factor through $ \phi $ and
$ [b \in G(n), s \colon n \to/ >->/ X] $ is in
$ \widetilde{G} (Y) $.
 \end{proof}

 \begin{corollary}
 
 If $ F \colon {\bf Set} \to {\bf Set} $ is soft analytic, e.g.~power
 series or analytic, then its Newton series converges to it
 $$
 F(X) \cong \int^{n \, \in\, {\rm Surj}} X^n \times \Delta^n [F](0)\rlap{ .}
 $$
 \end{corollary}

 Note that what we are calling power series are
 functors of the form
 $$
 F(X) = \sum^\infty_{n\,=\, 0} C_n X^n
 $$
 and only involve finite powers of $ X $. For
 polynomial functors as in Definition~\ref{Def-Poly}
 involving arbitrary powers of $ X $, the corollary
 doesn't hold. We merely get a comparison. For
 example, for an infinite set $ A $, the Newton series
 of $ F (X) = X^A $ converges to the functor
 $$
 G(X) =\{ f \colon A \to X\  | \mbox{\  the image of $ f $ is finite}\}.
 $$


\begin{thebibliography}{10}

\bibitem{AlvLem20}
Mario Alvarez-Picallo and Jean-Simon Pacaud-Lemay.
\newblock Cartesian difference categories.
\newblock {\em Logical Methods in Computer Science}, 17(3):23:1--23:48, 2021.

\bibitem{BarPar}
Michael Barr and Robert Par\'e.
\newblock Molecular toposes.
\newblock {\em J. Pure and Appl. Algebra}, 17:127--152, 1980.

\bibitem{Bla76}
Andreas Blass.
\newblock Exact functors and measurable cardinals.
\newblock {\em Pacific J. Math.}, 63(2):335--346, 1976.

\bibitem{Bla77A}
Andreas Blass.
\newblock Corrections to: ``{E}xact functors and measurable cardinals''
  ({P}acific {J}. {M}ath. 63 (1976), no. 2, 335-346).
\newblock {\em Pacific J. Math.}, 73(2):549, 1977.

\bibitem{Bla77B}
Andreas Blass.
\newblock Two closed categories of filters.
\newblock {\em Fundamenta Mathematicae}, 94(2):129--143, 1977.

\bibitem{CocCru14}
R.~Cockett and G.~Cruttwell.
\newblock Differential structure, tangent structure, and {SDG}.
\newblock {\em Applied Categorical Structures}, 22(2):331--417, 2014.

\bibitem{GamJoy17}
Nicola Gambino and Andr\'e Joyal.
\newblock {\em On Operads, Bimodules and Analytic Functors}, volume 249 of {\em
  Memoirs of the American Mathematical Society}.
\newblock American Mathematical Society, September 2017.

\bibitem{GamKoc13}
Nicola Gambino and Joachim Kock.
\newblock Polynomial functors and polynomial monads.
\newblock {\em Mathematical Proceedings of the Cambridge Philosophical
  Society}, 154(1):153--192, 2013.

\bibitem{Joy81}
Andr\'e Joyal.
\newblock Une th\'eorie combinatoire des s\'eries formelles.
\newblock {\em Advances in Mathematics}, 42(1):1--82, October 1981.

\bibitem{Koc09}
Joachim Kock.
\newblock Notes on polynomial functors.
\newblock \url{http://mat.uab.cat/~kock/cat/polynomial.html}, August 16 2016.

\bibitem{KocJoyBatMas10}
Joachim Kock, Andr\'e Joyal, Michael Batanin, and Jean-Fran\c{c}ois Mascari.
\newblock Polynomial functors and opetopes.
\newblock {\em Adv. Math.}, 224:2690--2737, 2010.

\bibitem{Man02}
Ernest~G. Manes.
\newblock Taut monads and $ {T}0 $-spaces.
\newblock {\em Theoretical Computer Science}, 275:79--109, 2002.

\bibitem{MyeSpi20}
David~Jaz Myers and David~I. Spivak.
\newblock Dirichlet functors are contravariant polynomial functors.
\newblock arXiv:2004.04183, April 2020.

\bibitem{NiuSpi23}
Nelson Niu and David~I. Spivak.
\newblock Polynomial functors: A mathematical theory of interaction.
\newblock \url{https://github.com/ToposInstitute/poly}, October 11 2023.

\bibitem{Plo67}
Jerzy P{\l}onka.
\newblock On a method of construction of abstract algebras.
\newblock {\em Fundamenta Mathematicae}, 61(2):183--189, 1967.

\bibitem{Rei71}
Jan Reiterman.
\newblock An example concerning set-functors.
\newblock {\em Commentationes Mathematicae Universitatis Carolinae},
  12(2):227--233, 1971.

\bibitem{SzaZaw15}
Stanislaw Szawiel and Marek Zawadowski.
\newblock Monads of regular theories.
\newblock {\em Appl Categor Struct}, 23:215--262, 2015.

\bibitem{Trn71A}
V\v{e}ra Trnkov\'a.
\newblock On descriptive classification of set-functors. {I}.
\newblock {\em Commentationes Mathematicae Universitatis Carolinae},
  12(1):143--174, 1971.

\bibitem{Trn71B}
V\v{e}ra Trnkov\'a.
\newblock On descriptive classification of set-functors. {II}.
\newblock {\em Commentationes Mathematicae Universitatis Carolinae},
  12(2):345--357, 1971.

\end{thebibliography}
\end{document}